%% file: arxiv.tex
\documentclass[11pt,a4paper,final]{article}

\input{arxiv_preamble}

\begin{document}
\title{Monotone Multispecies Flows}

\ifsiamart
\else
    \author[1]{Lauren Conger}
    \author[1]{Franca Hoffmann}
    \author[1]{Eric Mazumdar}
    \author[2]{Lillian J. Ratliff}
    \affil[ ]{\footnotesize \email{lconger@caltech.edu},
        \email{franca.hoffmann@caltech.edu},
    \email{mazumdar@caltech.edu},
    \email{ratliffl@uw.edu}}
    \affil[1]{\footnotesize Department of Computing and Mathematical Sciences, Caltech, USA}
    \affil[2]{\footnotesize Department of Electrical and Computer Engineering, University of Washington, USA}
    \date{}
\fi

\maketitle

\begin{abstract}
    We present a novel notion of $\lambda$-monotonicity for an $n$-species system of  partial differential equations governed by mass-preserving flow dynamics, extending monotonicity in Banach spaces to the  Wasserstein-2 metric space. We show that monotonicity implies the existence of and convergence to a unique steady state, convergence of the velocity fields and second moments, and contraction in the Wasserstein-2 metric, at rates dependent on $\lambda$.
    In the special setting of Wasserstein-2 gradient descent of different energies for each species, we prove convergence to the unique Nash equilibrium of the associated energies and delineate the relationship between monotonicity and displacement convexity. This extends known zero-sum results in infinite-dimensional game theory to the general-sum setting. We provide a number of examples of monotone coupled gradient flow systems, including cross-diffusion, gradient flows with potentials, nonlocal interaction, linear and nonlinear diffusion, and min-max systems, and draw connections to a class of mean-field games. Numerically, we demonstrate convergence of a four-player economic model for service providers and strategic users competing in a market, and a degenerately monotone game.
\end{abstract}



\section{Introduction}\label{sec:intro}

In the analysis of nonlinear dynamical systems, understanding the long-time behavior is important for modeling applications.  When the dynamical system evolves via an operator on Banach spaces, \emph{monotonicity} has proven to be a useful criteria under which convergence can be shown. In this work, we investigate mass-preserving flows, which are not in a normed space, and develop a notion of monotonicity in Wasserstein space which gives similar results to those for Banach spaces. 
Mass-preserving flow equations are commonly used to describe physical systems, such as fluid flow, biological systems~\cite[Ch 8]{perthame_parabolic_2015}, traffic~\cite{marcellini_ode-pde_2016}, and increasingly appear in emerging areas like generative modeling~\cite{kwon_score-based_2022}, opinion dynamics~\cite{during_boltzmann_2009}, and strategic learning environments~\cite{conger_strategic_2023,conger_coupled_2024}. Importantly, the monotonicity condition we propose has a structure that allows for systems where there is more than one entity evolving, such as in coupled flow equations. Understanding the long-time behavior of the coupled dynamics is a key question in these applications; in particular, monotonicity allows us to determine the existence of and convergence to a unique steady state.

Our notion of monotonicity is a natural extension of the well-studied condition for Euclidean dynamical systems in finite dimensions and Banach spaces in infinite dimensions, providing a condition for contraction of nonlinear dynamics. We show how this idea can be applied to velocity fields driving the evolution of probability measures. In the monotone flow setting, we consider an $n$--species system where each species $\rho_i\in\P(\R^{d_i})$, with $\P(\R^{d_i})$ the space of probability measures on $\R^{d_i}$, evolves according to a velocity field $v_i[\rho]:\R^{d_i}\to \R^{d_i}$, which can depend on all species, denoted $\rho=[\rho_1,\dots,\rho_n]$. The mass of each $\rho_i$ is conserved due to the divergence structure. We prove that monotonicity is a sufficient condition for dynamics of the form
\begin{align}\label{eq:monotone_flow_dynamics}
    \partial_t \rho_i +\div{\rho_i v_i[\rho]}=0 \quad \forall\, \iin\,,
\end{align}
to contract in the joint Wasserstein-2 metric.

We additionally study a particular choice of velocity fields which results in a system of coupled gradient flows. In the Euclidean setting, all norms are equivalent, whereas in infinite dimensions, the choice of metric determines the topology and the gradient flow trajectory. We focus on the Wasserstein-2 metric, which results in a system of partial differential equations (PDEs) where the evolution of each species is written as a gradient flow with respect to its own energy. Let $\P_2(\R^{d})$ be the space of probability measures with bounded second moments on $\R^d$, denote the product space $\Pbar=\P_2(\R^{d_1}) \times \cdots \times \P_2(\R^{d_n})$, and denote by $\Wass_2$ the Wasserstein-2 metric (see \cref{sec:appendix} for details). Consider energy functionals $F_i:\A\subseteq \Pbar \to \R$, for a suitable choice of domain $\A\subseteq \Pbar$. Selecting $v_i[\rho]=-\nabla_{x_i} \delta_{\rho_i} F_i[\rho]$ results in a system in which each species evolves in the direction of steepest descent of $F_i$ in the Wasserstein-2 metric, given by 
\begin{align}\label{eq:coupled_gradient_flow_dynamics}
     \partial_t \rho_i = -\nabla_{\Wass_2,\rho_i} F_i[\rho] = \divs{x_i}{\rho_i \nabla_{x_i} \delta_{\rho_i} F_i[\rho]}\quad \forall\, \iin\,.
\end{align}
This system also has a game theoretic interpretation; in a game where each player $i$ has a cost functional $F_i$ which it seeks to minimize, the player $i$ evolves a distribution $\rho_i\in \P_2(\R^{d_i})$ to minimize a cost functional $F_i[\rho_i,\rhoni]=F_i[\rho]:\A \rightarrow \R$, where $\rhoni$ denotes the non-$i^{th}$ player measures.
We show that monotonicity implies displacement convexity of each $F_i$; however, displacement convexity of each $F_i$ is not enough to ensure monotonicity because monotonicity also depends on the coupling structure among all the species. Additionally, under mild conditions, the steady state of \eqref{eq:coupled_gradient_flow_dynamics} corresponds to the unique Nash equilibrium of the game characterized by $(F_i)$.  

A primary contribution of this work is the definition of monotonicity for velocity fields and energy functionals in this setting.
\sloppy Denote by $\Gamma_i(\rhozero_i,\rhoone_i)$ the set of couplings whose marginals are $\rhozero_i,\rhoone_i\in\P_2(\R^{d_i})$, and let $\Gamma(\rhozero,\rhoone)\coloneqq \Gamma_1(\rhozero_1,\rhoone_1)\times \dots \times \Gamma_n(\rhozero_n,\rhoone_n)$. Let $\Gamma^*$ be the set of optimal transport plans, where $\Gamma_i^*=\{\gamma_i\, : \, \gamma_i\in\argmin_{\hat \gamma_i\in\Gamma_i(\rhozero_i,\rhoone_i)} \int \norm{x_i-y_i}^2\d\hat \gamma_i(x_i,y_i)$. Let $x=[x_1,\dots,x_n]\in\R^d$, with $d=\sum_{i=1}^n d_i$. Let $\Wbar$ be the joint metric over all species, given by $\Wbar(\rho,\mu)^2=\sum_{i=1}^n \Wass_2(\rho_i,\mu_i)^2$. 
The velocity field $v[\rho]:\R^d\to\R^d$ is given by
\begin{align*}
    v[\rho] := \begin{bmatrix}
        v_1[\rho] \\ \vdots \\ v_n[\rho] 
    \end{bmatrix}\,.
\end{align*}
\begin{definition}[Monotonicity]\label{def:monotonicity}
    Let $\A\subseteq\Pbar$ be a geodesically convex set. If for any $\rhozero,\rhoone\in\A$, 
\begin{align}\label{eq:dissipation-inequality}
     -\int \<x-y,v[\rhozero](x)-v[\rhoone](y)> \d \gamma(x,y) \ge \lambda \int \norm{x-y}^2 \d \gamma(x,y)=  \lambda\Wbar(\rhozero,\rhoone)^2  \,,
    \end{align}
    for all $\gamma\in\Gamma^*(\rhozero,\rhoone)$, then the velocity fields $\{v_i:\A \to \R^{d_i}\}_{i=1}^n$ are $\lambda$-monotone in $\A$. Similarly, we say the energy functionals $\{F_i:\A\to \R\}_{i=1}^n$ are $\lambda$-monotone in $\A$ if $v_i[\rho]=-\nabla_{x_i}\delta_{\rho_i} F_i[\rho]$ are $\lambda$-monotone.
\end{definition}
This definition of monotonicity is a multispecies version of the dissipativity condition given in \cite{cavagnari_dissipative_2023}, for single-valued $v$ instead of set-valued $v$.
Some texts use the word \textit{accretive} instead of monotone to describe operators which satisfy positive-definiteness over a Hilbert space \cite{brezis_hilleyosida_2011}; however, the space of measures endowed with the Wasserstein-2 metric is not a Hilbert space, and so one can think of our notion of monotonicity as an extension of the notion of accretivity.
It was recently shown in \cite[Theorem 7.6]{cavagnari_lagrangian_2023} that if a velocity field is dissipative, the dissipation inequality holds for all admissible plans, not just optimal couplings. In particular, when the dissipation inequality \eqref{eq:dissipation-inequality} holds for all pairs of measures on $\P_2$, it holds for all optimal plans between a pair of measures if and only if it holds for all admissible plans between those two measures. We prove a similar type of result for the multispecies setting.
\begin{theorem}[Admissible Couplings]\label{thm:admissible_couplings} Let $v:\Pbar \times \R^d \to \R^d$ be demicontinuous according to \cref{def:demicontinuity_v}.
    If for any $\rhozero,\rhoone \in \Pbar$ and $\gamma^* \in \Gamma^*(\rhozero,\rhoone)$ it holds 
    \begin{align*}
        -\int \<x-y,v[\rhozero](x)-v[\rhoone](y)> \d \gamma^*(x,y) \ge \lambda \int \norm{x-y}^2 \d \gamma^*(x,y)\,,
    \end{align*}
    then for any $\gamma\in \Gamma(\rhozero,\rhoone)$,
    \begin{align*}
        -\int \<x-y,v[\rhozero](x)-v[\rhoone](y)> \d \gamma(x,y)\ge  \lambda \int \norm{x-y}^2 \d \gamma(x,y) \,.
    \end{align*}
\end{theorem}
For a proof of this theorem, and a counterexample of why this result does not hold when $v$ is $\lambda$-monotone on an arbitrary geodesically convex subset $\A\subseteq\Pbar$, see \cref{sec:optimal_admissible_plans}.  
When the velocity field takes the form of a coupled gradient flow, we use the following notation for the row-wise gradient of the first variation 
    \begin{align*}
        \ndr F[\rho](x) \coloneqq \begin{pmatrix}
             \nabla_{x_1} \delta_{\rho_1} F_1[\rho](x_1) \\
        \vdots\\
        \nabla_{x_n} \delta_{\rho_n} F_n[\rho](x_n)
    \end{pmatrix}\,.
    \end{align*}
Then $\lambda$-monotonicity for $F=(F_i)$ writes as
\begin{align}\label{eq:monotonicity_condition_F}
          \int \<x-y,\ndr F[\rhozero](x)-\ndr F[\rhoone](y)> \d \gamma(x,y) \ge \lambda \Wbar(\rhozero,\rhoone)^2 \,.
\end{align}

\begin{theorem}[Convergence, informal]
    Let $v$ be $\lambda$-monotone in $\A$ and $\rho(t)\in\A$ evolve according to \eqref{eq:monotone_flow_dynamics} for all $t\ge 0$. Then
    \begin{enumerate}
        \item[(a)] there exists a unique steady state $\rho^\infty$ of \eqref{eq:monotone_flow_dynamics},
        and $\rho(t)$ converges exponentially fast to $\rho^\infty$ in $\Wbar$ with rate $\lambda$, and
        \item[(b)] if $v_i[\rho] =-\nabla_{x_i}\delta_{\rho_i} F_i[\rho]$ for all $\iin$, and the energies $F_i$ are lower semicontinuous, then $\rho^\infty$ is the unique Nash equilibrium in $\A$ for the game specified by $F$.
    \end{enumerate}
\end{theorem}
For the full statement of this result, see \cref{thm:cv_W2}.
Key supporting results include uniform second moment bounds and convergence of the $L^2$ norm of the velocity fields, which is related to the energy dissipation in the single-species gradient flow setting. For a summary of assumptions and results, see \cref{fig:assumptions}. The monotone structure mirrors existing finite-dimensional results (see \cite{gao_continuous-time_2022} and references therein), which are outlined in what follows. 

\begin{definition}[Finite-Dimensional Monotonicity]\label{def:finite_dim_monotone_dynamics}
    Let $u=[u_1,\dots,u_n]:\R^d \to \R^d$ where $u_i\in C^1(\R^d;\R^{d_i})$, and define $J(x)\coloneqq \nabla u(x) \in \R^{d\times d}$.
    Then $u$ is  \emph{$\lambda$-monotone} if it satisfies
    \begin{align*}
        &\text{\emph{(first-order)}}\qquad \<u(x)-u(x'),x-x'> \ge \lambda \norm{x-x'}^2\quad \forall\ x,x'\in\R^d\,, \text{ or}\\
        &\text{\emph{(second-order)}}\qquad 
       J(x)+J(x)^\top \succeq 2\lambda \Id_d \quad \forall\ x\in\R^d\,.
    \end{align*}
    for some $\lambda\in\R$; $u$ is \emph{strongly monotone} if $\lambda>0$.
\end{definition}
In the finite-dimensional monotone setting, for dynamics $\dot x(t)=-u(t)$, commonly-used Lyapunov functions include
\begin{align}\label{eq:lyapunov_functions}
    V_1(x) = \frac{1}{2} \norm{x-x^*}^2 \,, \qquad V_2(x) = \frac{1}{2}\sum_{i=1}^n \norm{u_i(x)}^2 
 \,,
\end{align}
where $x^*$ satisfies $u(x^*)=0$. When $u$ is strongly monotone, it can be shown that $\dot V_1(x) \le -2\lambda V_1(x)$ using the first-order monotonicity condition, and that $\dot V_2(x) \le -2\lambda V_2(x)$ using the second-order monotonicity condition. This results in exponential convergence of $V_1$ and $V_2$ to zero, which is sufficient to conclude exponential convergence of $x$ to $x^*$. We will analyze the infinite-dimensional equivalents:
\begin{align*}
     \frac{1}{2} \norm{x-x^*}^2 \quad \Leftrightarrow \quad \frac{1}{2}\Wbar(\rho,\rho^\infty)^2\,,  \qquad\frac{1}{2}\sum_{i=1}^n \norm{u_i(x)}^2 \quad \Leftrightarrow \quad \frac{1}{2} \sum_{i=1}^n \int_{\R^{d_i}} \norm{v_i[\rho](x_i)}^2 \d \rho_i(x_i) \,,
\end{align*}
where $V_1(x)$ is analogous to the squared joint metric distance $\Wbar$ between a measure $\rho$ and the steady state $\rho^\infty$, and $V_2$ is analogous to the weighted $L^2$ norm of the velocity fields. 
For a finite dimensional monotone system where the time derivative for each species is given by $-u_i(x)$, 
we can define a lifted infinite-dimensional dynamical system with a velocity field  $ v[\rho](x)=-\left[\int u_1(x)\d\rho_{-1}(x_{-1}),\dots,\int u_n(x) \d \rho_{-n}(x_{-n}) \right]$.
 Importantly, for monotone velocity fields in finite dimensions, the corresponding formulation in infinite dimensions  is also monotone.
 \begin{proposition}[Finite- to Infinite-Dimensional Monotonicity]\label{lem:monotonicity_finite_dimensions}
    Consider the finite dimensional velocity fields $u=[u_1,\dots,u_n]:\R^d \to \R^d$ and define $ v[\rho]=-\left[\int u_1(x)\d\rho_{-1}(x_{-1}),\dots,\int u_n(x) \d \rho_{-n}(x_{-n}) \right]$.  Then $u$ is $\lambda$-monotone according to \cref{def:finite_dim_monotone_dynamics} if and only if the velocity field $v$
        is $\lambda$-monotone in $\Pbar$ according to \cref{def:monotonicity}.
\end{proposition}
\begin{proof}
   Assume $u$ is $\lambda$-monotone according to \cref{def:finite_dim_monotone_dynamics}. Then the monotonicity condition for $v[\rho]$ to check is 
    \begin{align*}
        \sum_{i=1}^n \int \<x_i-y_i,\int \nabla_{x_i} u_i(x) \d \rhozero_{-i}(x_{-i})-\int \nabla_{y_i} u_i(y) \d \rhoone_{-i}(y_{-i})>\d \gamma_i(x_i,y_i) \ge \lambda \Wbar(\rhozero,\rhoone)\,.
    \end{align*}
    Since the marginals of $\gamma_i$ are $\rhozero_i$ and $\rhoone_i$, we can rewrite the integral and apply the monotonicity condition for $u$
    \begin{align*}
        &\sum_{i=1}^n \int \<x_i-y_i,\int \nabla_{x_i} u_i(x) \d \rhozero_{-i}(x_{-i})-\int \nabla_{y_i} u_i(y) \d \rhoone_{-i}(y_{-i})>\d \gamma_i(x_i,y_i) \\
        &=  \int \sum_{i=1}^n\<x_i-y_i, \nabla_{x_i} u_i(x) - \nabla_{y_i} u_i(y) >\d \gamma(x,y) \ge \lambda \int \norm{x-y}\d \gamma(x,y)\,,
    \end{align*}
    which completes the proof of finite-dimensional monotonicity implying infinite dimensional monotonicity. To show the other direction, selecting Dirac deltas for each species gives the desired inequality.
\end{proof}

In contrast with the finite-dimensional setting, it is possible for monotonicity to fail for measures that are a single Dirac, while it holds for non-atomic measures; this corresponds to searching over mixed-Nash equilibria rather than just pure Nash equilibria. This occurs, for example, when $v_i[\rho](x_i)$ includes terms like $\log \rho(x_i)$ corresponding to linear diffusion. Our setting includes the finite dimensional setting as a special case (\cref{lem:monotonicity_finite_dimensions}), and the notion of monotonicity introduced in \cref{def:monotonicity} is more general.

\subsection{Contributions}
In this work, we define a notion of monotonicity for a system of coupled flows in Wasserstein space which results in convergence to a steady state. We then consider the setting of coupled gradient flows, and provide examples of systems that satisfy this monotonicity condition, both analytically for general functional classes and numerically for specific applications.
\paragraph{Monotone and Coupled Gradient Flows}
In the monotone flow setting, we prove convergence in a weighted $L^2$-norm of the velocity fields, convergence of the second moments to a ball where they remain for all time, and convergence of solutions to the monotone flow in a joint Wasserstein-2 metric to a unique steady state. In the coupled gradient flow setting, we show that the steady state is a Nash equilibrium of the game specified by the cost functions encoded by the dynamics. We show that our infinite-dimensional generalization of monotonicity encapsulates the Euclidean definition, and prove that monotonicity of the full system implies displacement convexity of each energy functional with respect to its own species.
\paragraph{Applications}
We include examples of monotone coupled gradient flow systems, with terms such as cross-diffusion, nonlocal interaction, and linear and nonlinear diffusion.  We simulate a four-team game in which two service providers compete for users in two strategic populations, which manipulate their data over time. Then we draw a connection to multi-team mean-field games parameterized by a particular class of cost functions whose solution is a monotone coupled gradient flow. Finally, we empirically demonstrate convergence of a degenerately monotone game (a game which is not monotone, but almost monotone), using diffusion to ensure sufficient probability mass over the monotone region to still observe convergence. This suggests the possibility for our results to be extended to perturbations of monotone games.

\subsection{Related Literature}
This work draws from the optimization, game theory and PDE literature, applying convergence criteria and techniques from game theory to systems of flow equations. Contraction techniques for single-species PDEs fit elegantly into the multispecies framework, and the analysis of infinite-dimensional games utilizes classical PDE techniques.

\paragraph{Monotone Variational Inequalities in Optimization}
The notion of (nonlinear) monotone operators in Banach spaces is well-studied; it was developed at least as early as 1965 in \cite{browder_nonlinear_1965}. While the results include infinite-dimensional Banach spaces, it does not apply to the space of probability measures, as the Wasserstein-2 metric is not induced by a norm. Therefore, one can view our notion of monotonicity as an extension of classical monotonicity results to a metric space. For an overview of related results, see \cite[Ch 3]{kinderlehrer_introduction_2000}, \cite{cottle_variational_1981}. Other results detail algorithms for finding critical points, which are guaranteed to exist under monotonicity. The extragradient method \cite{korpelevich_extragradient_1976} iterates by computing the gradient at a nearby point in the direction of the current gradient, and taking a step using the neighboring gradient. The results in \cite{bruck_weak_1977} show that monotonicity implies weak convergence of a weighted history of a gradient descent trajectory, and the formulation of monotonicity in \cite{chen_convergence_1997} hints at our formulation, by relating a monotone mapping on $\R^n$ to a set-valued mapping defined via the tangent space, which the authors use to derive a numerical method for finding solutions to a variational inequality. All of these works assume a Banach space structure; our work extends this to a non-Banach-space setting.

\paragraph{Game Theory} 
In the context of game theory, monotonicity emerges in \textit{diagonally strictly concave games} \cite{rosen_existence_1965}, as a condition under which existence of a unique Nash equilibrium can be proven for an $n$-player game. More recent results focus on learning algorithms which minimize functions  under the assumption of monotone cost structures \cite{tatarenko_learning_2019}, or in the presence of strategic data \cite{narang_learning_2022,narang_multiplayer_2023}.
From the dynamical systems perspective, monotonicity is closely tied to the notion of contractive dynamics \cite{bullo_contraction_2024}; when the dynamics evolve a vector in Euclidean space, a $\lambda$-monotone system 
is said to be $\lambda$-contractive.

The extension of finite-dimensional games to infinite-dimensional games has recently been explored in related work. One such work provides a gradient ascent-descent scheme for solving a minimax problem over measures \cite{lu_two-scale_2023}, while  \cite{wang_exponentially_2023} presents a particle method for computing mixed Nash equilibria for two-player zero-sum games. A similar problem was formulated as an adversarially-robust optimization problem in \cite{garcia_trillos_adversarial_2024}, in which the authors also analyze a Wasserstein gradient ascent-descent particle method. Convergence to the mean-field limit of a particular zero-sum particle system was shown in \cite{lu_convergence_2025}. Convergence in time in a zero-sum, infinite-dimensional setting was proven in \cite{conger_coupled_2024}, and it was shown in \cite{conger_strategic_2023} that in a timescale-separated setting, the minimax problem over measures is a gradient flow, providing convergence to the steady state under suitable conditions. Our framework generalizes these settings to $n$ players and classes of cost functions beyond zero-sum, i.e. beyond the setting in which one player minimizes a cost that the other player maximizes. Many real-world settings have a non-zero-sum structure, where results such as Danskin's theorem are no longer available and timescale separation can change the steady state.

Some recent work \cite{graber_monotonicity_2023} gives a few monotonicity conditions for a single-species mean field game, including the Lasry–Lions monotonicity condition; while our condition is similar, we consider a dynamical system rather than the optimization problem of a mean field game. In one of our examples, we investigate the application of our monotonicity condition to a particular multi-population game with functionals of the form in \cite{zhang_mean-field_2023}. Many results consider multi-population games, such as \cite{de_pascale_variational_2024,cirant_multi-population_2015}, examining the existence of equilibria or as a method for clustering \cite{aquilanti_mean_2021}. In our example, we analyze the equilibrium of the mean field game, which is a set of velocities, and show how the resulting dynamics have monotone structure.

\paragraph{Coupled PDEs} 
The literature on multispecies interactions via convolution kernels can be viewed through a game-theoretic lens. These models are common in biology for instance for modeling dynamics of chemotaxis among other phenomena; see references in \cite{conger_strategic_2023}. The majority of the PDE literature in this area concerns proofs for existence of solutions; for example, \cite{francesco_measure_2013} provides a proof for existence, with uniqueness under additional assumptions, for solutions of two species with self interaction and cross-interaction; the kernel can be mildly singular. Under stronger regularity assumptions and for a large class of kernels, we give conditions for convergence to a steady state.
Another work \cite{carrillo_zoology_2018} characterizes the stationary states of a number of interaction systems. The results in \cite{carrillo_measure_2020} as well as references therein are very relevant to our setting due to their characterization of the equilibria of a two-species gradient flow system with nonlocal interaction terms, in addition to proving existence and uniqueness of solutions. While our condition for convergence to the steady state is generally more restrictive than settings considered in existing literature on existence of solutions, we offer a starting point for a new set of results on convergence conditions for these dynamics. For example, quadratic and attractive-repulsive kernels satisfy our notion of monotonicity. The recent work \cite{du_particle_2024} poses a mean field variational inference problem as an optimization over $n$ measures. This can be equivalently stated as a cooperative $n$-player game, in which all players aim to minimize the same energy functional. Our setting generalizes this setting to non-cooperative energies. In fact, the $n$-player dynamics in this work can be considered within the framework of vector-valued optimal transport that has been proposed very recently in \cite{craig_vector_2025}. There, the authors formulate a multispecies optimal transport problem, where mass moves both within a species and between species. Analyzing dynamics which are a gradient flow with respect to the metric proposed in \cite{craig_vector_2025} is an interesting extension of this work; see \cite[Section 1.4]{craig_vector_2025} for details.
The notion of monotonicity (or accretivity, or dissipativity) for operators on Hilbert spaces used in \cite{brezis_hilleyosida_2011}, which we extend here to Wasserstein spaces, was used as a criterion for existence of solutions for measure evolution equations \cite{cavagnari_dissipative_2023}, in which the authors consider multi-valued vector fields; see also references therein.

\subsection{Structure}
We summarize relevant notation and definitions in \cref{sec:preliminaries}. Next, we present the existence of and convergence to the steady state results and proofs in \cref{sec:main_results}. \cref{sec:convexity} details properties of monotonicity in the coupled gradient flow setting. \cref{sec:applications} presents applications in machine learning and game theory, illustrating how our theoretical results can be applied. 
In \cref{sec:appendix}
we provide  additional monotonicity results, such the equivalence between first order and second order monotonicity conditions (\cref{sec:2nd-order-monotonicity}), and a comparison of finite dimensional and infinite dimensional monotonicity as well as comparison between convexity (single species) and monotonicity (multi-species)
(\cref{subsec:monotonicity_convexity}). Next,
\cref{sec:interaction_kernel} includes detailed calculations for interaction kernels supporting the results in \cref{subsec:monotonicity_cross_interaction}, and in \cref{sec:optimal_admissible_plans}, we prove \cref{thm:admissible_couplings}, which shows that optimal plans can be replaced by any admissible plan for the monotonicity condition in $\Pbar$.

\section{Preliminaries}\label{sec:preliminaries}
We consider a dynamical system with $n$ species. The set of measures with finite second moment is denoted by $\P_2$, and each species $\rho_i$ lives in $\P_2(\R^{d_i})$. 
Let $x=[x_1,\dots,x_n]\in\R^d$ where $d=\sum_{i=1}^n d_i$, $\rho\coloneqq[\rho_1,\dots,\rho_n]$ be the vector of all measures, and $\rhoni$ denotes the vector of all non-$i^{th}$ species, $\rhoni\coloneqq[\rho_1, \dots,\rho_{i-1},\rho_{i+1},\dots,\rho_n]$. Similarly, for a vector-valued function $\varphi:\R^d \to \R^n$, $\nv_{-i}=[\nabla_{x_1} \varphi_1, \dots,\nabla_{x_{i-1}} \varphi_{i-1},\nabla_{x_{i+1}} \varphi_{i+1},\dots,\nabla_{x_n} \varphi_{n}] $ and $x_{-i}=[x_1,\dots,x_{i-1},x_{i+1},\dots,x_n]$.  The notation $\int \norm{x}^2\d \rho(x)$ means
\begin{align*}
    \int \norm{x}^2 \d \rho(x) = \sum_{i=1}^n \int \norm{x_i}^2 \d \rho_i(x_i)\,.
\end{align*}
 Let $C_c^\infty$ denote the set of infinitely-differentiable functions with compact support, $C^k$ the set of $k$-times continuously differentiable functions, and $\P^{ac}(\R^{d_i})$ the set of probability measures on $\R^{d_i}$ that are absolutely continuous with respect to the Lebesgue measure. The Lebesgue measure on $\R^d$ is denoted by $\L^{d}$. The set of functions which are square integrable with respect to $\rho_i$ are denoted by $L^2(\rho_i)$. The narrow topology is defined as convergence in duality with continuous bounded functions and the weak-* topology is defined in duality with continuous functions vanishing at infinity. We will use the following related notion of convergence, which we refer to as \emph{weak topology}.
\begin{definition}[Weak Convergence]\label{def:weakcv}
    A sequence of measures $(\rho_n)$ converges in the \emph{weak topology}, denoted by $\rho_n \rightharpoonup \rho$, when $(\rho_n)$ converges narrowly, that is, in duality with all continuous bounded functions, and there exists a uniform second moment bound for $(\rho_n)$.
\end{definition}
 Let $\delta_{\rho_i}$ denote the first variation with respect to $\rho_i$ and $\nabla_{x_i}$ the gradient with respect to $x_i$. For matrices $A,B$ of the same dimension, the operator $A \succ B$ denotes that $A-B$ is positive-definite. Let $\Id_d$ denote the  identity matrix of dimension $d\times d$ and $\id$ the identity operator.
Throughout this paper, we use the following notion for interpolation along transport plans and geodesics.
\begin{definition}[Displacement Interpolants]\label{def:disp_interpolant} 
    Define the displacement interpolant (or geodesic) between $\rhozero$ and $\rhoone$ as $\rhos_i = I^{(s)}_{\sharp} \gamma_i$, where $I^{(s)}(x,y)=(1-s)x + s y$, for all $s\in[0,1]$ and $\gamma\in\Gamma^*(\rhozero,\rhoone)$.
    If a transport map exists from $\rhozero_i\in\P_2$ to $\rhoone_i\in\P_2$ for all $i$, then there exists convex functions $\varphi_i:\R^{d_i}\to\R$ such that $\rhoone_i = [\nabla \varphi_i ]_\sharp \rhozero_i$. Then $\rhos$ can also be expressed as
$\rhos_i = [(1-s)\id+s\nabla \varphi_i]_{\sharp}\rhozero_i$. We define the operator $\nv(x) \coloneqq [\nabla \varphi_1(x_1), \dots, \nabla \varphi_n(x_n)]$ as the vector of the convex functions which define the parameterization. 
\end{definition}
Any displacement interpolant $\rhos$ satisfies the following geodesic conservation of mass and momentum equations:
\begin{align}\label{eq:geodesic_eqns}
    \partial_{s} \rhos_i + \divs{x_i}{\rhos_i \ws_i} = 0 \,, \qquad \partial_{s}(\rhos_i \ws_i) + \divs{x_i}{\rhos_i \ws_i \otimes \ws_i} = 0\,,
\end{align}
and if an invertible pushforward map exists between $\rhozero$ and $\rhoone$, then $w^{(0)}(x)= \nv(x)-x$ and $w^{(1)}(x) = x-(\overline{\nabla \varphi})^{-1}(x)$, where $(\overline{\nabla \varphi})^{-1}(x)$ is defined for a.e. $x\in\supp \rhoone$.

A number of convexity results follow from the monotonicity definition (\cref{def:monotonicity}). The notion of convexity we use is \textit{(uniform) displacement convexity}, which can be viewed as convexity of a functional along $\Wass_2$ geodesics between any two measures in the domain of the functional.
\begin{definition}[Displacement Convexity]\label{def:displacement_convexity}
    A functional $G:\P_2(\R^{d_i})\to\R$ is \emph{displacement convex} if for all $\rhozero_i,\rhoone_i$ that are atomless we have
    \begin{align*}
        G(\rhos_i) \leq (1-s)G(\rhozero_i) + s G(\rhoone_i)\,,
    \end{align*}
    where $\rhos_i$ is given in \cref{def:disp_interpolant}. 
    Further, $G:\P_2(\R^{d_i})\to\R$ is \emph{uniformly displacement convex with constant $\lambda>0$} if 
    \begin{align*}
        G(\rhos_i) \leq (1-s)G(\rhozero_s) + s G(\rhoone_s)-s(1-s) \frac{\lambda}{2} \Wass_2(\rhozero_i,\rhoone_i)^2\,.
    \end{align*}
\end{definition}
Uniform convexity is analogous to strong convexity for Euclidean functions, with the inequality using the Wasserstein-2 distance  instead of the two-norm. 
Equivalent first-order conditions for uniform displacement convexity with constant $\lambda$ are, where $\gamma\in\Gamma^*(\rhozero,\rhoone)$,
\begin{subequations}\label{eq:first_order_convexity}
\begin{align}
    \int \<\nabla \delta_\rho G[\rhozero](x), y - x> \d \gamma(x,y) \le G(\rhoone) - G(\rhozero) -&\frac{\lambda}{2} \Wass_2(\rhozero,\rhoone)^2\,, 
    \\
    \text{ or } \qquad \int \<x-y,\nabla \delta_\rho G[\rhozero](x)-\nabla \delta_\rho G[\rhoone]( y)> \d \gamma(x,y) \ge& \lambda \Wass_2(\rhozero,\rhoone)^2\,.
\end{align}
\end{subequations}
For details regarding these convexity inequalities, see \cite[Proposition 16.2]{Villani07}. 
A sufficient second-order condition for uniform displacement convexity with constant $\lambda$ is
\begin{align*}
    \ddss G(\rhos_i) \ge \lambda \Wass_2(\rhozero_i,\rhoone_i)^2\,;
\end{align*}
see \cite{mccann_convexity_1997} for displacement convexity details and \cref{tab:convexity_monotonicity_comparison_first_order,tab:convexity_monotonicity_comparison_second_order} for comparison with monotonicity conditions.

For this work, we are not concerned with existence results for solutions to \eqref{eq:monotone_flow_dynamics} and assume enough regularity of the velocity fields so that solutions $\rho(t)\in\Pbar$ exist.
To analyze the long-time behavior of the system, we define the steady-state of the dynamics \eqref{eq:monotone_flow_dynamics}.
\begin{definition}[Steady states for \eqref{eq:monotone_flow_dynamics}]\label{def:steady_state_v}
    Given $\rho^\infty\in \Pbar$, then $\rho^\infty$ is a steady state for the system \eqref{eq:monotone_flow_dynamics} if it satisfies
\begin{align}
   \int \nabla \phi \cdot v_i[\rho^\infty](x_i) \d \rho_i^\infty(x_i) = 0 \quad \forall\, \phi\in C_c^\infty(\R^{d_i})\,,\  \forall\, i\in [1,\dots,n] \label{eq:steady_state_v}\,.
\end{align}
\end{definition}

In the coupled gradient flow setting, each player $i$ evolves in the direction of steepest descent of $F_i$ with respect to its own measure, in the Wasserstein-2 metric.
Importantly, the dynamics \eqref{eq:coupled_gradient_flow_dynamics} are \textit{not} necessarily a joint gradient flow; rather, each player is doing gradient descent along its own energy functional that depends on the other agents. 
We will prove conditions under which dynamics with corresponding functionals that satisfy monotonicity will converge to a unique steady state, which is also the unique Nash equilibrium of the game $F=(F_i)$. Using the standard definition of Nash equilibrium, we present the definition of a Nash equilibrium in terms of the action $\rho_i$ and objective $F_i$ of each player.
\begin{definition}[Nash Equilibrium]\label{def:Nash_eq}
    A set of measures $\rho^*\in\A$ is a Nash equilibrium over $\A$ if,
    \begin{align}\label{eq:mixed_nash_equilibrium}
        F_i[\rho_i^*,\rhoni^*] \le F_i[\rho_i,\rhoni^*] \quad \forall \ \rho_i\in \A_i\,, \quad \forall\, \iin \,.
    \end{align}
\end{definition}
A set of actions $\rho^*$ is a Nash equilibrium when for each player $i$, all other players' actions are fixed and the $i^{th}$ player has no incentive to deviate from its current strategy. Our notion of monotonicity can been see as a generalization of monotonicity for mixed Nash equilibria.

\section{Main Results}\label{sec:main_results}
Our key result is convergence of the coupled flow system \eqref{eq:monotone_flow_dynamics} to a unique steady state. We prove that monotonicity implies that the $L^2$ norm of the velocity fields converge to zero, and the second moments converge to a ball where they remain for all time. 
We prove a contraction result in $\Wbar$, which results in the existence of a unique steady state. We can then show convergence from an initial condition to that steady state. In the setting where the system is a coupled gradient flow, we also show that the steady state corresponds to a Nash equilibrium.

For the remainder of the manuscript, any $\lambda$-monotone vector field $v$ will be considered with a domain that is geodesically convex in $\Pbar$. Below, we summarize additional assumptions used for our main results to hold.
\begin{assumption}[Absolutely Continuous Curves]\label{assump:loc_abs_cont_curves}
    The solution $\rho(t)$ to \eqref{eq:monotone_flow_dynamics} is absolutely continuous, that is, there exists some $g\in L^1([0,\infty))$ such that
    \begin{align*}
        \Wbar(\rho(t),\rho(s)) \le \int_t^s g(\tau)  \d \tau\,, \quad \forall\ 0 \le t \le s < \infty\,.
    \end{align*}
\end{assumption}
\begin{assumption}[Lower Semicontinuity]\label{assump:F_lsc}
    For all $\iin$ and for any fixed $\rho\in \A$, the functional $F_i[\cdot,\rhoni]$ is lower-semicontinuous with respect to the weak topology over the domain $\A_i$.
\end{assumption}

\cref{assump:loc_abs_cont_curves} is needed to compute the time derivative of the $\Wbar$ metric.
\cref{assump:F_lsc} is used to prove the existence of minimizers for each $F_i[\cdot,\rhoni]$ in the coupled gradient flow setting, which is then applied together with monotonicity to show that the steady state is the unique Nash equilibrium.

The next three assumptions are not used in the main theorem; they are relevant for proving equivalence of first- and second-order notions of monotonicity, which is then used to show convergence of velocity fields. More precisely, \cref{assump:regularity_v} allows us to compute the time derivative of the $L^2$ norm of the velocities, and \cref{assump:continuity_v,assump:tangent_space} are used to construct a local, second-order notion of monotonicity (\cref{lem:second_order_monotonicity_v}). 
These assumptions are not chosen to be optimal, and we expect relaxations  to be possible.

\begin{assumption}[Differentiability of $v$]\label{assump:regularity_v}
    The terms $J_1[\rho]=\R^d\to\R^{d\times d}$ and $J_2[\rho]:\R^d \times \R^d \to \R^{d\times d}$ given by
    \begin{align*}
    \left(J_1[\rho](x) \right)_{ij} &\coloneqq 
    \begin{cases}
        \nabla_{x_i} v_i[\rho](x_i) & \text{if } i=j \\
        0 & \text{if } i \ne j
    \end{cases} \quad \in \R^{d_i\times d_i} \,, \\
     ( J_2[\rho](x,\hat x))_{ij} &:= \nabla_{\hat x_j} \delta_{\rho_j} v_i[\rho](x_i,\hat x_j) \qquad \in \R^{d_i \times d_j} \,,
\end{align*}
    are well-defined for $\rho$-a.e. $x$ for all $\rho\in\A$. 
\end{assumption}
\begin{assumption}[Continuity of $v$]\label{assump:continuity_v}
The functions $J_1,J_2$ defined in \cref{assump:regularity_v}
satisfy the following continuity property: If $\rho_\eps\to \bar\rho$ in $\Wbar$ as $\eps\to 0$ and $\rho_\eps, \bar\rho\in\A$, then for any $w\in L^2(\bar\rho)$,
\begin{align*}
     &\iint \<w(x),J_2[\rho_\eps](x+\eps w(x),\hat x+\eps w(\hat x))\cdot w(\hat x) >\d \bar\rho(x) \d \bar\rho(\hat x)
     +\int\<w(x),J_1[\rho_\eps](x+\eps w(x)) \cdot w(x) > \d \bar\rho(x) \\
    &\quad\to \iint \<w(x),J_2[\bar\rho](x,\hat x)\cdot w(\hat x) >\d \bar\rho(x) \d \bar\rho(\hat x)  +\int\<w(x),J_1[\bar\rho](x) \cdot w(x) > \d \bar\rho(x) \text{ as } \eps\to 0\,.
\end{align*}
\end{assumption}
\begin{assumption}[Admissible Velocity Fields]\label{assump:tangent_space}
    The velocity fields satisfy $v_i[\rho]\in Y_i(\rho_i,\A_i)$ for all $\rho\in\A$ and all $\iin$, where
        \begin{align*}
       &Y_i(\rho_i,\A_i)= \overline{\{c_i(r_i-\id)\in L^2(\rho_i) \ : \ (\id \times r_i)_\sharp \rho_i \in \Gamma^*(\rho_i,(r_i)_\sharp \rho_i)\,,\ (r_i)_\sharp \rho_i\in \A_i  \} }^{L^2(\rho_i)}
    \end{align*}
   where $ r_i:\R^{d_i}\to\R^{d_i}$ and $c_i>0$ is a constant.
\end{assumption}
 We can view the set $Y_i(\rho_i,\A_i)$ as a subset of the tangent space at $\rho_i$; for the setting where $\A_i = \P_2(\R^{d_i})$, it was shown in \cite[Theorem 8.5.1]{ags} that $Y_i$ is exactly the Wasserstein-2 tangent space at $\rho_i$. For example, if $A_i=L^1(\R^{d_i})$, then $r_i=0$ is not in $Y_i$ since it pushes $\rho_i\in L^1(\R^{d_i})$ to a Dirac, which is not in $ L^1(\R^{d_i})$; using the entire Wasserstein-2 tangent space may push densities outside of $A_i$.

\begin{remark}\label{rmk:assumpt-equiv}
Absolute continuity of solution curves (\cref{assump:loc_abs_cont_curves}) can be obtained as a consequence of related properties.
\begin{enumerate}
    \item[(i)]\label{rmk:assumpt-equiv-one} For instance, a sufficient condition for \cref{assump:loc_abs_cont_curves} to hold is $v[\rho(t)](\cdot)\in L^2(\rho(t))$ for all $t\ge 0$ and $\int_0^\infty \int\|v[\rho(t)](x)\|^2\d \rho(t)\d t<\infty$ \cite[Theorem 5.14]{santambrogio_OTAM}.
    \item[(ii)]\label{rmk:assumpt-equiv-two} If $\rho(t)\in\A$ for all $t\ge 0$, and \cref{assump:regularity_v,assump:continuity_v,assump:tangent_space} together with 
    \begin{equation*}
        \frac{1}{2}\sum_{i=1}^n \int \norm{v_i[\rho(0)](x_i)}^2  \d \rho_i(0,x_i)<\infty
    \end{equation*}
    hold, it follows from $\lambda$-monotonicity for $\lambda> 0$ (\cref{lem:convergence_D_v}) that the sufficient condition in (i) holds, and so \cref{assump:loc_abs_cont_curves} is satisfied.
\end{enumerate}
    \end{remark}
For a given velocity field $v:\A\to\R^d$, we consider the set
\begin{align*}
    S = \left\{ \rho \in \Pbar \ : \ \int_\Omega \left\| v[\rho](x)\right\|^2 \d \rho(x) < \infty \quad \forall \ \Omega\subset \R^d\,,\ \Omega\ \text{compact} \right\}\,.
\end{align*}
Note that for any steady state $\rho^\infty$ of the flow \eqref{eq:monotone_flow_dynamics} satisfying \cref{def:steady_state_v}, we have $v_i[\rho^\infty](x_i)=0$ for $\rho_i^\infty$-a.e. $x_i$ and therefore $\rho^\infty\in S$.

\begin{theorem}[Main Result]\label{thm:cv_W2}
   \sloppy Let \cref{assump:loc_abs_cont_curves} hold. Suppose \eqref{eq:monotone_flow_dynamics} is $\lambda$-monotone in $\A\subseteq\Pbar$ with $\lambda>0$, and $\rho(t)$ evolves according to \eqref{eq:monotone_flow_dynamics}. If $\rho(t)\in\A$ for all $t\ge 0$, then
    \begin{enumerate}
        \item[(a)] there exists a unique steady state $\rho^\infty\in S$ of \eqref{eq:monotone_flow_dynamics} and if $\A$ is compact in $\Wbar$, then $\rho^\infty\in S\cap\A$;
        \item[(b)] if $\rho^\infty\in S\cap\A$, the solution to \eqref{eq:monotone_flow_dynamics} $\rho(t)$ converges exponentially fast to $\rho^\infty$ in $\Wbar$ with rate $\lambda$,
        \begin{align*}
            \Wbar(\rho(t),\rho^\infty) \le e^{-\lambda t } \Wbar(\rho(0),\rho^\infty)\,.
        \end{align*}
        \item[(c)] If, additionally,  $v_i=-\nabla_{x_i}\delta_{\rho_i} F_i[\rho]$, \cref{assump:F_lsc} holds, and $\rho^\infty \in \A \cap S$, then $\rho^\infty$ is the unique Nash equilibrium of $F$ in $\A$.
    \end{enumerate}
\end{theorem}
We postpone the proof of this result to \cref{sec:proof-main-thm}. In the case where $\A$ is not compact, for specific functionals it can still be shown that $\rho^\infty \in \A$. See \cite{conger_coupled_2024} for an example where diffusion is included in the evolution,   and the steady state is shown to be in $\A=L^1(\R^d)$, which is not compact in the weak topology on $\P_2(\R^d)$.

In finite-dimensional game theory, convergence to Nash equilibria can be proven using a Lyapunov function dependent on the velocity, $V_2(x)=\frac{1}{2}\sum_{i=1}^n \norm{u_i(x)}^2$. The analogous infinite-dimensional Lyapunov function offers insight into the $L^2$-norm of the velocity fields. In the single-player gradient flow setting where $v_1[\rho_1]=-\nabla \delta_{\rho_1}F_1[\rho_1]$, the Lyapunov function is closely related to the \textit{energy dissipation}.  It is the energy dissipation that would result from a single, decoupled gradient flow, but in the multispecies setting, each player's velocity field in its dynamics is dependent on the other players' actions. We present a convergence result for this functional and comment on the additional requirements needed to show convergence. We define the Lyapunov function $D:\A \to \R$
\begin{align}
    D(\rho) &= 
        \sum_{i=1}^n D_i(\rho) 
    \,, \quad D_i(\rho) = \frac{1}{2}\int \norm{v_i[\rho](x_i)}^2  \d \rho_i(x_i) \,.\label{eq:Lyap_D_v}
\end{align}
We show that $D(\rho)$ converges to zero exponentially with a rate dependent on the monotonicity parameter.

\begin{proposition}[Convergence of $D$]\label{lem:convergence_D_v}
Let $v$ be $\lambda$-monotone in $\A$, let \cref{assump:regularity_v,assump:continuity_v,assump:tangent_space} hold, and suppose $\rho(t)\in\A$ for $t\ge 0$, with $D(\rho(0))<\infty$. Then $D$ as defined in \eqref{eq:Lyap_D_v} along solutions to \eqref{eq:monotone_flow_dynamics} satisfies
\begin{align}\label{eq:velocity_dissipation}
    D(\rho(t)) \le e^{-2\lambda t} D(\rho(0)) \quad \forall\, t\ge 0\,.
\end{align}
\end{proposition}
\begin{proof}
    First, computing the derivative of $D_i$ along solutions of \eqref{eq:monotone_flow_dynamics},
    \begin{align*}
    \dot D_i(\rho) &= \sum_{j=1}^n \iint \<v_i[\rho](x_i), \delta_{\rho_j} v_i[\rho](x_i,\bar x_j)> \partial_t \rho_j(\bar x_j) \d \rho_i(x_i) \d \bar x_j 
    + \frac{1}{2} \int \|v_i[\rho](x_i)\|^2 \partial_t\rho_i(x_i)\d x_i\,.
\end{align*}
Substituting the dynamics and using integration by parts we have 
\begin{align*}
    \dot D_i(\rho)& = \sum_{j=1}^n \bigg(\iint \< v_i[\rho](x_i),\nabla_{ \bar x_j} \delta_{ \rho_j} v_i[\rho](x_i,\bar x_j) \cdot v_j[\rho](\bar x_j) >\d \rho_j(\bar x_j) \d \rho_i(x_i)
 \\&\qquad \qquad 
 +  \int \<v_i[\rho](x_i), \nabla_{x_i} v_i[\rho](x_i) \cdot v_i[\rho](x_i) >\d \rho_i(x_i) \bigg) \,.
\end{align*}
Combining the derivatives for all species into vectors and matrices, we have
\begin{align*}
    \dot D(\rho) &=   \int v[\rho](x)^\top J_1 [\rho](x) v[\rho](x) \d \rho(x) +\iint v[\rho](x)^\top J_2[\rho](x,\bar x) v[\rho](\bar x)\d\rho(x) \d\rho(\bar x) \,.
\end{align*}
Since $D(\rho(0))<\infty$ and the second-order derivative terms are well-defined by \cref{assump:regularity_v}, $\dot D(\rho(0))$ is also well-defined. Since $v_i[\rhozero]\in Y_i(\rhozero_i,\A_i)$ thanks to \cref{assump:tangent_space}, applying the monotonicity result \cref{lem:second_order_monotonicity_v} at $t=0$ with $w(x)=v[\rhozero](x)$, we have 
\begin{align*}
    \dot D(\rhozero) &\le -\lambda \int \|v[\rhozero](x)\|^2 \d \rhozero(x) = -2\lambda D(\rhozero)\,,
\end{align*}
and Gr\"onwall's lemma gives an exponential estimate of $D$ with rate $2\lambda$, allowing  the above argument to be applied for any $t\ge0$.
\end{proof}

\begin{remark}
    If $\A = \Pbar$, then we can use any admissible map in the proof of \cref{lem:second_order_monotonicity_v}, due to \cref{thm:admissible_couplings}, and so \cref{lem:convergence_D_v} holds without \cref{assump:continuity_v,assump:tangent_space}.
\end{remark}

\subsection{Second Moments}
Monotonicity is closely tied to properties of the second moment of each species. We detail conditions under which the second moment is uniformly bounded. We use $M(\rho):\Pbar \to \R$ to denote the sum of the second moments of all species, 
\begin{align}\label{def:second_moment}
    M(\rho) = \frac{1}{2}\int \norm{x}^2 \d \rho(x) = \frac{1}{2}\sum_{i=1}^n \int \norm{x_i}^2 \d \rho_i(x_i)\,.
\end{align}
The time derivative of the second moment along solutions to \eqref{eq:monotone_flow_dynamics} is
\begin{align}\label{eq:2nd_moment_derivative_estimate}
    \ddt M(\rho(t))=  \int \<x,v[\rho](x)>\d\rho(t,x) \quad \forall\, t\ge 0\,.
\end{align}
Using monotonicity, we upper-bound the right hand side of \eqref{eq:2nd_moment_derivative_estimate}. The following properties specify settings where the second moment is not blowing up in finite time (\cref{property:finite_2nd_moment}) and where the second moment is uniformly bounded for all time (\cref{property:bdd_2nd_moment}).

\begin{property}[Finite Second Moments]\label{property:finite_2nd_moment}
Given an initial condition $\rho(0)\in\Pbar$, the solution $\rho(t)$ to \eqref{eq:monotone_flow_dynamics} satisfies
\begin{align*}
   M(\rho(t)) <\infty \quad \forall\, t\ge 0\,.
\end{align*}
\end{property}

\begin{property}[Uniformly Bounded Second Moments]\label{property:bdd_2nd_moment}
Given an initial condition $\rho(0)\in\Pbar$, the solution $\rho(t)$ to \eqref{eq:monotone_flow_dynamics} satisfies, for some $c\in\R_+$,
\begin{align*}
     M(\rho(t)) \le c  \quad \forall\, t\ge 0\,,
\end{align*}
\end{property}

While
\cref{property:finite_2nd_moment} is a sufficient condition for applying a contraction theorem in the analysis of $\Wbar$ convergence, it is possible for it to hold while monotonicity fails. In the single-species setting, the solution to $\partial_t \rho_1(t)=\div{\rho_1(t)\nabla \log\rho_1(t)}=\Delta \rho_1(t)$ satisfies \cref{property:finite_2nd_moment} but does not satisfy strict monotonicity or \cref{property:bdd_2nd_moment}.
\cref{property:bdd_2nd_moment} can be shown using monotonicity, with additional assumptions needed depending on the specific choice of $\A$. First, we show that monotonicity in $\Pbar$ results in \cref{property:bdd_2nd_moment} holding under mild conditions. 
Let $\bar \delta_0 \coloneqq [\delta_0,\dots,\delta_0]$ be the Dirac at the origin; in the proof of the following lemma, we will use that the second moment of $\rho$ can be characterized via $M(\rho)=\frac{1}{2}\Wbar(\rho,\bar \delta_0)^2$.
\begin{lemma}[Uniformly Bounded Second Moments in Domains with Diracs]\label{lem:monotonicity_in_Pbar}
  Assume $\bar \delta_0\in\A$. Let \cref{assump:regularity_v,assump:continuity_v,assump:tangent_space} hold, $v:\A \to \R^d$ be $\lambda$-monotone in $\A$ with $\lambda>0$, and $\rho(t)\in\A \ \forall\, t\ge 0$ evolves according to \eqref{eq:monotone_flow_dynamics}. If $\norm{v[\bar \delta_0](0)}<\infty$, $D(\rho(0))<\infty$ and for all $\rho\in\A$, the velocity fields $v_i[\rho]$ satisfy
    $\nabla_{x_i} v_i[\rho]\in L_{loc}^\infty(\R^{d_i};\R^{d_i}\times \R^{d_i})$, then \cref{property:bdd_2nd_moment} holds.
\end{lemma}
\begin{proof}[Proof of \cref{lem:monotonicity_in_Pbar}]
     The sum of the second moments can be equated to $\Wbar(\bar \delta_0,\rho(t))^2$ since $\bar \delta_0= \nv(t)_\sharp \rho(t)$ with $\nv(t,x)=0$ for a.e. $x\in\supp \rho$ an optimal map, and therefore
    \begin{align*}
        \Wbar(\bar \delta_0,\rho(t))^2 = \sum_{i=1}^n \int \norm{x_i}^2 \d \rho_i(t,x_i) \,.
    \end{align*}
    Because \eqref{eq:monotone_flow_dynamics} is $\lambda$-monotone in $\A$ and $\A$ contains $\bar \delta_0$, the monotonicity condition ensures that
    \begin{align*}
         \int \<x,v[\rho(t)](x)-v[\bar \delta_0](0)> \d \rho(t,x) \le -\lambda\Wbar(\bar \delta_0,\rho(t))^2  \,.
    \end{align*}
    We use \cref{assump:regularity_v,assump:continuity_v,assump:tangent_space} to conclude from \cref{lem:convergence_D_v} that the $L^2$ norm of the velocity field is bounded.
    This allows us to apply \cite[Theorem 23.9]{Villani07} to compute the time derivative of the Wasserstein metric,
    \begin{align*}
        \ddt \Wbar(\rho(t),\bar \delta_0)^2 &= 2\int \<x,v[\rho(t)](x)> \d \rho(t,x)\,.
    \end{align*}
    Substituting  the monotonicity estimate for $\int \<x,v[\rho(t)](x)-v[\bar \delta_0](0)> \d \rho(t,x)$ results in
    \begin{align*}
        \ddt \Wbar(\rho(t),\bar \delta_0)^2 &\le -2\lambda \Wbar(\rho(t),\bar \delta_0)^2  +2 \int \<x,v[\bar \delta_0](0)>\d \rho(t,x) \\
        &\le -2\lambda \Wbar(\rho(t),\bar \delta_0)^2 + 2\left( \int \norm{x}^2 \d \rho(t,x) \right)^{1/2} \norm{v[\bar \delta_0](0)} \\
        & \le -2\lambda \Wbar(\rho(t),\bar \delta_0)^2 + c \Wbar(\rho(t),\bar \delta_0)\,,
    \end{align*}
    where $c=2\norm{v[\bar \delta_0](0)}<\infty$ by assumption, and the H\"older inequality was used on the second line.
    Then we have that $\Wbar(\rho(t),\bar \delta_0)^2=2M(\rho(t))$ exponentially converges to a ball with radius dependent on $c$ and $\lambda$. Therefore $M(\rho(t))$ is uniformly bounded for all time inside a ball of size $R$, where
    \begin{align*}
        R = \max \left \{ \frac{c}{4 \lambda}, M(\rho(0)) \right \} \,.
    \end{align*}
\end{proof}

 While \cref{lem:monotonicity_in_Pbar} applies to the  setting when $\A$ contains a Dirac,
it is possible for the results to hold in greater generality. When Diracs are not included in $\A$, such as when $v$ contains linear or nonlinear diffusion, applying monotonicity of \eqref{eq:monotone_flow_dynamics} in $\Pbar$  to control the second moment is not feasible; we provide an extension of the previous lemma below.
To show that the second moments are bounded under more general conditions, we will use a  mollifier $\rhotau\in\A$ as a proxy for $\bar \delta_0$, which satisfies $\Wbar(\rhotau,\bar \delta_0) < \infty$. For example, when $\A=L^1(\R^{d_i})^{\otimes n}$, selecting
    \begin{align*}
        (\rhotau(x))_i \coloneqq \frac{1}{(2\pi \tau)^{d_i/2}} \exp\left(-\frac{\norm{x_i}^2}{2\tau} \right)\,,
    \end{align*}
    results in $\Wbar(\rhotau,\bar \delta_0) =\sqrt{d\tau}$, and allows us to apply the monotonicity inequality on $\rhotau$ when it cannot be applied to $\bar \delta_0$.
    This mollifier can be selected to have tails appropriate for the given velocity fields. The following lemma gives conditions under which \cref{property:bdd_2nd_moment} is satisfied.

\begin{lemma}[Bounded Second Moments]\label{lem:bdd_second_moment}
     Let $v$ be $\lambda$-monotone in $\A$, $\rho(t)\in \A$ evolve according to \eqref{eq:monotone_flow_dynamics} for all $t\ge 0$, and \cref{assump:loc_abs_cont_curves} hold. Fix $\tau>0$. 
   If there exists $\rhotau\in\A$ such that 
   $$c_\tau \coloneqq \left(\int \norm{v[\rhotau](x)}^2 \d \rhotau(x) \right)^{1/2} < \infty \quad \text{ and } \Wbar(\rhotau,\bar \delta_0)<\infty\,,$$
then
   \begin{align*}
       \sqrt{2 M(\rho(t))} \le \begin{cases} c_1 + c_0 e^{-\lambda t}  & \text{if } \lambda \ne 0 \\
        c_2 + c_\tau t & \text{if } \lambda=0 \,,
        \end{cases}
    \end{align*}
    where 
    $$c_0\coloneqq \max\left\{0,\Wbar(\rho(0),\rhotau)-\frac{c_\tau}{\lambda}\right\},\ \ 
    c_1 = \frac{c_\tau}{\lambda}+\Wbar(\rhotau,\bar \delta_0)\,,\ \ 
    c_2=\Wbar(\rho(0),\rhotau) + \Wbar(\rhotau,\bar \delta_0)\,.$$
    In particular, if $\lambda>0$, then $\rho(t)$ satisfies \cref{property:bdd_2nd_moment}.
\end{lemma}
\begin{proof}
The second moment \eqref{def:second_moment} can be upper-bounded using the triangle inequality:
    \begin{align*}
    \sqrt{2M(\rho(t))}=\Wbar(\rho(t),\bar \delta_0) \le \Wbar(\rho(t),\rhotau) + \Wbar(\rhotau,\bar \delta_0)\,.
\end{align*}
Let $\gamma_i(t)\in\Gamma^*(\rho_i(t),\rhotau_i)$. Using \cref{assump:loc_abs_cont_curves} to apply \cite[Theorem 3.14]{cavagnari_dissipative_2023} for the time derivative of $\Wbar(\rho(t),\rhotau)^2$, 
 \begin{align*}
     \frac{1}{2}\ddt \Wbar(\rho(t),\rhotau)^2 &= \int \<x-y,v[\rho(t)](x)> \d \gamma(t,x,y) 
     \\
     &\le
     -\lambda \Wbar(\rho(t),\rhotau)^2 + \int\<x-y,v[\rhotau](y)> \d \gamma(t,x,y)\,,
 \end{align*}
 where the last line follows from adding and subtracting $\int\<x-y,v[\rhotau](y)>\d \gamma(t,x,y)$ and applying the monotonicity inequality. We use H\"older's inequality to obtain
 \begin{align*}
     \int\<x-y,v[\rhotau](y)> \d \gamma(t,x,y) &\le \bigg(\int \norm{x-y}^2 \d \gamma(t,x,y) \int \norm{v[\rhotau](x)}^2 \d \rhotau(x) \bigg)^{1/2}
     =c_\tau \Wbar(\rho(t),\rhotau)\,.
 \end{align*}
Combining the last two estimates,
 \begin{align*}
     \frac{1}{2}\ddt \Wbar(\rho(t),\rhotau)^2  \le
     -\lambda \Wbar(\rho(t),\rhotau)^2 + c_\tau \Wbar(\rho(t),\rhotau)\,.
 \end{align*}
 When $\lambda \neq 0$, 
 we have
 \begin{align*}
     \Wbar(\rho(t),\rhotau) \le \frac{c_\tau}{\lambda} + c_0e^{-\lambda t} \,,
 \end{align*}
 where $c_0= \max\{ 0,\Wbar(\rho(0),\rhotau)-\frac{c_\tau}{\lambda} \}$.
 When $\lambda=0$, 
 \begin{align*}
     \Wbar(\rho(t),\rhotau) \le \Wbar(\rho(0),\rhotau) + c_\tau t\,.
 \end{align*}
 Using the triangle inequality, the second moments of $\rho(t)$ satisfy
\begin{align*}
    \sqrt{2M(\rho(t))}= \Wbar(\rho(t),\bar \delta_0) \le \begin{cases}
        \frac{c_\tau}{\lambda} + c_0 e^{-\lambda t} + \Wbar(\rhotau,\bar \delta_0) & \text{if } \lambda \ne 0 \\
        \Wbar(\rho(0),\rhotau) + c_\tau t + \Wbar(\rhotau,\bar \delta_0) & \text{if } \lambda =0
    \end{cases}\,.
\end{align*}
Setting $c_1 = \frac{c_\tau}{\lambda} + \Wbar(\rhotau,\bar \delta_0)$ and $c_2=\Wbar(\rho(0),\rhotau) + \Wbar(\rhotau,\bar \delta_0)$ gives the desired inequality. From this inequality, we conclude that
if $\lambda>0$, then \cref{property:bdd_2nd_moment} is satisfied and if $\lambda \le 0$, then \cref{property:finite_2nd_moment} is satisfied.
\end{proof}

\subsection{Proof of main theorem}\label{sec:proof-main-thm}
In this section, we will prove our main result, \cref{thm:cv_W2}. The components of the proof and related assumptions are illustrated in \cref{fig:assumptions}. Monotonicity is used to prove convergence of the velocity fields and a uniform bound on the second moments; in the setting where invertible pushforwards exist between any two points along the trajectory, these two key results could then be used to prove the main contraction and convergence results. However, to generalize the results to transport plans, \cref{assump:loc_abs_cont_curves} takes the place of these results. In the coupled gradient flow setting, we use the fact that monotonicity of $F$ implies displacement convexity of each $F_i$, which we discuss in detail later in \cref{sec:convexity}. This, along with the lower semicontinuity assumption, allows us to conclude that the steady state $\rho^\infty$ is the Nash equilibrium of  $F$. 
\begin{figure}[tbhp]
\includegraphics[width=0.99\textwidth, trim={1.0cm 3.3cm 1.4cm 2.5cm}, clip]{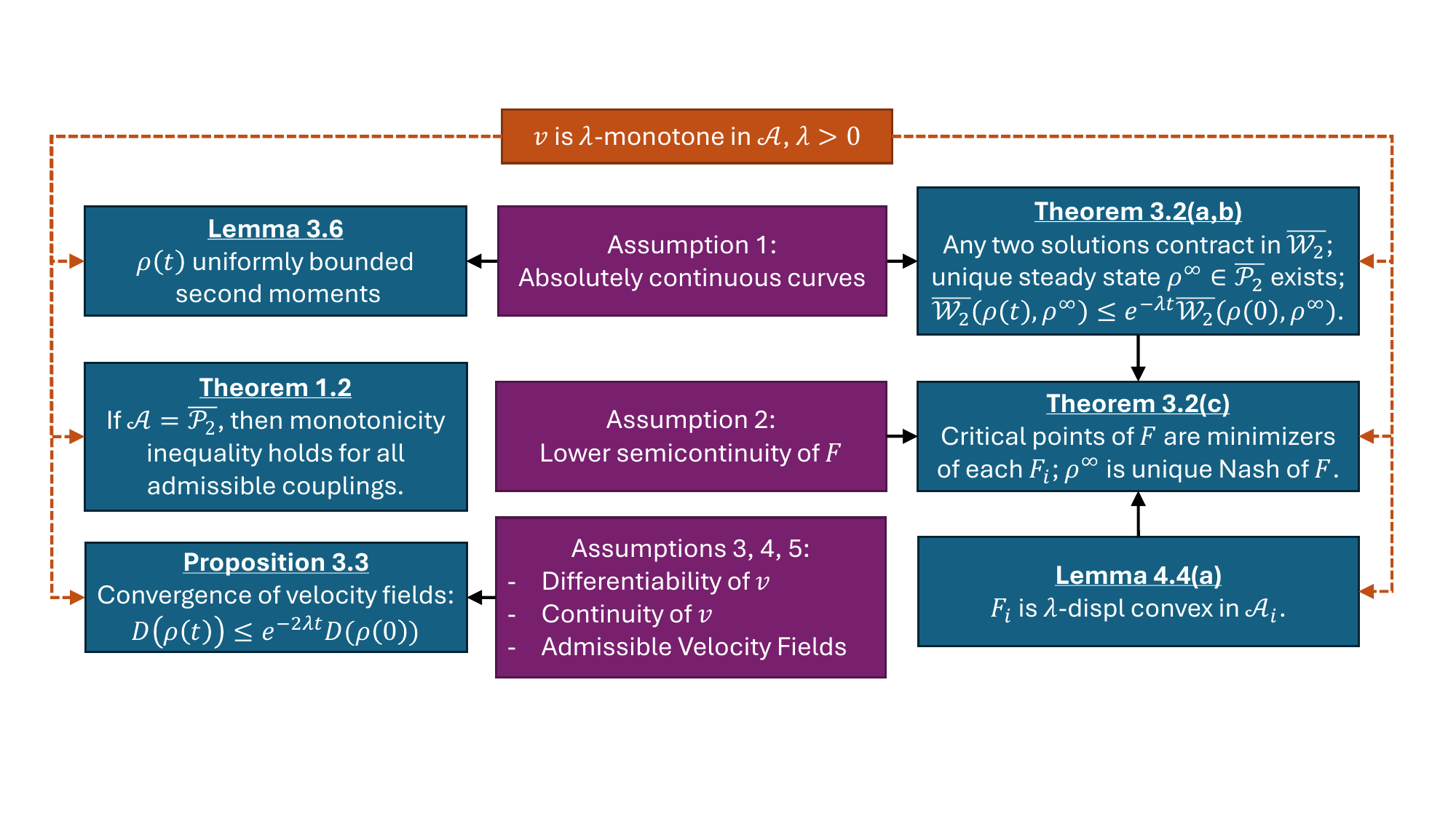}
    \caption{
    The relationships between main results (in blue), assumptions (in violet), and the monotonicity condition (in orange) are illustrated.
    Assuming that solutions to \eqref{eq:monotone_flow_dynamics} are absolutely continuous curves (\cref{assump:loc_abs_cont_curves}) allows us to compute the time derivative of the Wasserstein-2 metric via \cite[Thm 3.14]{cavagnari_dissipative_2023} to show contraction. Then via \cite[Lemma 7.3]{carrillo_contractive_2007}, a unique steady state $\rho^\infty\in\Pbar$ exists (\cref{thm:cv_W2}a). Convergence of $\rho(t)$ to $\rho^\infty$ follows by similar arguments (\cref{thm:cv_W2}b). If the dynamics are a coupled gradient flow, that is, take the form \eqref{eq:coupled_gradient_flow_dynamics}, then using \cref{assump:loc_abs_cont_curves}, \cref{assump:F_lsc} and displacement convexity of each $F_i$ (\cref{lem:monotonicity_convexity}a), we show that $\rho^\infty$ is critical point and the unique Nash equilibrium of $F$ (\cref{thm:cv_W2}c).
     The convergence of the velocity fields (\cref{lem:convergence_D_v}) and second moment bounds (\cref{lem:bdd_second_moment}) for solutions to \eqref{eq:monotone_flow_dynamics} follow from monotonicity. 
    }
    \label{fig:assumptions}
\end{figure}
To prove properties about the steady state and the Nash equilibrium in the coupled gradient flow setting, we utilize some additional lemmas.

\begin{lemma}[Uniqueness of Critical Points]\label{lem:uniqueness_NE}
    There is no more than one critical point $\rho^*\in \A$ for the game specified by $F$ under the monotonicity condition (\cref{def:monotonicity}) when $\lambda > 0$.
 \end{lemma}
\begin{proof}
    Consider $\rhozero,\rhoone \in \A$ which are both critical points. The Euler-Lagrange condition for each $F_i$ requires that $\delta_{\rho_i} F_i[\rhozero](x_i)=b_i$ for $\rhozero_i$-a.e. $x_i$, and $\delta_{\rho_i}F_i[\rhoone](x_i) = c_i$ for  $\rhoone_i$-a.e. $x_i$ for some constants $b_i, c_i \in \R$. 
    Therefore $\nabla_{x_i}\delta_{\rho_i} F_i[\rhozero](x_i)=0$ for $\rhozero_i$-a.e. $x_i$ and $\nabla_{x_i}\delta_{\rho_i}F_i[\rhoone](x_i) = 0$ for  $\rhoone_i$-a.e. $x_i$. Using these values in the monotonicity definition, we have
    \begin{align*}
        &\int \<x -y, \ndr F[\rhozero](x)-\ndr F[\rhoone](y)>\d \gamma(x,y) 
        = \int \< x-y, 0>\d \gamma(x,y) =0 \ge \lambda \Wbar(\rhozero,\rhoone)^2    \,.
    \end{align*}
    Hence, $\Wbar(\rhozero,\rhoone)=0$, and there is at most one critical point in $\A$. 
\end{proof}

\begin{lemma}[Steady State is Unique Critical Point]\label{lem:ss_critical_point}
    Under $\lambda$-monotonicity of $F$ in $\A$, a steady state of \eqref{eq:coupled_gradient_flow_dynamics} $\rho^\infty \in\A\cap S$ is the unique critical point of $F$ in $\A$.
\end{lemma}
\begin{proof} 
     The steady state $\rho^\infty\in S$ satisfies 
     \begin{align*}
         \int \nabla \phi \cdot \nabla_{x_i} \delta_{\rho_i} F_i[\rho^\infty](x_i)\d \rho_i^\infty(x_i) = 0 \quad \forall\, \phi \in C_c^\infty(\R^{d_i})\,,
     \end{align*}
     implying $\nabla_{x_i} \delta_{\rho_i} F_i[\rho^\infty](x_i)=0$ for $\rho^\infty_i$ a.e. $x_i$, and so $\delta_{\rho_i} F_i[\rho^\infty](x_i)=c_i$ 
     for some constant $c_i\in\R$, $\rho_i^\infty$ a.e. $x_i$; therefore $\rho^\infty$ is a critical point of each $F_i$.
    From \cref{lem:uniqueness_NE}, there is at most one critical point in $\A$, so $\rho^\infty\in\A$ is the unique critical point in $\A$.
\end{proof}
\begin{lemma}[Steady State is Nash Equilibrium]\label{lem:ss_Nash}
    Under $\lambda$-monotonicity in $\A$ and \cref{assump:F_lsc}, a steady state $\rho^\infty \in \A \cap S$ of \eqref{eq:coupled_gradient_flow_dynamics} is the unique Nash equilibrium of $F$ in $\A$, according to \cref{def:Nash_eq}.
\end{lemma}
\begin{proof}
    From \cref{lem:ss_critical_point}, the steady state $\rho^\infty$ is the unique critical point of $F$. It remains to show that the critical point is also the Nash equilibrium, that is, that $\rho^\infty$ satisfies
    \begin{align*}
        F_i[\rho_i^\infty,\rhoni^\infty] \le F_i[\rho_i,\rhoni^\infty]\quad \forall\ \rho_i\in \A \,,\quad \forall\ \iin\,.
    \end{align*}
    By \cref{lem:monotonicity_convexity}(a),
    $F_i[\cdot,\rhoni]$ is displacement convex for all $\rhoni\in\A$. By \cref{assump:F_lsc}, $F_i[\cdot,\rhoni]$ is lower semi-continuous in the weak topology of $(\P_2,\Wass_2)$. From \cite[Lemma 2.4.8]{ags}, a unique minimizer of $F_i[\cdot,\rhoni]$ exists in $\P_2$, and because the critical point of the game $F$ is unique in $\A \subseteq \Pbar$, it must be the minimizer. Applying this for all $\iin$, the Nash condition holds and therefore the steady state is a Nash equilibrium. 
    \end{proof}
\begin{proof}[Proof of \cref{thm:cv_W2}]
Recall the definition of the set $S$
\begin{align*}
    S = \left\{ \rho \in \Pbar \ : \ \int_\Omega \left\| v[\rho](x)\right\|^2 \d \rho(x) < \infty \quad \forall \ \Omega\subset \R^d\,,\ \Omega\ \text{compact} \right\}\,.
\end{align*}
        Proof of (a): The proof follows from showing that the dynamics are a contraction in the joint Wasserstein-2 metric. 
        Let $\rho(t)$ and $\eta(t)$ be two solutions to \eqref{eq:monotone_flow_dynamics}. 
               We have from \cref{assump:loc_abs_cont_curves} that $\rho(t),\eta(t)\in\A$ are absolutely continuous curves.
        Hence, the time derivative of the joint Wasserstein-2 metric between the $\rho$ and $\eta$ is
    \begin{align*}
        \ddt \Wbar(\rho(t),\eta(t))^2 &= 2 \int \<x-y,v[\eta(t)](x)-v[\rho(t)](y)>\d \gamma(t,x,y)\,,
    \end{align*}
    by \cite[Theorem 3.14]{cavagnari_dissipative_2023}, where $\gamma_i(t)\in\Gamma^*(\rho_i(t),\eta_i(t))$. 
    The right-hand side of the inequality is exactly the negative of the left-hand side of the monotonicity definition, so by $\lambda$-monotonicity in $\A$,
    \begin{align*}
        \ddt \Wbar(\rho(t),\eta(t))^2 \le -2\lambda \Wbar(\rho(t),\eta(t))^2\,.
    \end{align*}
    By Gr\"onwall's inequality, $\Wbar(\rho(t),\eta(t)) \le e^{-\lambda t} \Wbar(\rho(0),\eta(0))$, and therefore the dynamics are contractive in the joint Wasserstein-2 metric. From \cite[Proposition 7.1.5]{ags}, $\P_2(\R^d)$ endowed with the Wasserstein-2 metric is complete because $\R^d$ is complete, and since $\rho(t),\eta(t)$ are absolutely continuous, we can apply
     \cite[Lemma 7.3]{carrillo_contractive_2007} to obtain the existence of a unique stationary point $\rho^\infty\in\Pbar$. 
    If $\A$ is compact and the solution stays in $\A$ for all time, the long-time limit $\rho^\infty$ is also in $\A$.

    \sloppy Proof of (b): This step mirrors the convergence shown for $V_1(x)=\frac{1}{2}\norm{x-x^*}^2$ in the Euclidean setting of monotonicity. We use that $\ddt \Wbar(\rho(t),\rho^\infty)^2 = \sum_{i=1}^n \ddt \Wass_2(\rho_i(t),\rho_i^\infty)^2$ and let $\gamma_i(t)\in\Gamma^*(\rho_i(t),\rho_i^\infty)$. Using \cite[Theorem 3.11]{cavagnari_dissipative_2023} for the time derivative of the Wasserstein-2 metric, we have
     \begin{align*}
          \sum_{i=1}^n \ddt \Wass_2(\rho_i(t),\rho_i^\infty)^2 &= 2 \sum_{i=1}^n \int \<x_i - y_i,v_i[\rho(t)](x_i)> \d \gamma_i(t,x_i,y_i)
     \end{align*}
      Since $v_i[\rho^\infty](y_i)=0$ for $\rho_i^\infty$-a.e. $y_i$,  
     \begin{align*}
        \left| \int \<x_i,v_i[\rho^\infty](y_i)>\d\gamma_i(t,x_i,y_i)\right| &\le \sqrt{\int \norm{x_i}^2 \d \rho_i(t,x_i)\int \norm{v_i[\rho^\infty](y_i)}^2 \d\rho_i^\infty(y_i)} = 0\,, 
     \end{align*}
     and so we can add this term to the time derivative:
     \begin{align*}
         \sum_{i=1}^n\ddt \Wass_2(\rho_i(t),\rho_i^\infty)^2 = 2\sum_{i=1}^n\int \<x_i - y_i,v_i[\rho(t)](x_i)-v_i[\rho^\infty](y_i)> \d \gamma(t,x,y)\,.
     \end{align*}
    Applying the monotonicity condition with the optimal plan, using that $\rho^\infty\in\A$, gives
     \begin{align*}
         \sum_{i=1}^n \ddt \Wass_2(\rho_i(t),\rho_i^\infty)^2 &\le -2 \lambda \sum_{i=1}^n \Wass_2(\rho_i(t),\rho_i^\infty)^2\,,
     \end{align*}
     which results in $\ddt \Wbar(\rho(t),\rho^\infty)^2 \le -2\lambda \Wbar(\rho(t),\rho^\infty)^2$. Exponential convergence follows from Gr\"onwall's inequality.

     Part (c) follows from \cref{lem:ss_Nash}.
\end{proof}

\section{Coupled Gradient Flow Setting}\label{sec:convexity}
In this section, we present and discuss results for the setting where the dynamics are a coupled gradient flow. 
Monotonicity can be viewed as a \emph{extension of displacement convexity to the multispecies setting}, where instead of having a condition on a single energy functional, the condition holds for the set of energies $F=[F_1,\dots,F_n]$. The notion of monotonicity also generalizes the results in \cite{conger_coupled_2024}, which proves exponential convergence for a two-species min-max flow with $F_1=-F_2$. For a displacement convex-concave game $F=[F_1,F_2]$ with $F_1=-F_2$, $F$ is an example of a $\lambda$-monotone set of energies.
Recall the dynamics \eqref{eq:coupled_gradient_flow_dynamics}, where each species $i$ evolves in the direction of steepest descent for $F_i[\cdot,\rhoni]$ in the $\Wass_2$ metric:
\begin{align*}
    \partial_t \rho_i(t) = -\nabla_{\Wass_2,\rho_i} F_i[\rho(t)] = \divs{x_i}{\rho_i \nabla_{x_i}\delta_{\rho_i}F_i[\rho(t)]}\,,\quad \forall\, \iin\,.
\end{align*}
Recall also the definition of monotonicity for coupled gradient flows, \cref{def:monotonicity}.
    Let $\A$ be a geodesically convex set. If for any $\rhozero,\rhoone\in\A$, it holds that
\begin{align*}
          \lambda \Wbar(\rhozero,\rhoone)^2 \le \int \<x-y,\ndr F[\rhozero](x)-\ndr F[\rhoone](y)> \d \gamma(x,y) \,,
    \end{align*}
     for $\gamma\in\Gamma^*(\rhozero,\rhoone)$, then the set of \emph{energy functionals} $\{F_i:\A \to \R\}_{i=1}^n$ is \emph{$\lambda$-monotone} in $\A$.
In this section, we show how monotonicity implies convexity but convexity is does not necessarily imply monotonicity, because monotonicity accounts for dependencies among species. Like convexity, monotonicity is additive, and we show how to relate lower-bounds on Hessians of cross-interaction kernels to the monotonicity coefficient $\lambda$.  The results are summarized in the following properties:
\begin{enumerate}
    \item[(C1)] Monotonicity implies convexity: if $F$ is $\lambda$-monotone in $\A$, then each $F_i[\cdot,\rhoni]$ is $\lambda$-displacement convex in $\A_i$ for all $\rhoni\in\A_{-i}$. See \cref{lem:monotonicity_convexity}, \cref{ex:convexity_neq_monotonicity} (\cref{subsec:monotonicity_implies_convexity}).
    \item[(C2)] Additivity of monotonicity: if $F$ is $\lambda$-monotone in $\A$ and $E$ is $\lambda$-monotone in $\B$, then $F+E$ is $\lambda$-monotone in $\A \cap \B$. See  \cref{lem:additivity}, \cref{ex:monotoncity_nonlinear_diffusion} (\cref{subsec:additivity}).
    \item[(C3)] Monotonicity of cross-interaction kernels: for $\mathcal{W}_{ij}$ defined as $\mathcal{W}_{ij}=\int W_{ij}(x_i,x_j)\d\rho_i(x_i)\d\rho_j(x_j)$, $F_i=\sum_{j\ne i} \mathcal{W}_{ij}$, let $\alpha_{ij}$ be the Lipschitz coefficient of $\nabla_{x_i} W_{ij}(x_i,\cdot)$  and $\sum_{i\ne j}\nabla^2_{x_i} W_{ij}(x_i,x_j)\succeq c_i \Id_{d_i}$. Define 
    \begin{align*}
        \Lambda := \begin{bmatrix}
            c_1 & -\alpha_{12} & \dots & -\alpha_{1n} \\
            -\alpha_{21} & c_2 & & -\alpha_{2n} \\
            \vdots &  & \ddots &  \vdots \\
            -\alpha_{n1} & -\alpha_{n2} & \dots & c_n
        \end{bmatrix}\,.
    \end{align*}Then
    $F$ is $\lambda$-monotone in $\A$, where $\frac{1}{2}(\Lambda + \Lambda^\top) \succeq \lambda \Id_{d}$. See
 \cref{lem:interaction_kernels}, \cref{ex:interaction_kernels} (\cref{subsec:monotonicity_cross_interaction}).
\end{enumerate}

To prove these properties, we use an equivalent, second-order notion of monotonicity. The second-order definition parallels the finite-dimension notion; we first present the second-order finite-dimension definition with the infinite-dimensional version following.
In a finite-dimensional game in $\R^d$, each player selects a deterministic vector $x_i\in \R^{d_i}$ minimizing a cost function $f_i:\R^d\to \R$, which depends on the selections of all other players. When each player does gradient descent on its cost function, the dynamics are $\dot x_i = -\nabla_{x_i} f_i(x)$. The definition of monotonicity for a finite-dimensional game is given in terms of the vector of gradients.
\begin{definition}[Finite-Dimensional Monotonicity]\label{def:finite_dim_monotone}
    Let $f=[f_1,\dots,f_n]:\R^d \to \R^n$ where $f_i:\R^d \to \R$, and define $\overline \nabla f(x) \coloneqq [ \nabla_{x_1} f_1(x),\dots, \nabla_{x_n} f_n(x)]\in\R^d$ and $\overline \nabla^2 f(x) \in\R^{d \times d}$ such that $(\overline \nabla^2 f(x))_{ij}=\nabla_{x_j x_i}^2 f_i(x) \in \R^{d_i \times d_j}$.
    A game $f$ is monotone if it satisfies
    \begin{align*}
        &\text{\emph{(first order)}} \qquad \<\overline \nabla f(x)-\overline \nabla f (x'),x-x'> \ge \lambda \norm{x-x'}^2\quad \forall\ x,x'\in\R^d\,, \text{ or} \\
        &\text{\emph{(second order)}}\qquad
        \overline \nabla^2 f(x) + \overline \nabla^2 f(x)^\top \succeq 2\lambda \Id_d \quad \forall\ x\in\R^d\,,
    \end{align*}
    for some $\lambda \in\R$. The game is strongly monotone if $\lambda>0$.
\end{definition}
Importantly, $\overline \nabla$ is \textit{not} a gradient; it is an entry-wise operator. In order to present the corresponding infinite-dimensional result, we define Hessian-like operators $\nndr$ and $\nnddr$.
\begin{definition}[Matrix Operators]\label{def:matrix-ops}
For sufficiently regular $F$, the operators $\nndr$ and $\nnddr$ are defined as follows, where $(\cdot)_{ij}$ denotes the $i^{th}$ row block and $j^{th}$ column block:
    \begin{align*}
    \left(\nndr F[\rho](x)\right)_{ij} &\coloneqq 
    \begin{cases}
        \nabla_{x_i}^2 \delta_{\rho_i} F_i[\rho](x_i) & \text{if } i=j \\
        0 & \text{if } i \ne j
    \end{cases} \quad \in \R^{d_i\times d_i}\,, \\
    \left(\nnddr F[\rho](x,\hat x)\right)_{ij} &\coloneqq \nabla^2_{x_i \hat x_j} \delta^2_{\rho_i,\rho_j} F_i[\rho](x_i,\hat x_j) \quad \in \R^{d_i \times d_j}\,.
\end{align*}
\end{definition}
\begin{proposition}[Reformulation of Monotonicity] \label{prop:monotonicity_implication}
    Let \cref{assump:regularity_v} hold.
    Then $\lambda$-monotonicity in $\A$ holds if and only if for any two $\rhozero,\rhoone\in \A$, 
    \begin{align*}
        &\int_0^1 \bigg( \iint  \< x-y,\nnddr F[\rhos](I^{(s)}(x,y),I^{(s)}(\hat x,\hat y)) \cdot (\hat x-\hat y)>\d\gamma(x,y)\d\gamma(\hat x,\hat y) \\
        &\qquad +\int \< x-y,\nndr F[\rhos](I^{(s)}(x,y)) \cdot (x-y)> \d \gamma(x,y)\bigg) \d s \ge \lambda \Wbar(\rhozero,\rhoone)^2 \,,
    \end{align*}
    where $\rhos$ is the displacement interpolant between $\rhozero$ and $\rhoone$, and $\gamma\in\Gamma^*(\rhozero,\rhoone)$.\\
   If additionally \cref{assump:continuity_v,assump:tangent_space} hold, then $F:\A\to\R^n$ is $\lambda$-monotone in $\A$ if and only if 
    \begin{align*}
        &\iint  \< w(x),\nnddr F[\rho](x,\hat x) \cdot w(\hat x )>\d\rho(x)\d\rho(\hat x) \\
        &\qquad +\int \< w(x),\nndr F[\rho](x) \cdot w(x)> \d \rho(x) \ge \lambda \int \norm{w(x)}^2 \d \rho(x) \,.
    \end{align*}
    for all $\rho\in\A$, $w_i\in Y_i(\rho_i,\A_i)$ for all $\iin$. 
\end{proposition}
For the proof, see \cref{sec:2nd-order-monotonicity}. 
The second reformulation of monotonicity above provides a characterization that depends on only a single element $\rho\in\A$ and a velocity field set $\{w_i\in Y_i(\rho_i,\A_i)\}_i$, instead of a pair $\rhozero,\rhoone\in\A$. This can therefore be viewed as a "local" notion of monotonicity; however, it does require sufficient regularity for the second order derivatives of the energies.

The key difference between convexity and monotonicity for both the Euclidean ($\R^d$) and measure $\P(\R^d)$ settings is that for monotonicity, the terms in the inner product are no longer gradients of a single object; they are player-specific gradients. The distance is the Euclidean distance $\norm{\cdot}$ or the $\Wbar$ distance, and the inner product is either the Euclidean inner product or the integral of the inner product weighted by a measure. For the second-order conditions in \cref{prop:monotonicity_implication}, similar to the first-order case, the gradients for monotonicity are dependent on all of the cost functions and in general cannot be written as the Hessian of a single cost function. Notably, there are two matrices for the $\P_2(\R^d)$ monotonicity setting whereas only one for the Euclidean monotone setting; this is due to the fact that the gradient of the Euclidean metric tensor is zero while the gradient of the Wasserstein-2 metric tensor is not. For more details and a summary of the relationship between monotonicity and convexity in finite and infinite dimensions, see \cref{subsec:monotonicity_convexity}.

In \cref{sec:example}, we present an example illustrating how monotonicity properties assist in satisfying the requirements for \cref{thm:cv_W2}, resulting in convergence of a coupled gradient flow system to a unique Nash equilibrium. In the next subsections, we present lemmas and examples for each of the properties (C1)-(C3).
\subsection{Example: Two Species Zero-Sum Game with Nonlinear Diffusion}\label{sec:example}
In this example, we consider the two-species setting with $F=[F_1,-F_1]$, where
    \begin{align*}
        F_1[\rho_1,\rho_2] &= \iint f(x_1,x_2)\d \rho_1(x_1)\d\rho_2(x_2) +\alpha_1 \int U_{m_1}(\rho_1) \d x_1  - \alpha_2 \int U_{m_2}(\rho_2) \d x_2 
    \end{align*}
    for some $\alpha_1,\alpha_2>0$, with $f\in C^2(\R^{d_1}\times\R^{d_2};\R)$ such that $[f,-f]$ satisfies finite-dimensional monotonicity (see \cref{def:finite_dim_monotone}) for some $\lambda>0$.    
    For a measure $\eta \in \P(\R^k)$, the nonlinear diffusion term $U_{m}(\eta):\P(\R^k) \to \{\R \cup +\infty\}$ is defined as 
    \begin{align}\label{eq:nonlinear_diffusion}
        U_{m}(\eta) = \begin{cases}
            \eta \log \eta & \text{for } m=1\,,\, \eta \ll \L^{k} \,, \\
            \frac{\eta^{m}}{m-1} & \text{for } m>1\,,\, \eta \ll \L^{k}\,,\\
            +\infty & \text{else } \,.
        \end{cases}
    \end{align}
    For $\rho\in \overline L^1 := L^1(\R^{d_1}) \times \dots \times L^1(\R^{d_n})$, the corresponding flow \eqref{eq:monotone_flow_dynamics} has velocity field $v_i[\rho]=-\nabla_{x_i} \delta_{\rho_i} F_i[\rho]
   =w_i[\rho] 
   - \alpha_i \nabla_{x_i}U_{m_i}'(\rho_i)$,
   where
    $$
     w[\rho]:=\left[-\int \nabla_{x_1} f(x_1,x_2)\d \rho_2(x_2),\int \nabla_{x_2} f(x_1,x_2) \d \rho_1(x_1)\right]\,.
    $$
    We will use (C1)-(C2) to show that the game $F$  is $\lambda$-monotone in $\overline L^1\cap\Pbar $ in the sense of \cref{def:monotonicity}. Let $\A_i=L^1(\R^{d_i}) \cap \P_2(\R^{d_i})$ and $\B_i= \P_2(\R^{d_i})$. Since $[f,-f]$ satisfies $\lambda$-monotonicity in finite dimensions (\cref{def:finite_dim_monotone}), 
    we have that $w$
    is $\lambda$-monotone in $\B=\Pbar$ by \cref{lem:monotonicity_finite_dimensions}, so 
    $$E:=\left[\iint  f(x_1,x_2) \d \rho_1(x_2) \d \rho_2(x_2),-\iint  f(x_1,x_2) \d \rho_1(x_1) \d \rho_2(x_2)\right]$$ is $\lambda$-monotone in $\B$.  
    Using that $\int U_{m_i}(\rho_i) \d x_i$ is $0$-convex in $L^1$ \cite[Section 3.1]{carrillo_kinetic_2003}, it follows from \cref{lem:monotonicity_convexity}(b) (C1, monotonicity vs convexity) that $\tilde E$ with $\tilde E_i:=\alpha_i \int U_{m_i}(\rho_i) \d x_i$
    is $0$-monotone in $\A$.
    Then applying \cref{lem:additivity} (C2, additivity), we conclude that $F=E+\tilde E$ is $\lambda$-monotone in $\A=\Pbar \cap \overline L^1$. 

    To illustrate that the theoretical results on second moment control as well as existence of and convergence to a unique Nash equilibrium apply to this example, we assume from now on that \cref{assump:loc_abs_cont_curves} holds. 
    First, we show that the second moments of solutions to \eqref{eq:coupled_gradient_flow_dynamics} defined by the game $F$ are uniformly bounded. 
   Assuming $\rho(t)\in \A=\overline L^1 \cap \Pbar$ for all $t\ge0$, \cref{lem:bdd_second_moment} results in a uniform bound on the second moments.

  Next, we establish that \cref{assump:F_lsc} holds. The  finite-dimensional monotonicity of $[f,-f]$ with $\lambda>0$ implies $f(\cdot,x_2)$ is strongly convex for any $x_2 \in \R^{d_2}$ and $f(x_1,\cdot)$ is strongly concave for any $x_1\in\R^{d_1}$; in combination with $f\in C^2$, we obtain $\int f(\cdot,x_2)\d\rho_2(x_2)\ge c_2$, $\int f(x_1,\cdot)\d\rho_1(x_1)\le c_1$ where $c_1,c_2$ depend on $\rho_1,\rho_2$ respectively. Together with continuity of these terms in $x_1$ and $x_2$ respectively, we have lower semicontinuity of $\iint f(x_1, x_2)\d\rho_1(x_1) \d\rho_2(x_2)$  with respect to $\rho_1$ for any fixed $\rho_2$ and upper semicontinuity  of $\iint f(x_1, x_2)\d\rho_1(x_1) \d\rho_2(x_2)$ 
    with respect to  $\rho_2$ for any fixed $\rho_1$, due to \cite[Proposition 7.1]{santambrogio_OTAM}. 
    For the diffusion term, if $m_i>1$, the $L^{m_i}$ norm is lower semicontinuous \cite[Proposition 7.7, Remark 7.8]{santambrogio_OTAM}. 
    For the case $m_i=1$, lower semicontinuity of the diffusion term follows from the uniform second moment bound and \cite[Lemma C.4]{conger_coupled_2024}. Therefore \cref{assump:F_lsc} holds.

  Using \cref{assump:loc_abs_cont_curves} and applying \cref{thm:cv_W2}(a), there exists a unique steady state in $\rho^\infty \in S$. In \cite{conger_coupled_2024} it was shown that $\rho^\infty\in \overline L^1 \cap \Pbar$ as long as $\alpha_1,\alpha_2>0$ in the case of linear diffusion, and the result for $m_i>1$ follows in the same way by considering the Euler-Lagrange condition
    \begin{align*}
 \frac{ \alpha_i m_i}{m_i - 1} \rho_i^{m_i-1}(x_i) + (-1)^{i+1}\int f(x_1,x_2) \d \rho_{-i}(x_{-i})=c_i \quad \text{ on } \supp\rho_i\,.
  \end{align*}
  Then, since $\rho^\infty\in \A \cap S$, 
 by \cref{thm:cv_W2}(b), we have exponential convergence in $\Wbar$ with rate $\lambda$. And moreover, thanks to \cref{assump:F_lsc},  we can apply \cref{thm:cv_W2}(c) to conclude that $\rho^\infty$ is the unique Nash equilibrium.

\subsection{Monotonicity Implies Convexity}\label{subsec:monotonicity_implies_convexity}
Here we state the lemma which describes property (C1). While monotonicity implies convexity, the converse is not true, as illustrated in the example in this section.
\begin{lemma}[Monotonicity and Convexity]\label{lem:monotonicity_convexity}
The following statements detail the relationship between monotonicity and convexity.
\begin{enumerate}
    \item[(a)] \emph{Monotonicity implies convexity.} Let $F:\A \to \R^n$ be $\lambda$-monotone in $\A$. Then $F_i[\cdot,\rhoni]$ is $\lambda$-displacement convex for all fixed $\rhoni\in\A_{-i}$, for all $\iin$.
    \item[(b)] \emph{Uncoupled convexity implies monotonicity.} Let $E_i:\A_i \to \R$ be a $\lambda_i$-displacement convex functional for all $\rho_i\in \A_i$, for all $\iin$. Then the functional $E=[E_1, \dots, E_n]:\A \to \R^n$ is $\lambda$-monotone in $\A$, where $\lambda\coloneqq\min_i \lambda_i$.
\end{enumerate}
\end{lemma}
\begin{proof}[Proof of \cref{lem:monotonicity_convexity}]
    We begin with the proof of $(a)$.     To prove that monotonicity implies convexity, select any $\pi\in\A$. Selecting a particular value of $i$, set $\rhozero_j=\rhoone_j=\pi_j$ for all $j\ne i$. Let $\rhozero_i=\pi_i$ and $\rhoone_i$ satisfy $\Wass_2(\rhozero_i,\rhoone_i)^2 >0$. First-order monotonicity results in 
    \begin{align*}
        \int \<x_i-y_i, \nabla_{x_i}\delta_{\rho_{i}}F_i[\rhozero](x_i) -\nabla_{x_i}\delta_{\rho_{i}}F_i[\rhoone](y_i)> \d \gamma_i(x_i,y_i) \ge \lambda \Wass_2(\rhozero_i,\rhoone_i)^2\,,
    \end{align*}
    which is a sufficient condition for $\lambda$-displacement convexity in $\A$. See \cite[Proposition 16.2]{Villani07} for details on convexity inequalities.
    
    Next we prove (b). Since each $E_i$ is $\lambda_i$-displacement convex for all $\rho_i\in \A_i$, for all $\iin$ it holds that
    \begin{align*}
        \int \<x_i-y_i , \nabla_{x_i} \delta_{\rho_i} E_i[\rhozero_i](x_i) - \nabla_{x_i} \delta_{\rho_i} E_i[\rhoone_i](y_i)> \d \gamma_i(x_i,y_i) \ge \lambda_i \Wass_2(\rhozero_i,\rhoone_i)^2 \,,
    \end{align*}
    for all $\rhozero,\rhoone \in\A$ and $\gamma_i\in\Gamma^*(\rhozero_i,\rhoone_i)$. Since there is no coupling among species, summing over all $i$ results in
    \begin{align*}
        \int \<\ndr E[\rhozero](x) - \ndr E[\rhoone](y),x-y > \d \gamma(x,y) \ge \lambda \Wbar(\rhozero,\rhoone)^2\,,
    \end{align*}
    where $\lambda=\min_i \lambda_i$. The above is the monotonicity condition.
\end{proof}

In the following example, we illustrate a setting in which $F_i[\cdot,\rhoni]$ is displacement convex for all $i$, but $F$ is not monotone due to the nature of the coupling between the two species, showing why monotonicity is a stronger condition than convexity.
\begin{example}[Convexity Does Not Imply Monotonicity]\label{ex:convexity_neq_monotonicity}
    Consider a system $F:\Pbar\to \R^2$ given by
    \begin{align*}
        F_1[\rho_1,\rho_2] &= \frac{\lambda}{2}\int \norm{x_1}^2\d \rho_1(x_1)+ \iint x_1^\top A_1 x_2 \d \rho_1(x_1) \d \rho_2(x_2)\,, \\
        F_2[\rho_1,\rho_2]&=\frac{\lambda}{2}\int \norm{x_2}^2 \d \rho_2(x_2)+ \iint x_1^\top A_2 x_2 \d \rho_1(x_1) \d \rho_2(x_2) \,,
    \end{align*}
    where $A_1,A_2\in \R^{d_1 \times d_2}$ and $\lambda>0$. The second derivative of $F_i$ along geodesics between $\rhozero,\rhoone\in\Pbar$ is
    \begin{align*}
        \ddss F_i[\rhos_i,\rhoni] = \lambda \int \norm{\vs_i(x_i)}^2 \d \rhos_i(x_i) = \lambda W_2(\rhozero_i,\rhoone_i)^2\,, \ \forall\, i\in\{1,2\}\,, 
    \end{align*}
   where $(\rhos_i,\vs_i)$ solve \eqref{eq:geodesic_eqns}, and therefore $F_1[\cdot,\rho_2]$ and $F_2[\rho_1,\cdot]$ are $\lambda$-displacement convex. The first-order monotonicity condition, where $A=\begin{bmatrix}
       0 & A_1 \\ A_2^\top & 0
   \end{bmatrix}$, is
   \begin{align*}
   &\int \<x-y, \ndr F[\rhozero](x)-\ndr F[\rhoone](y) > \d \gamma(x,y) \\
   &=
       \int \<x-y, \lambda (x-y) + \int A (\hat x-\hat y)\d \rhozero(\hat x) \d \rhoone(\hat y) > \d \gamma(x,y) \\
       &= \lambda \int \norm{x-y}^2 \d \gamma(x,y) + \E(x-y)^\top A \E(x-y)\,,
   \end{align*}
   where $\E(x-y)=\int (x-y) \d \rhozero(x)\d\rhoone(y)$. Then for choices of $A_1,A_2$ such that $A+A^\top$ has one or more negative eigenvalues, $F$ does not satisfy the monotonicity condition. 
    However, selecting $A_1=-A_2$ results in $\lambda$-monotonicity in $\Pbar$, since $A+A^\top=0$.
\end{example}

\subsection{Additivity of Monotonicity}\label{subsec:additivity}
Here we state the lemma which describes property (C2). This property means that the sum of two monotone systems is monotone, with the monotonicity coefficient at least the sum of the original coefficients.

\begin{lemma}[Additivity of Monotonicity]\label{lem:additivity}
        Let $E=[E_1,\dots,E_n]:\A\to\R^n$ be $\lambda$-mononotone in $\A$  and $F=[F_1,\dots,F_n]:\B\to\R^n$  be $\tilde \lambda$-monotone in $\B$, and define $\hat F=[\hat F_1, \dots,\hat F_n]:\A \cap\B \to \R^n$ where $\hat F_i = \alpha E_i + \tilde \alpha F_i$, for some $\alpha,\tilde \alpha \in\R$. Then $\hat F$ is $ \alpha\lambda+\tilde \alpha \tilde\lambda$-monotone in $\A\cap \B$.  
\end{lemma}
\begin{proof}
    The proof follows by the linearity of the monotonicity condition in terms of the energy functionals.
\end{proof}
The result from \cref{lem:monotonicity_convexity}(b) shows that displacement convexity directly implies monotonicity when there is no coupling among different species. \cref{lem:additivity} is useful for computing monotonicity of more intricate functionals because one can verify the monotonicity of only the inter-species terms to ensure overall monotonicity. For terms such as $\int \rho_i \log \rho_i$ which do not have any multispecies coupling and for which displacement convexity has been shown, one can instead check the monotonicity of $F$ without these terms, and apply the previous lemmas, which allow for the addition of displacement convex terms with dependence on only one species to be added to monotone $F$, and the resulting game is monotone with an updated parameter.

\begin{example}[Pairwise Zero-Sum Game with Nonlinear Diffusion: Application of \cref{lem:additivity}]\label{ex:monotoncity_nonlinear_diffusion}
        In this example, let $\A_i=L^{m_i}(\R^{d_i})\cap \P_2(\R^{d_i})$ with $m_i\ge 1$. Let $\hat F_i:\A\to\R$ take the form
    \begin{align*}
        \hat F_i[\rho]&= \sum_{j\ne i} \iint f_{ij}(x_i,x_j) \d \rho_i(x_i) \d \rho_j(x_j) + \int V_i(x_i) \d \rho_i(x_i)  + \alpha_i\int U_{m_i}(\rho_i) \d x_i 
    \end{align*}
    for $\alpha_i\ge 0$.
Monotonicity for $\hat F$ is not defined in $\Pbar$ because of the nonlinear diffusion terms, for which the left-hand side of \eqref{eq:monotonicity_condition_F} is not defined on Diracs, but we will show that $\hat F$ is monotone in $\A\subset\Pbar$. 
\sloppy Let $f_{ij}\in C^2(\R^{d_i}\times\R^{d_j},\R)$ satisfy $-f_{ij}(x_i,x_j)=f_{ji}(x_j,x_i)$, $\tilde\lambda_i\Id_{d_i} \preceq \sum_{j\neq i} \nabla_{x_i}^2 f_{ij}(x_i,x_j) $ for all $x_i\in\R^{d_i},x_j\in\R^{d_j}$, and $\int \sum_{j\neq i} \nabla_{x_i}^2 f_{ij}(x_i,x_j) \d \rhoni(x_{-i})\in \R^{d_i \times d_i}$ for all $\rhoni \in \A_{-i}$ and $x_i\in \R^{d_i}$. Let $V_i\in C^2 (\R^{d_i},\R)$ satisfy $\nabla_{x_i}^2 V_i(x_i) \succeq \lambda_i \Id_{d_i}$ for all $x_i\in \R^{d_i}$.
Consider the set of functionals
\begin{align*}
    F_i[\rho]&= \sum_{j\ne i} \iint f_{ij}(x_i,x_j) \d \rho_i(x_i) \d \rho_j(x_j) \,.
\end{align*}
To check the reformulated monotonicity condition in \cref{prop:monotonicity_implication}, we compute the Hessian terms introduced in \cref{def:matrix-ops}: 
\begin{align*}
  (\nndr F[\rho](x))_{ii} &= \sum_{j\ne i} \int \nabla_{x_i}^2 f_{ij}(x_i,x_j) \d \rho_j(x_j) = \int \sum_{j\ne i} \nabla_{x_i}^2 f_{ij}(x_i,x_j) \d \rhoni(x_{-i}) \succeq \tilde \lambda_i\Id_{d_i} \,,\\
    (\nnddr F[\rho](x,\hat x))_{ij} &= \begin{cases}
        0 & \text{if } i=j\,, \\
        \nabla_{ x_i \hat x_j}^2 f_{ij}( x_i,\hat x_j) & \text{if } i \neq j\,. 
    \end{cases}
\end{align*}
Using a change of notation and $-f_{ij}(x_i,x_j)=f_{ji}(x_j,x_i)$,
\begin{align*}
    &\int_0^1\iint \<x-y,\nnddr F[\rhos](I^{(s)}(x,y),I^{(s)}(\hat x,\hat y)) \cdot (\hat x - \hat y)> \d \gamma(x,y) \d \gamma(\hat x,\hat y) \d s  
    \\
    & =  
    \frac{1}{2} \int_0^1  \iint \big \langle x-y,\left[(\nnddr F[\rhos](I^{(s)}(x,y),I^{(s)}(\hat x,\hat y))
     + \nnddr F[\rho](I^{(s)}(\hat x,\hat y),I^{(s)}(x,y))^\top )\right]\\
    & \hspace{10cm}\cdot (\hat x-\hat y) \big \rangle \d \gamma(x,y) \d \gamma(\hat x,\hat y) \d s \\&  = 0\,.
\end{align*}
The matrices $J_1$ and $J_2$, given by 
\begin{align*}
    (J_1[\rho](x))_{ii}= -(\nndr  F[\rho](x))_{ii} &=  - \int \sum_{j\ne i} \nabla_{x_i}^2 f_{ij}(x_i,x_j) \d \rho_{-i}(x_{-i}) \,, \\
    (J_2[\rho](x,\hat x))_{ij} =-(\nnddr  F[\rho](x,\hat x))_{ij}&= \begin{cases}
        0 & \text{if } i=j\,, \\
        -\nabla_{ x_i \hat x_j}^2 f_{ij}( x_i,\hat x_j) & \text{if } i \neq j\,,
    \end{cases}
\end{align*} 
are well-defined $\rho$-a.e. due to $f_{ij}\in C^2$ and $ \left|\int \sum_{j\ne i} \nabla_{x_i}^2 f_{ij}(x_i,x_j) \d \rho_{-i}(x_{-i})\right|<\infty$, and so \cref{assump:regularity_v} holds. 
 Applying \cref{prop:monotonicity_implication} 
 we conclude that $F$ is $\tilde \lambda $-monotone in $\Pbar$ with $\tilde \lambda :=\min_{i}\tilde \lambda_i $. Further, let $\lambda=\min_{i}\lambda_i$. We can apply \cref{lem:additivity} for each player, using that $\int U_{m_i}(\rho_i) \d x_i$ is $0$-displacement convex \cite{carrillo_kinetic_2003} over $L^{m_i}(\R^{d_i})$ and $\int V_i(x_i) \d \rho_i(x_i)$ is $\lambda_i$-displacement convex because $\nabla_{x_i}^2 V_i\succeq \lambda_i\Id_{d_i}$, to obtain that $\hat F$ is $(\lambda+\tilde \lambda)$-monotone in $\A$, where $\lambda=\min_i \lambda_i$.
Hence, \cref{lem:additivity} provides a method for using knowledge on the displacement-convexity of additive functionals to show monotonicity\footnote{A potentially larger monotonicity coefficient $\min_i(\lambda_i +\tilde \lambda_i)$ can be obtained by including $V_i$ in $F_i$ instead of considering it separately via \cref{lem:additivity}.}.

\end{example}

\subsection{Monotonicity of Cross-Interaction Kernels}\label{subsec:monotonicity_cross_interaction}
Next, we investigate monotonicity for a specific type of commonly-used functionals. In the PDE literature for biological models, interaction kernels are commonly employed to capture phenomena of interacting species. Our monotonicity condition provides convergence guarantees for systems driven by interaction kernels that have lower-bounded Hessians. To illustrate this, we draw a connection to the finite-dimensional contractivity literature; \cite[Theorem 41]{Davydov_Bullo_Contraction} proves that when the convexity of cost functions dominate inter-species effects, the overall system is contractive. We extend this results the infinite-dimensional setting in the following lemma.
\begin{lemma}[Cross-Interaction Kernels]\label{lem:interaction_kernels}
 Consider the interaction term $\W_{ij}:\A_{i} \times \A_{j} \to \R$, parameterized by the interaction kernel $W_{ij}\in C^2(\R^{d_{i}}\times \R^{d_{j}},\R)$, given by
    \begin{align*}
        \W_{ij}(\rho_{i},\rho_{j}) = \iint W_{ij}(x_{i}, x_{j}) \d\rho_{i}(x_{i}) \d\rho_{j}(x_{j})\,.
    \end{align*}
    Let $F_i=\sum_{j\ne i} \W_{ij}$ for all $\iin$. If for all $i\ne j$, there exist $c_i\in \R$ and $\alpha_{ij}\ge0$ such that for all $x_j,y_j\in\R^{d_j}\,,\, x_i\in\R^{d_i}$,
    \begin{align*}
        \norm{\nabla_{x_i}W_{ij}(x_i,x_j) - \nabla_{x_i}W_{ij}(x_i,y_j)} \le \alpha_{ij} \norm{x_j-y_j}\,, \quad  
        \sum_{j\ne i} \nabla_{x_i}^2 W_{ij}(x_i,x_j) \succeq c_i\Id_{d_i}\,,
    \end{align*}
   and the matrix 
    \begin{align*}
        \Lambda := \begin{bmatrix}
            c_1 & -\alpha_{12} & \dots & -\alpha_{1n} \\
            -\alpha_{21} & c_2 & & -\alpha_{2n} \\
            \vdots &  & \ddots &  \vdots \\
            -\alpha_{n1} & -\alpha_{n2} & \dots & c_n
        \end{bmatrix}\,,
    \end{align*}
    satisfies $\frac{1}{2}(\Lambda + \Lambda^\top) \succeq \lambda \Id_{d}$, then $F$ is $\lambda$-monotone in $\A$.
\end{lemma}
\begin{proof}
     The monotonicity condition to check is: for all $\rhozero, \rhoone\in \A$ and for all $\gamma\in\Gamma^*(\rhozero,\rhoone)$,
    \begin{align*}
        \sum_{i,j\ne i}\int \<\int \nabla_{x_i}W_{ij}(x_i,x_j)\d\rhozero_j(x_j) - \int \nabla_{x_i}W_{ij}(y_i,y_j)\d \rhoone_j(y_j),x_i-y_i> \d \gamma_i(x_i,y_i) 
        \ge \lambda \int \norm{x-y}^2 \d\gamma(x,y)\,.
    \end{align*}
   We use the optimal plans to write the inner product as 
    \begin{align*}
        \sum_{i,j\ne i}\int \<\int \nabla_{x_i}W_{ij}(x_i,x_j)\d\rhozero_j(x_j) - \int \nabla_{x_i}W_{ij}(y_i,y_j)\d \rhoone_j(y_j),x_i-y_i> \d \gamma_i(x_i,y_i) \\
        = \int \sum_{i,j\ne i} \< \nabla_{x_i}W_{ij}(x_i,x_j)- \nabla_{x_i}W_{ij}(y_i,y_j),x_i-y_i> \d \gamma(x,y)\,,
    \end{align*}
  and apply the following exact Taylor expansion
    \begin{align*}
        \nabla_{x_i} W_{ij}(y_i,y_j) = \nabla_{x_i} W_{ij}(x_i,y_j) + (y_i-x_i)^\top \nabla_{x_i}^2 W_{ij}(\tilde x_i,y_j)
    \end{align*}
    for some $\tilde x_i\in\R^{d_i}$. 
    Then the integrand satisfies
    \begin{align*}
        \sum_{i,j\ne i} \< \nabla_{x_i}W_{ij}(x_i,x_j)- \nabla_{x_i}W_{ij}(y_i,y_j),x_i-y_i>  &=  \sum_i (x_i-y_i)^\top \bigg(\sum_{j\ne i}\nabla^2_{ij} W_{ij}(\tilde x_i,y_j)\bigg)  (x_i-y_i)\\
      & \quad
      + \sum_{i,j\ne i}\< \nabla_{x_i}W_{ij}(x_i,x_j)- \nabla_{x_i} W_{ij}(x_i,y_j),x_i-y_i> \,.
    \end{align*}
    Applying $\sum_{j\ne i} \nabla_{x_i}^2 W_{ij}(\tilde x_i,y_j) \succeq c_i\Id_{d_i}$ for the first term and  the Cauchy-Schwarz inequality plus the Lipschitz condition for the second term, we have
    \begin{align}\label{eq:cross_interaction_ineq}
    \begin{split}
        &\sum_{i,j\ne i}  \< \nabla_{x_i}W_{ij}(x_i,x_j)- \nabla_{x_i}W_{ij}(y_i,y_j),x_i-y_i> \\
        &\qquad \ge \sum_i  c_i \norm{x_i-y_i}^2 -\sum_{i,j\ne i} \alpha_{ij} \norm{x_j-y_j} \norm{x_i-y_i} =z^\top \Lambda z = \frac{1}{2}z^\top (\Lambda+\Lambda^\top) z
       \,,
    \end{split}
    \end{align}
    where $z=z(x,y)\in\R^n$, $z_i:=\norm{x_i-y_i}$, and $\Lambda$ is defined in the lemma statement.  Plugging \eqref{eq:cross_interaction_ineq} into the monotonicity condition, we conclude
    \begin{align*}
        \int \sum_{i,j\ne i} \< \nabla_{x_i}W_{ij}(x_i,x_j)- \nabla_{x_i}W_{ij}(y_i,y_j),x_i-y_i> \d \gamma(x,y) 
        &\ge \frac{1}{2} \int z(x,y)^\top (\Lambda+\Lambda^\top) z(x,y) \d \gamma(x,y)\\
        &\ge \lambda \int \norm{x-y}^2 \d \gamma(x,y)\,.
    \end{align*}
\end{proof}
\begin{remark}
    \sloppy The previous result can be adjusted to include nonlocal self-interaction terms $\iint W_i(x_i,\tilde x_i) \d \rho_i(x_i) \d \rho_i(\tilde x_i)$, using existing displacement convexity properties and the additivity property (\cref{lem:additivity}). Moreover, if $-\Lambda$ is Hurwitz (i.e. all its eigenvalues have strictly negative real parts), even if $\Lambda + \Lambda^\top \nsucceq 0$, we expect contraction in a reweighted joint $\Wass_2$ metric following the approach in \cite[Thm 41]{Davydov_Bullo_Contraction}. 
\end{remark}\cref{lem:interaction_kernels} shows that if the interaction kernels are strongly convex with large convexity coefficients $c_i$ and small Lipschitz constants $\alpha_{ij}$, which encodes how strongly the other species influence species $i$, then the overall system is monotone.
For more general conditions that ensure monotonicity of $F$, see \cref{lem:Cross-Interaction-Kernels_appendix}.

\begin{example}\label{ex:interaction_kernels}
    Examples of kernels with lower-bounded Hessians include the quadratic and Morse-like attractive-repulsive kernels, which are radially symmetric. In the setting where $\R^{d_i}=\R^{d_j}$, let $x=x_i-x_j$ and $W_{ij}(x_i,x_j)=W(x_i-x_j)=W(x)$ for any $i,j$; some commonly-used kernels are
\begin{align*}
    W^{(k)}(x) &= \frac{\norm{x}^k}{k}\,, \quad  W^{(p)}(x) = \frac{\norm{x}^a}{a}-\frac{\norm{x}^b}{b}\,, \\ 
    W^{(m)}(x) &= C_r \exp\left(-\frac{\norm{x}^2}{l_r}\right) - C_a \exp\left(-\frac{\norm{x}^2}{l_a}\right)\,,
\end{align*}
where $k>1-d_i$, $l_a>l_r>0$, $C_r>C_a>0$ and $a>b>\max (2-d_i,2)$ with $a,b$ even integers so that the second derivative is well-defined (see \cite{balague_confinement_2012,craig_convergence_2016} and references therein). The kernel $W^{(k)}$ is purely attractive for any power $k$, and satisfies the Lipschitz condition in \cref{lem:interaction_kernels} for $k=2$.
The kernel $W^{(m)}$ is a Morse-like attractive-repulsive potential which is repulsive for small $x$ and attractive for large $x$; this kernel also satisfies the Lipschitz condition. The power law potential $W^{(p)}$ is attractive-repulsive; as it does not satisfy the Lipschitz condition in \cref{lem:interaction_kernels}, monotonicity is an open question, and may be tackled making use of the local Lipschitzness of this kernel.
See the computations in \cref{sec:interaction_kernel} for details on the Hessian lower bounds of these kernels, and see \cite{carrillo_measure_2020} for more details and references regarding these particular kernels. When the relevant conditions are satisfied, \cref{lem:interaction_kernels} can be applied to coupled PDEs such as
\begin{align*}
    \partial_t \rho_1 &= \div{\rho_1 \nabla (H_1\ast \rho_1 + K_1 \ast \rho_2} \\
     \partial_t \rho_2 &= \div{\rho_2 \nabla (H_2\ast \rho_2 + K_2 \ast \rho_1}\,,
\end{align*}
for which existence of solutions was studied in \cite{francesco_measure_2013,evers_equilibria_2017,carrillo_measure_2020}, motivated by applications in biological systems.
\end{example}

\section{Applications}\label{sec:applications}
In this section, we illustrate how our model can be applied to the service provider - population game, drawing on models from \cite{shekhtman_strategic_2024}. Then we draw connections between a $n$-player monotone game and a particular set of mean field games.  Finally, we show a numeric example of a two-population game which is not monotone, but almost monotone, and we use diffusion to ensure sufficient probability mass over the monotone region to still observe convergence. Numerical simulations for these examples are implemented using the two-dimensional finite volume method analyzed in \cite{carrillo_finite-volume_2015}.

\subsection{Multi-Learner System}
In this example, we illustrate monotonicity results on a multi-learner setting with strategic users. In this setting, users in population $i\in[1,\dots,N]$ choose to interact with service provider $j\in[1,\dots,M]$, for a total of $n=N+M$ species in the system. Both the user populations and service providers update their parameters over time, evolving according to a gradient flow with respect to their own state. The user populations are denoted as in the rest of the paper by $\rho_i \in \P(\R^{d_\rho})$, and the service providers are assumed to be Diracs on the space of measures, that is, $\rho_{N+j}\coloneqq \delta_{h_j}$, with $h_j\in \R^{d_h}$ for $j\in[1,\dots,M]$. The user populations have utility functions associated with each service provider, denoted $u_{ij}:\R^{d_\rho} \times \R^{d_h} \to \R$. The proportion of population $i$ which purchases from service provider $j$ is $a_{ij}\in[0,1]$, such that $\sum_{j=1}^M a_{ij} = 1$ for all $i\in[1,\dots,N]$. The service provider $j$ has a loss function for interacting with population $i$ given by 
 $l_{ij}: \R^{d_\rho} \times \R^{d_h} \to \R$. Additionally, each population nonlocally interacts with the other populations. The resulting energy functionals are, for all $i\in[1,\dots,n]$ and $j\in[1,\dots,M]$, 
 \begin{align*}
    \text{user populations:}\quad  &F_i[\rho,h] = -\sum_{j=1}^M \int a_{ij} u_{ij}(x_i,h_j) \d \rho_i(x_i) + \sum_{k=1}^N \int \rho_i W_{ik}\ast \rho_k + \alpha_i KL( \rho_i|\rhot_i)  \\
    \text{service providers:}\quad  &F_{N+j}[\rho,h] = \sum_{i=1}^N \int a_{ij} l_{ij}(x_i,h_j) \d \rho_i(x_i) + \beta_j \norm{h_j}^2  \,,
 \end{align*}
 where 
 $\rhot_i$ are given log-concave reference measures and 
 $\alpha_i,\beta_j>0$.
The allocation weights $a_{ij}$ update over time; populations will seek services from providers with the best utility. We select the weight update from \cite{shekhtman_strategic_2024}, given by
\begin{align*}
     a_{ij}(t) \coloneqq \frac{\exp[\eta \int u_{ij}(x_i,h_j(t))\d\rho_i(t,x_i)]}{\sum_{m=1}^M \exp[\eta \int u_{im}(x_i,h_m(t) )\d\rho_i(t,x_i)]  }  \,,
\end{align*}
where $\eta\in\R_+$ parameterizes the weight update steepness.
The system evolves according to \eqref{eq:coupled_gradient_flow_dynamics}.
The system is monotone depending on the specific choice of $u_{ij}$, $l_{ij}$, $W_{ik}$, $\alpha_i$, and $\beta_j$. In the setting where the finite-dimensional game 
\begin{align*}
 -\bigg[\sum_{j=1}^M &\int a_{1j} u_{1j}(x_1,h_j) ,\dots,\sum_{j=1}^M \int a_{Nj} u_{Nj}(x_N,h_j) , 
 \sum_{i=1}^N \int a_{i1} l_{i1}(x_i,h_1) ,\dots, \sum_{i=1}^N \int a_{iM} l_{iM}(x_i,h_M)  \bigg]  
\end{align*}
is $\lambda$-monotone over Euclidean space, with $\lambda \ge 0$, and $W_{ik}=0$
with $\alpha_i,\beta_j>0$, the infinite-dimensional game $F$ is monotone over $L^1(\tilde \rho_i)^{\otimes n}$ by \cref{lem:monotonicity_finite_dimensions}, 
\cref{lem:monotonicity_convexity}(b) and \cref{lem:additivity}. However, if the kernel $W_{ik}$ is, for example, a Morse-like kernel which is locally repulsive and has a negative convexity parameter, monotonicity depends on the condition in \cref{lem:interaction_kernels}, in combination with sufficient monotonicity from $u$, $l$, and the regularizing terms $\beta_j\norm{h_j}^2$
and $\alpha_i KL(\rho_i\, | \, \rhot_i)$.

For a numeric example, we consider a system with $N=2$ user populations, $M=2$ service providers, $d_h=2$, and $d_\rho=2$. An example application for this setting was illustrated in \cite{shekhtman_strategic_2024} as a bank account fraud problem, in which a population putting forward valid requests to open new bank accounts, $\rho_1$, and a population submitting fraudulent requests, $\rho_2$, are distributions over features such as income, customer age, number of months at a current residence, and similarity between the email and name associated with the application. We initialize two log-loss classifiers with the same cost function but different initial conditions. We compare two settings: one, neither population has nonlocal interactions, and two, the fraudulent population has attractive-repulsive nonlocal interactions. The attractive-repulsive kernel models the fraudulent agents attempting to differentiate themselves from other fraudulent agents in order to be misclassified, while having limited ability to move away from the overall population.
The functions used in the simulations are
\begin{align*}
    W_{11}= C_r \exp\left(-\frac{\norm{x}_1}{l_r}\right) - C_a \exp\left(-\frac{\norm{x}_1}{l_a}\right)\,, \quad  
    p(x,h) \coloneqq (1+\exp(-c \theta(h)^\top x + b(h)))^{-1}\,, \\
    u_{ij}(x_i,h_j)= 1-p(x_i,h_i)\,, \quad l_{1j}(x_1,h_j)=\log(p(x_1,h_j))\,,\quad l_{2j}=\log(1-p(x_2,h_j))\,, 
\end{align*}
where $\theta(h)\coloneqq [\cos h(0),\sin h(0)]$ and $b(h)=h(1)$ where $h(0)$ and $h(1)$ are the first and second entries of $h\in\R^2$. The function $p(x,h)$ is the probability of being approved for a loan, and $l_{ij}$ is a log-likelihood, which has a larger gradient (with respect to $h$) than $p$. The coefficients used are $\eta=0.5$, $c=2$, $\alpha_1=\alpha_2=0.5$, $\beta_1=\beta_2=2.4$, $C_r = 8$, $l_r = 0.5$, $C_a = 2$, and $l_a = 1$.  Even though this system is not $\lambda$-monotone in, for example, $\P_2(\R^2)\cap L^1(\R^2)\times \P_2(\R^2)\cap L^1(\R^2) \times \R^2 \times \R^2$, we empirically observe convergence.
\begin{figure}
    \begin{subfigure}{0.23\textwidth}
        \includegraphics[width=0.9\linewidth,trim={2.65cm 0cm 15.7cm 1.05cm},clip]{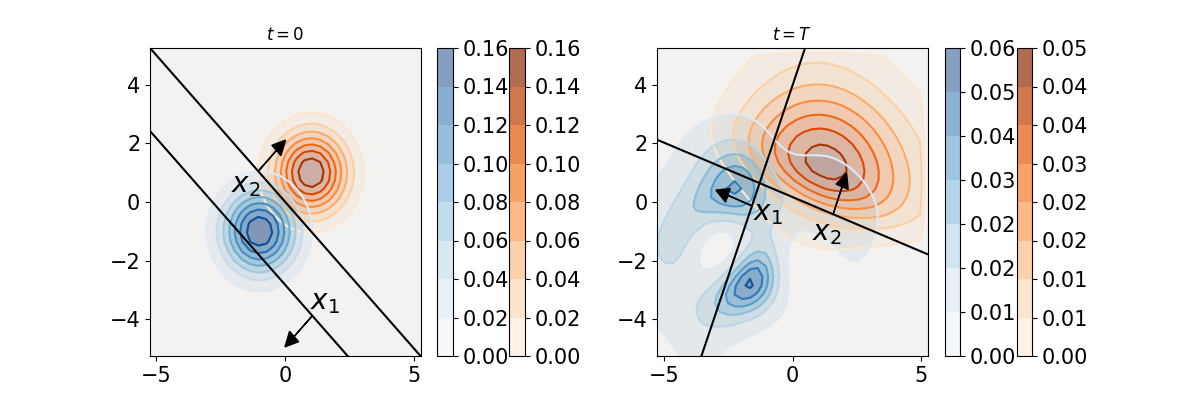}
    \caption{$T=0$
    }
    \end{subfigure}
    \begin{subfigure}{0.23\textwidth}
        \includegraphics[width=0.9\linewidth,trim={15.85cm 0 2.5cm 1.05cm},clip]{images/experiment2_partialweight_kernel_6.png}
    \caption{$T=6$
    }
    \end{subfigure}
    \begin{subfigure}{0.23\textwidth}
        \includegraphics[width=0.9\linewidth,trim={15.85cm 0 2.5cm 1.05cm},clip]{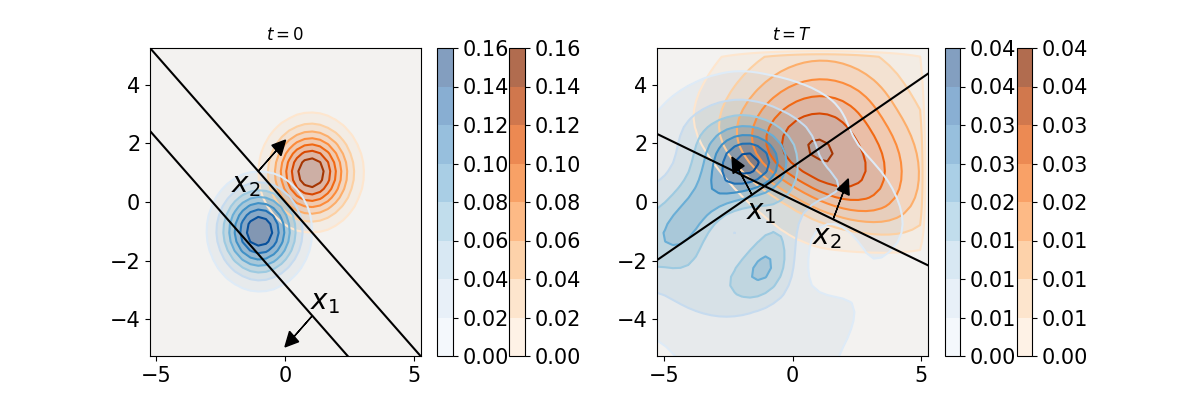}
        \caption{$T=12$
        }
    \end{subfigure}
    \begin{subfigure}{0.23\textwidth}
        \includegraphics[width=0.9\linewidth,trim={15.85cm 0 2.5cm 1.05cm},clip]{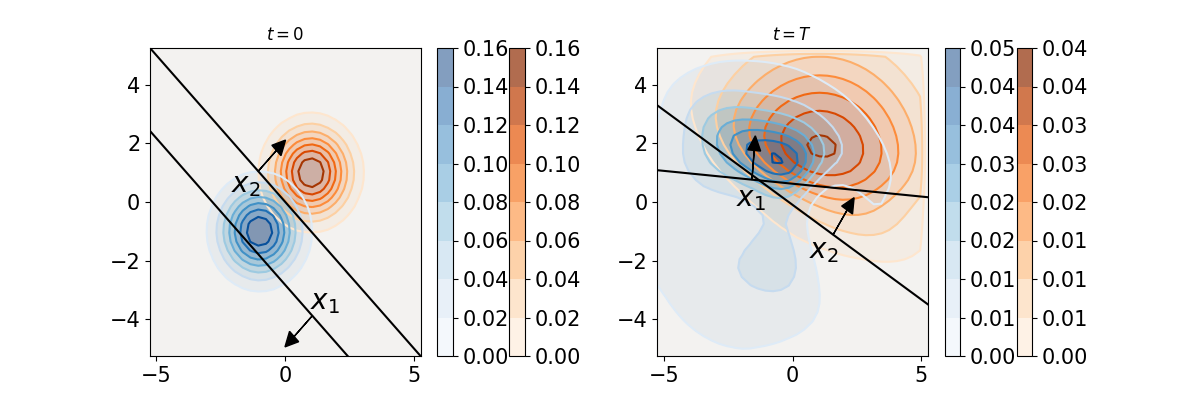}
        \caption{$T=20$
        }
    \end{subfigure}
    \caption{(Multi-Learner System) The fraudulent population (blue) and valid population (orange) both aim for high probability of approval from the classifiers. The black lines labeled $x_1$ and $x_2$ mark the $50\%$ classification probability threshold for $h_1$ and $h_2$ respectively, with the arrows pointing in the direction of increasing probability of loan approval. In this setting, the fraudulent population has an interaction kernel, and the fraudulent population splits into two modes, strongly preferring the blue classifier.}
    \label{fig:attractive_repulsive_kernel}
\end{figure}

\begin{figure}
    \begin{subfigure}{0.23\textwidth}
        \includegraphics[width=0.9\linewidth,trim={2.65cm 0cm 15.7cm 1.05cm},clip]{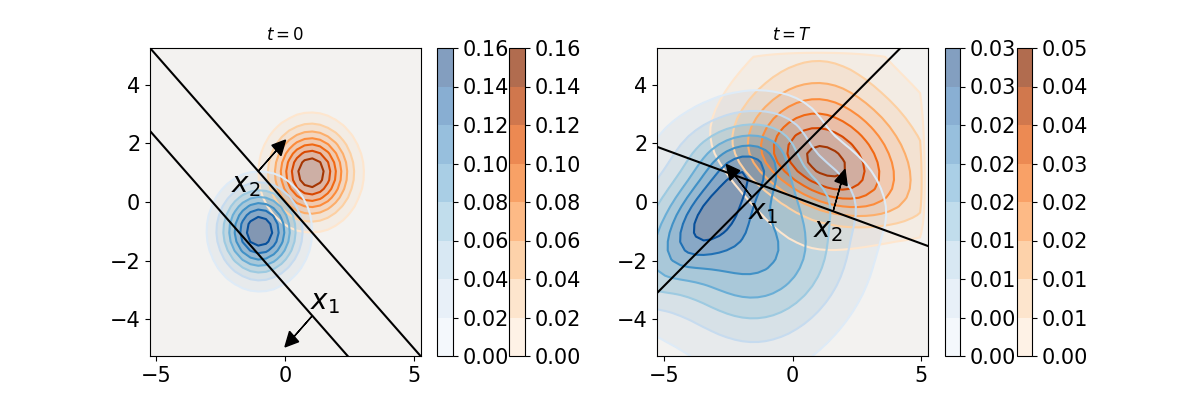}
    \caption{$T=0$
    }
    \end{subfigure}
    \begin{subfigure}{0.23\textwidth}
        \includegraphics[width=0.9\linewidth,trim={15.85cm 0 2.5cm 1.05cm},clip]{images/experiment2_partialweight_6.png}
    \caption{$T=6$
    }
    \end{subfigure}
    \begin{subfigure}{0.23\textwidth}
        \includegraphics[width=0.9\linewidth,trim={15.85cm 0 2.5cm 1.05cm},clip]{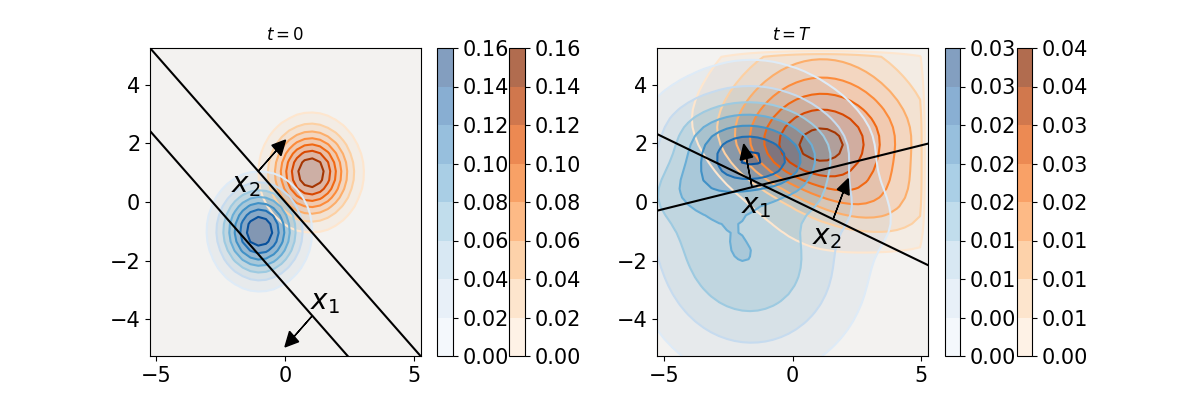}
        \caption{$T=12$
        }
    \end{subfigure}
    \begin{subfigure}{0.23\textwidth}
        \includegraphics[width=0.9\linewidth,trim={15.85cm 0 2.5cm 1.05cm},clip]{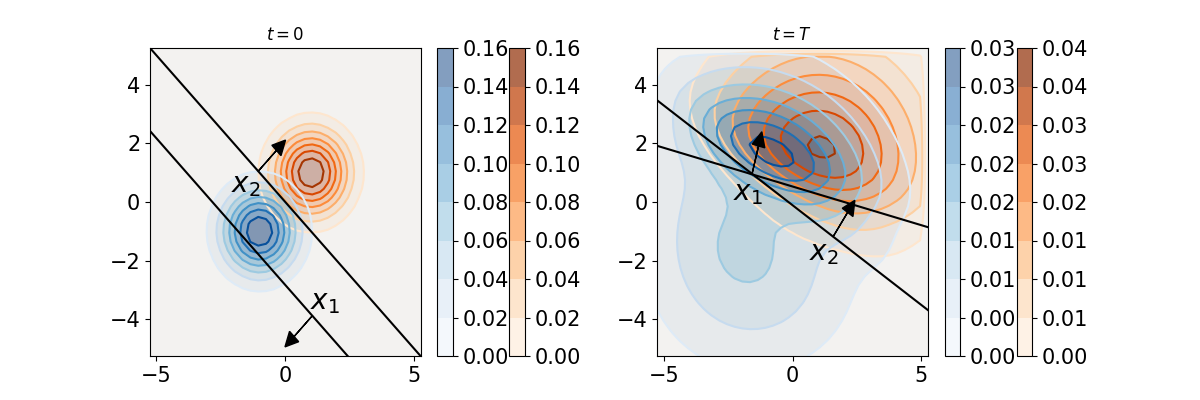}
        \caption{$T=20$
        }
    \end{subfigure}
    \caption{(Multi-Learner System) In the setting without any interaction kernels, the fraudulent population does not split. Additionally, there are more fraudulent agents obtaining services from bank $2$ than in the setting with the interaction kernel.}
    \label{fig:no_kernel}
\end{figure}

\begin{table}[]
    \centering
    \begin{tabular}{c c||c|c|c|c|}
   & kernel & $\rho_1$ using bank $1$ & $\rho_1$ using bank $2$ & $\rho_2$ using bank $1$ & $\rho_2$ using bank $2$ \\
    \hline \hline
    $T=0$ &   yes & $a_{11}=0.74$ & $a_{12}=0.26$ & $a_{21}=0.01$ & $a_{22}=0.99$ \\
     &   no & $a_{11}=0.74$ & $a_{12}=0.26$ & $a_{21}=0.01$ & $a_{22}=0.99$ \\ 
     \hline 
     $T=6$ &   yes & $a_{11}=0.73$ & $a_{12}=0.27$ & $a_{21}=0.03$ & $a_{22}=0.97$ \\
     &   no & $a_{11}=0.65$ & $a_{12}=0.35$ & $a_{21}=0.07$ & $a_{22}=0.93$ \\ 
     \hline
     $T=12$ &   yes & $a_{11}=0.58$ & $a_{12}=0.42$ & $a_{21}=0.11$ & $a_{22}=0.89$ \\
     &   no  & $a_{11}=0.52$ & $a_{12}=0.48$ & $a_{21}=0.25$ & $a_{22}=0.75$ \\
     \hline
     $T=20$ &   yes & $a_{11}=0.50$ & $a_{12}=0.50$ & $a_{21}=0.35$ & $a_{22}=0.65$ \\
     &   no & $a_{11}=0.45$ & $a_{12}=0.55$ & $a_{21}=0.40$ & $a_{22}=0.60$ \\ 
    \end{tabular}
    \caption{(Multi-Learner System) At times $T=6,12$, bank $1$ has a larger share of the fraudulent population ($\rho_1$) than bank $2$, with a larger margin in the kernel setting. At time $T=20$, the valid population ($\rho_2$) prefers bank $2$ with or without the interaction kernel, but the interaction kernel changes the bank preference for the fraudulent population.}
    \label{tab:classifier_allocation}
\end{table}

In comparing the trajectories from both settings, we observe two key differences. In \cref{fig:no_kernel}, the population never splits into a bimodal distribution, unlike in the setting with the nonlocal interaction kernel shown in \cref{fig:attractive_repulsive_kernel}. Secondly, as show in \cref{tab:classifier_allocation}, in the interaction kernel setting, bank $1$ has a worse customer profile at time steps $T=6$ and $T=12$, as it has more fraudulent customers and less valid customers than in the setting with no interaction kernel. Likewise, bank $2$, which has the better initial condition, has a better customer profile. Even though the final population distributions appear to be similar, the customer allocation between the two classifiers favors bank $2$.

\subsection{Team Mean Field Games}\label{ex:mean_field_game}
Consider a setting with a single energy $F:\A_1 \to \R$ and single species $\rho=\rho_1$. We would like to solve a mean-field game parameterized by the energy functional $F$, viewing $\rho$ as the density of particles participating in the game. Each particle aims to minimize a cost function, which includes a velocity penalty, $\norm{v}^2$, as well as a penalty on the entire distribution, given in terms of $F[\rho]$.  For long time horizons, a discount factor $\varepsilon$ is used to weigh the cost at earlier times more heavily than the cost at later times and a terminal cost is given by $g_\varepsilon$.  This model is a generalization of the model in \cite{butano_discounted_2024}, which describes the density of a crowd  of people with an intruder passing through. The optimization problem is 
\begin{align*}
    \inf_{v} \left\{ \begin{aligned}
       &\int_0^T \int\left( \frac{\norm{v(t,y(t))}^2}{2} e^{-t/\varepsilon} + f_\varepsilon(t,y(t),\rho(t)) \right) \d \rho(y(t))\,\d t + g_\varepsilon(T,y(T),\rho(T))
      \\
    &\qquad \quad  \text{such that }\d y(t)=v(t,y(t)) \d t\,,\   y(0)\sim \rho(0,\cdot)\,, \ \text{Law}\{y\} =\rho(t,\cdot) 
    \end{aligned}
     \right\}
    \,,
\end{align*}
where $f_\varepsilon(t,y(t),\rho(t))\coloneqq \frac{e^{-t/\varepsilon}}{\varepsilon} \delta_\rho F[\rho(t)]$ and $g_\varepsilon(t,y(t),\rho(t))=\varepsilon f_\varepsilon(t,y(t),\rho(t))=e^{-t/\varepsilon} \delta_\rho F[\rho(t)]$. 
The particle $y(t)$ is sampled from the distribution $\rho(t)$, and driven by the velocity field $v$.
This control problem is equivalent to the mean field game
\begin{align*}
    \inf_{v,\rho} \left\{ \begin{aligned}
      &\int_0^T \bigg(\int \frac{e^{-t/\varepsilon}}{2}\norm{v(t,x)}^2 \d\rho(t,x) + \frac{e^{-t/\varepsilon}}{\varepsilon} F[\rho(t)]\bigg)\d t + F[\rho(T)] e^{-T/\varepsilon} 
\\
&\qquad \quad \text{such that } \partial_t \rho(t) + \nabla \cdot (v(t) \rho(t)) = 0
    \end{aligned}
     \right\}
    \,,
\end{align*}
This mean field game is studied in \cite{zhang_mean-field_2023}, and it was shown that the optimal $(\rho,v)$ results in a perturbation of a gradient flow \cite[Theorem 14, Remark 15]{zhang_mean-field_2023}. Here, we consider an extension to the setting where different particles have different cost functionals depending on their type, or species. Collecting the particles by species, the optimization problem becomes
\begin{align*}
    \inf_{v_i,\rho_i} \left\{ \begin{aligned}
       &\int_0^T \left(\int \frac{e^{-t/\varepsilon}}{2}\norm{v_i(t,x_i)}^2 \d\rho_i(t,x_i) + \frac{e^{-t/\varepsilon}}{\varepsilon} F_i[\rho(t)]\right)\d t + F_i[\rho(T)] e^{-T/\varepsilon} 
 \\
&\qquad  \quad \text{such that }\partial_t \rho_i(t) + \nabla_{x_i} \cdot (v_i(t) \rho_i(t)) = 0\,,\quad \forall\, \iin
    \end{aligned}
     \right\}
    \,,
\end{align*}
Following a similar approach as the proof in \cite{zhang_mean-field_2023} (equations 74-75), the optimal solution takes the form
\begin{align}\label{eq:n_player_mfg}
\begin{split}
    \partial_t \rho_i &= \divs{x_i}{\rho_i \nabla_{x_i} \delta_{\rho_i}F_i[\rho]} + \varepsilon \divs{x_i}{\rho_i \nabla_{x_i}(\partial_t U_i - \frac{1}{2}|\nabla U_i|^2)} \\
    -\varepsilon \partial_t U_i &+ U_i + \frac{\varepsilon}{2}|\nabla_{x_i} U_i|^2 = \delta_{\rho_i} F_i[\rho] \,.
\end{split}
\end{align}
For monotone $F$, \eqref{eq:n_player_mfg} is a perturbation of coupled gradient flows with monotone energies. It is an interesting open problem to determine the sufficient conditions on $F_i[\rho_i,\cdot]$, such as smoothness, under which we expect that the result in \cite[Theorem 14]{zhang_mean-field_2023} for the limit as $\varepsilon \rightarrow 0$ can be shown in this multispecies setting. The long-time behavior of the solution in the limit $\varepsilon \rightarrow 0$ can then be analyzed using our monotonicity framework, since it follows the evolution \eqref{eq:coupled_gradient_flow_dynamics}.

\subsection{Degenerately Monotone Game with Diffusion}
For an energy functional that is "almost" strongly convex, exponential convergence of the corresponding gradient flow can still be shown if the energy functional also contains a linear entropy term. The notion of convexity used is called \emph{degenerate} convexity. A potential function $V:\R^d\to \R$ is degenerately convex if $\nabla^2 V \ge \psi(x)$ for some $\psi\in C(\R^d)$ with $\psi(x)>0$ if $x\ne 0$, and $\psi(x)$ is uniformly bounded from below by some positive constant as $\norm{x}\to \infty$. Exponential convergence for gradient flows with energies of the form
\begin{align*}
    F_i[\rho_i]=\int V_i(x_i)\d\rho_i(x_i) + \int W_i(x_i-y_i) \d \rho_i(x_i) \d \rho_i(y_i) + \int \log\rho_i(x_i) \d \rho_i(x_i)\,,
\end{align*}
where $V_i$ and $W_i$ are degenerately convex, are given in \cite{carrillo_kinetic_2003}; this occurs because the entropy term ensures $\rho_i$ has mass in regions of the $\R^{d_i}$ where the Hessians of $V_i$ and $W_i$ are strongly convex.
In finite dimensions, a convex potential that is not strongly convex results in a polynomial rate, but not an exponential rate. Here we apply the results from \cite{carrillo_kinetic_2003}, Rule \#5, to obtain a rate in combination with the effects of coupling between the species. From a PDE analysis perspective, it is interesting that the strength of the inter-species coupling does not necessarily change the exponential convergence rate. From a game theory perspective, our framework expands the set of games typically considered by including entropy in the cost  function, and a framework for understanding why adding entropy, even though it is not strongly convex, enables an exponential rate. 
The energy functionals are $F_1=F$ and $F_2=-F$, with $F$ given by
\begin{align*}
    F[\rho_1,\rho_2] &= \frac{a}{2}\int \norm{x_1}^4 \d\rho_1(x_1) - \frac{a}{2}\int \norm{x_2}^4 \d \rho_2(x_2) 
    + \frac{b}{2}\iint x_1 ^\top x_2 \d \rho_1(x_1) \d \rho_2(x_2) + c(\Hc(\rho_1)-\Hc(\rho_2))\,,
\end{align*}
where $\Hc(\rho_i) = \int \log \rho_i(x_i) \d \rho_i(x_i)$, and $a,b,c>0$. In \cite{carrillo_kinetic_2003}, the authors illustrate how functionals with degenerate convexity can still have exponential convergence due to diffusion, which keeps mass in an area of the domain where the energy can be lower-bounded by a strongly convex function. In our example, the quartic term is not strongly convex; however, the diffusion term allows us to obtain a rate and we call this setting \emph{degenerately monotone}.
The dynamics are given by
\begin{align*}
    \partial_t \rho_1 = -\nabla_{\Wass_2,\rho_1}F[\rho_1,\rho_2] \qquad \partial_t \rho_2 = \nabla_{\Wass_2,\rho_2} F[\rho_1,\rho_2]\,.
\end{align*}
In \cref{fig:quartic_cv}, the energy is plotted as a function of time, with the initial energy normalized to $F[\rho(0)]=1$. Fixing $a=15$ and $c=0.5$, we see that the value of $b$ has little effect on the convergence rate but does affect the oscillation frequency. This is because the interaction term is linear for each player.
\begin{figure}[!htb]
    \centering
    \includegraphics[width=0.8\linewidth]{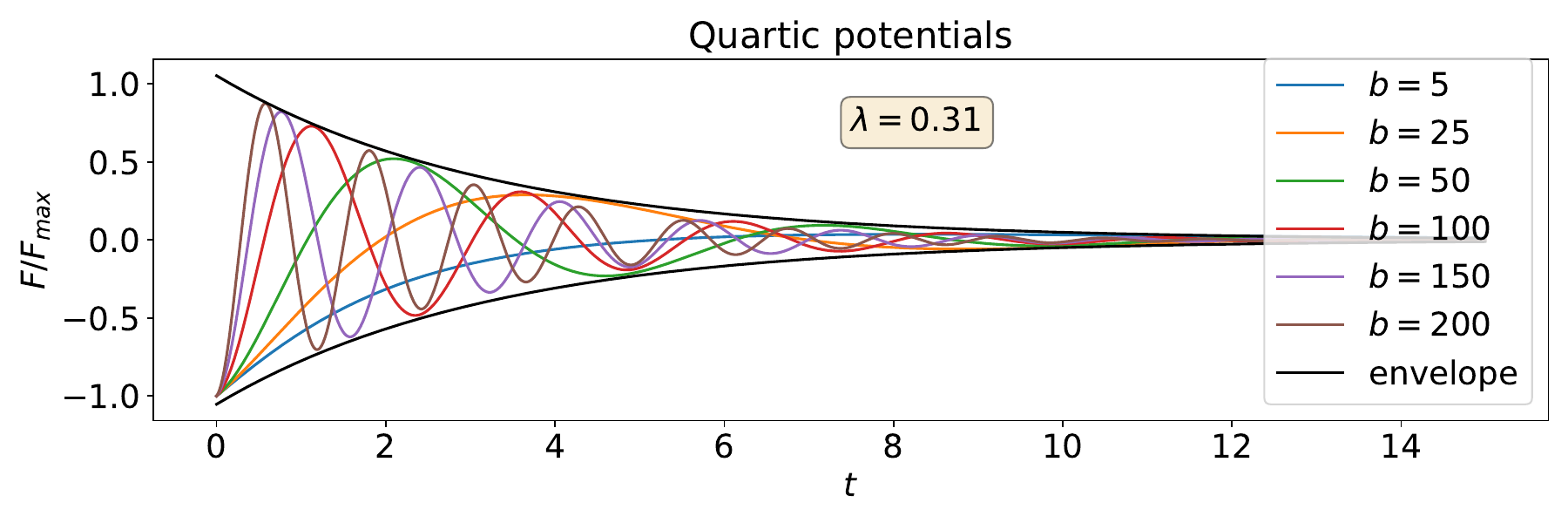}
    \caption{(Degenerately Monotone Game with Diffusion) The upper-bound for the convergence rate across the tested values of $b$ is $\lambda=0.31$. As the coupling between the two players strengthens, the oscillations of $F$ increase in frequency, but still decay at roughly the same rate.}
    \label{fig:quartic_cv}
\end{figure}
\begin{figure}
    \centering
    \includegraphics[width=0.99\linewidth]{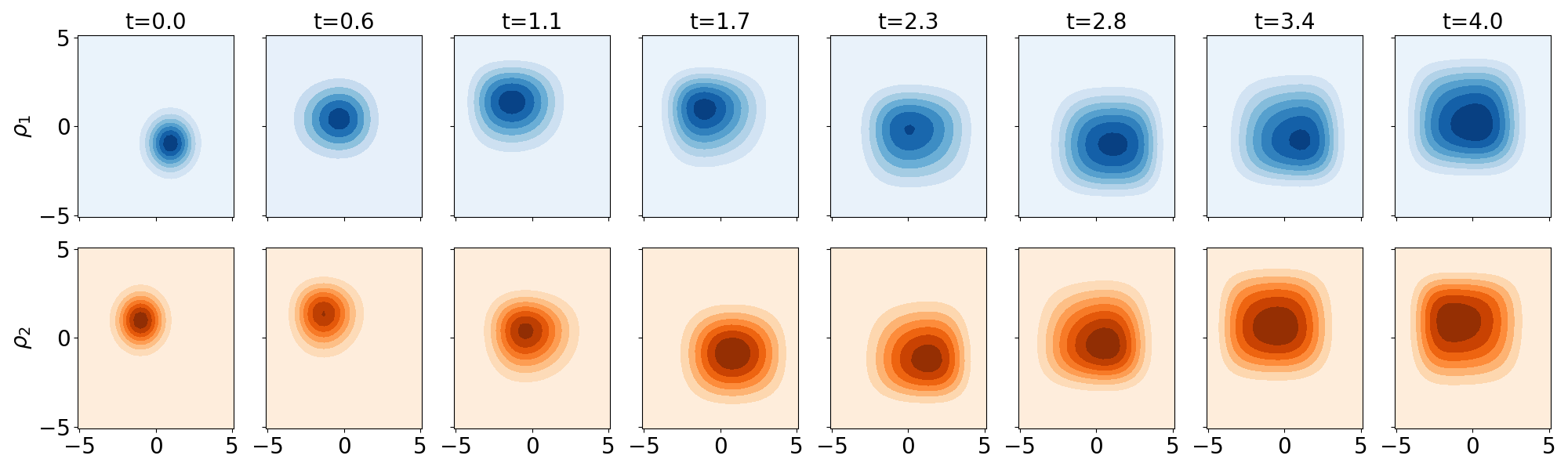}
    \caption{(Degenerately Monotone Game with Diffusion) In the setting where $b=150$, the density of each species is plotted from $t=0$ to $t=4$, during which time the energy has oscillated a few times. The densities for $\rho_1$ (blue) and $\rho_2$ (orange) spread out due to diffusion, eventually minimizing a quartic-shaped basin. The species $\rho_2$ aims to overlap with $\rho_1$ while $\rho_1$ aims to be in the opposite quadrant as $\rho_2$, due to the coupling term $\iint x_1^\top x_2 \d \rho_1(x_1)\d\rho_2(x_2)$, which $\rho_1$ aims to minimize and $\rho_2$ aims to maximize.}
    \label{fig:quartic_densities}
\end{figure}
The simulation indicates that convergence can still occur without strong monotonicity, as  expected.
In \cref{fig:quartic_densities}, the densities of $\rho_1$ and $\rho_2$ are plotted for $t\in[0,4]$, and the oscillating behavior of both species is demonstrated. Eventually the densities have a shape consistent with quartic tails over $\R^2$. An interesting area of future work would be to explore not only the setting of degenerate monotonicity, but also more general non-monotone settings.

\section*{Acknowledgments}
The authors are grateful for discussions with Jos\'e A. Carrillo about displacement convexity, with Giuseppe Savar\'e about accretivity and recent literature, and with Xianjin Yang about mean field games, as well as to Dejan Slep\v cev, Eitan Levin, and Matthieu Darcy. Guilia Cavagnari provided a correction for the assumption of \cref{thm:admissible_couplings}. LC is funded by an NDSEG fellowship from the AFOSR and a PIMCO fellowship, FH is supported by start-up funds at the California Institute of Technology and by NSF CAREER Award 2340762, EM is supported in part from NSF award 2240110, LJR is supported in part by NSF 1844729, NSF 2312775, and ONR YIP N000142012571.

\appendix
\section{Auxiliary Definitions and Results}\label{sec:appendix}
We start with an overview of different characterizations of the Wasserstein-2 metric used throughout the manuscript. Then we detail how to derive the equivalence between first order and second order monotonicity conditions in \cref{sec:2nd-order-monotonicity}.

\begin{definition}[Wasserstein-2 Metric]\label{def:wasserstein}
The Wasserstein-2 metric between two probability measures $\mu, \nu\in\P_2(\R^d)$ is given by
    \begin{align*}
        \Wass_2(\mu, \nu)^2 = \inf_{\gamma\in\Gamma(\mu,\nu)} \int \norm{z-z'}_2^2 \d \gamma(z,z')
    \end{align*}
    where $\Gamma$ is the set of all joint probability distributions with marginals $\mu$ and $\nu$, i.e. $\mu (\d z) = \int \gamma(\d z,z')\d z'$ and $\nu(\d z') = \int \gamma(z,\d z') \d z$. 
\end{definition}
If $\mu,\nu\in\P_2(\R^d)$ are such that at least one transport map $T$ pushing forward $\mu$ to $\nu$ exists (written as $\nu=T_\sharp \mu$), we obtain the \emph{Monge formulation} of the Wasserstein-2 distance:
\begin{align*}
    \Wass_2(\mu,\nu)^2 = \inf_{T\ :\ \nu=T_\sharp \mu} \int \norm{x-T(x)}^2 \d \mu(x) \,.
\end{align*}
Since this is not true for all pairs $\mu,\nu\in\P_2(\R^d)$, the Kantorovich formulation \cref{def:wasserstein} is more general than the Monge formulation.
Another way of writing the Wasserstein-2 Metric is via the \emph{Benamou-Brenier formulation} (also known as the \emph{dynamic formulation} of $\Wass_2$), expressing the distance between two probability measures $\mu,\nu\in\P_2(\R^d)$ as the path solving a continuity equation with minimal kinetic energy:
\begin{align*}
    \Wass_2(\mu,\nu)^2 = \inf_{(\mu_t,w_t)} \left\{ \int_0^1 \int \norm{w(t,x)}^2 \d \mu(t,x) \d t\,;\ \partial_t\mu_t + \nabla \cdot (\mu_t w_t)=0\, \ \text{s.t.} \ \mu_0=\mu,\,\mu_1=\nu\right\}\,.
\end{align*}
The Benamou-Brenier formulation is useful for geodesic calculations, as the optimal $(\mu,w)$ are also solutions to the geodesic equation \eqref{eq:geodesic_eqns}. 

\subsection{Second order monotonicity condition}\label{sec:2nd-order-monotonicity}

In what follows, we prove the equivalence of the first and second order monotonicity conditions (see \cref{def:monotonicity} and \cref{prop:monotonicity_implication}). 
Recall the displacement interpolant between $\rhozero_i$ and $\rhoone_i$ is defined as $\rhos_i = (I^{(s)}_i)_{\sharp} \gamma_i$, where $I^{(s)}_i(x_i,y_i)=(1-s)x_i + s y_i$, for all $s\in[0,1]$ and $\gamma_i\in\Gamma^*(\rhozero_i,\rhoone_i)$.
\begin{lemma}[Second-Order Monotonicity, Geodesics]\label{lem:second_order_monotonicity_v_geodesic} 
    Let \cref{assump:regularity_v} hold and $\A\subset \Pbar$ be geodesically convex. Then $v:\A \to \R^n$ is $\lambda$-monotone in $\A$ if and only if
    \begin{align*}
    &-\int_0^1\bigg(\iint \<x-y,J_2[\rhos](I^{(s)}(x,y), I^{(s)}(\hat x,\hat y))\, (\hat x-\hat y) >\d \gamma(x,y) \d \gamma(\hat x,\hat y) \\
    & \qquad+\int\<x-y,J_1[\rhos](I^{(s)}(x,y))\, (x-y) > \d \gamma(x,y)\bigg) \d s \ge \lambda \int\norm{x-y}^2\d\gamma(x,y)\,,
\end{align*}
holds for all $\rhozero,\rhoone\in\A$ and $\gamma\in\Gamma^*(\rhozero,\rhoone)$, where $(\rho^{(s)})_{s\in[0,1]}$ denotes the geodesic between $\rhozero$ and $\rhoone$.
\end{lemma}
\begin{proof}
    Select $\rhozero,\rhoone\in\A$ and $\gamma\in\Gamma^*(\rhozero,\rhoone)$.
     We rewrite \cref{def:monotonicity} in terms of the derivatives of the interpolants evaluated at $s=0$ and $s=1$:
\begin{align*}
    \left(\sum_{i=1}^n \int \< y_i-x_i , v_i[\rhos]((1-s)x_i + s y_i) >\d \gamma_i(x_i,y_i)\right)\bigg|_{s=0} &= \int \<y-x,v[\rhozero](x)> \d \gamma(x,y) \,, \\
        \left(\sum_{i=1}^n\int \< y_i-x_i , v_i[\rhos]((1-s)x_i + s y_i) > \d \gamma_i(x_i,y_i) \right)\bigg|_{s=1} 
        &= \int \<y-x ,v[\rhoone](y)> \d \gamma(x,y)\,.
\end{align*}
    Noting that
    \begin{align*}
        &\left(\sum_{i=1}^n \int \< y_i -x_i, v_i[\rhos]((1-s)x_i + s y_i) > \d \gamma_i(x_i,y_i) \right)\bigg|_{s=1}\\
        &-\left(\sum_{i=1}^n \int \< y_i -x_i, v_i[\rhos]((1-s)x_i + s y_i) >\d \gamma_i(x_i,y_i) \right)\bigg|_{s=0}  \\
        &\qquad = \int_0^1 \frac{\d}{\d \tau} \left(\sum_{i=1}^n \int \< y_i-x_i , v_i[\rhotau]((1-\tau)x_i + \tau y_i) > \d \gamma_i(x_i,y_i) \right)\d \tau\,,
    \end{align*}
monotonicity according to \cref{def:monotonicity} can be equivalently written as
    \begin{align*}
        -\int_0^1 \frac{\d}{\d \tau} \left(\sum_{i=1}^n \int \< y_i-x_i , v_i[\rhotau]((1-\tau)x_i + \tau y_i) > \d \gamma_i(x_i,y_i) \right)\d \tau \ge \lambda \int\norm{x-y}^2\d\gamma(x,y)\,.
    \end{align*}
We will use derivatives of the interpolants, which are computed in the weak form via
    \begin{align*}
       \int \partial_s \rhos_i(x_i) \phi(x_i) \d x_i = \int  \partial_s \phi(I^{(s)}(x_i,y_i))\d \gamma_i(x_i,y_i) 
       =  \int \<y_i-x_i,\nabla \phi(I^{(s)}(x_i,y_i))>\d \gamma_i(x,y) \\
       \forall\, \phi\in C_c^1(\R^{d_i})\,.
    \end{align*} 
Writing out the derivative results in
\begin{align*}
    \dds  &\bigg(\sum_{i=1}^n \int \< y_i-x_i , v_i[\rhos](I^{(s)}_i(x_i,y_i)) >\d \gamma_i(x_i,y_i) \bigg)  \\
    &=\sum_{i=1}^n \int \<x_i-y_i,\nabla_{x_i} v_i[\rhos](I^{(s)}_i(x_i,y_i))\, (x_i-y_i)> \d \gamma_i(x_i,y_i) \\
    &+\sum_{i,j=1}^n \iint \<\hat x_j-\hat y_j,\nabla_{x_j} \delta_{ \rho_j} v_i[\rhos](I^{(s)}_i(x_i,y_i),I^{(s)}_i(\hat x_j,\hat y_j))\, (x_i-y_i)> \d \gamma_i(x_i,y_i)\, \d \gamma_j(\hat x_j,\hat y_j) \\
        &= \iint \<x-y,J_2[\rhos](I^{(s)}(x,y),I^{(s)}(\hat x,\hat y))\, (\hat x -\hat y) >\d \gamma(x,y) \d \gamma(\hat x,\hat y) 
        \\
        &
        +\int\<x-y,J_1[\rhos](I^{(s)}(x,y))\, (x-y) > \d \gamma(x,y) \,. 
\end{align*}
Now we have shown that, for $\rhos$ the geodesic interpolant between $\rhozero$ and $\rhoone$,
\begin{align*}
    &\int_0^1\bigg(\iint \<x-y,J_2[\rhos](I^{(s)}(x,y),I^{(s)}(\hat x,\hat y))\, (\hat x-\hat y) >\d \gamma(x,y) \d \gamma(\hat x,\hat y)  \\
    &\qquad+\int\<x-y,J_1[\rhos](I^{(s)}(x,y))\, (x-y) > \d \gamma(x,y)\bigg) \d s \\
    &= 
    \int \<x-y,v[\rhozero](x)-v[\rhoone](y)> \d \gamma(x,y) \,,
\end{align*}
and therefore the second-order geodesic notion of monotonicity is equivalent to the first-order definition.
\end{proof}

Since $\Wbar$ has only a formal Reimannian structure, we cannot directly apply results for monotonicity in the Riemannian manifold setting, but we can use an appropriate notion of tangent vectors instead of exponential maps, following the proof techniques which are outlined in \cite[Proposition 16.2]{Villani07}.

\begin{lemma}[Second-Order Monotonicity, Velocity Fields]\label{lem:second_order_monotonicity_v} 
    Let \cref{assump:regularity_v,assump:continuity_v,assump:tangent_space} hold and $\A\subset \Pbar$ be geodesically convex. 
Then $v:\A \to \R^n$ is $\lambda$-monotone in $\A$ if and only if
    \begin{align*}
   - \int\< w(x),J_1[\rhozero](x)\, w(x) > \d \rhozero(x) -\iint \< w(x),J_2[\rhozero](x,\hat x)\,  w(\hat x) >\d \rhozero(x) \d \rhozero(\hat x) \qquad
       \\ \ge \lambda \int \norm{w(x)}^2 \d \rhozero(x)\,.
\end{align*}
    for all $\rhozero\in\A$ and  $w_i\in Y_i(\rho^{(0)}_i,\A_i)$ for all $\iin$.
\end{lemma}
\begin{proof}
    We will show this local notion of monotonicity using ideas from the proof of \cite[Proposition 16.2]{Villani07}. Fix $\rho\in\A$ and a set of velocity fields $\{w_i\in Y_i(\rho_i,\A_i)\}_{i=1}^n$. Since $w_i\in Y_i(\rho_i,\A_i)$ by \cref{assump:tangent_space},
\begin{align*}
\rhot_i:=(w_i/c + \id)_\sharp \rho_i \in \A_i \quad \text{and} \quad
    (\id,w_i/c+ \id)_\sharp \rho_i \in \Gamma^*(\rho_i, \rhot_i) 
\end{align*}
for some $c>0$.  The geodesic interpolant between $\rho_i$ and $\rhot_i$ is given by
\begin{align*}
    \rhotau_i := [(1-\tau)\id + \tau(w_i/c + \id) ]_\sharp \rho_i = [\id + \tau w_i /c]_\sharp \rho_i\,,\qquad \tau\in[0,1] \,.
\end{align*}
Since both $\rho_i,\rhot_i\in \A_i$, we can use \cref{lem:second_order_monotonicity_v_geodesic} applied to $\rho$ and $\rhotau$.
 Noting that $\gamma_i^{(\tau)} := (\id,\id +\tau w_i / c)_\sharp \rho_i$ is an optimal plan between $\rho_i$ and $\rhotau_i$, and $I^{(s)}(x,x+\tau w(x)/c)=x+s \tau w(x)/c$, we have that the geodesic at $s\in [0,1]$ between $\rho_i$ and $\rhotau_i$ is given by $\rho^{(s\tau)}$, and
\begin{align*}
    &-\int_0^1\bigg(\iint \<x-y,J_2[\rho^{(s \tau)}](I^{(s)}(x,y),I^{(s)}(\hat x,\hat y))\, (\hat x-\hat y) >\d \gamma^{(\tau)}(x,y) \d \gamma^{(\tau)}(\hat x,\hat y)  \\
    &\qquad 
    +\int\<x-y,J_1[\rho^{(s\tau)}](I^{(s)}(x,y))\, (x-y) > \d \gamma^{(\tau)}(x,y)\bigg) \d s \\
    & =-\frac{\tau^2}{c^2} \int_0^1  \bigg(\iint \<w(x),J_2[\rho^{(s \tau)}](x+s \tau w(x)/c,\hat x+s \tau w(\hat x)/c)\cdot w(\hat x) >\d \rho(x) \d \rho(\hat x)  \\
    & \qquad  +\frac{\tau^2}{c^2}\int\<w(x),J_1[\rho^{(s\tau)}](x+s \tau w(x)/c) \cdot w(x) > \d \rho(x) \bigg) \d s \ge \lambda \frac{\tau^2}{c^2} \int \norm{w(x)}^2 \d \rho(x)\,.
\end{align*}
Since $\A$ is convex, $\rho^{(\tau)}\in\A$ for all $\tau\in[0,1]$, and the monotonicity condition holds as $\tau \rightarrow 0$. Thanks to the continuity properties in \cref{assump:continuity_v}, dividing both sides by $\frac{\tau^2}{c^2}$ and taking the limit as $\tau \rightarrow 0$ gives
\begin{align*}
    &-\iint \< w(x),J_2[\rho](x,\hat x)\,  w(\hat x) >\d \rho(x) \d \rho(\hat x) 
       -\int\< w(x),J_1[\rho](x)\, w(x) > \d \rho(x)
        \ge \lambda \int \norm{w(x)}^2 \d \rho(x)\,.
\end{align*}
\end{proof}

\begin{proof}[Proof of \cref{prop:monotonicity_implication}]
\sloppy The proof follows by applying \cref{lem:second_order_monotonicity_v_geodesic} and \cref{lem:second_order_monotonicity_v} with $v_i[\rho]=-\nabla_{x_i}\delta_{\rho_i}F_i[\rho]$.
\end{proof}

\subsection{Monotonicity vs Convexity and Measures vs Vectors}\label{subsec:monotonicity_convexity}
In \cref{tab:convexity_monotonicity_comparison_first_order,tab:convexity_monotonicity_comparison_second_order}, the differences between monotonicity and convexity and between games in Euclidean spaces and games in measure spaces over Euclidean spaces are detailed, where the finite-dimensional games are defined as in \cref{def:finite_dim_monotone}. 
For both convexity and monotonicity, a key difference between finite and infinite dimensions is the role of the Onsager operator's dependence on $\rho$ for the $\Wass_2$ metric. For simplicity, we will illustrate this in the single species setting, the multi-species setting follows similarly. Let $(U,d)$ be a metric space, and let $E:U \to \R\cup \{+\infty\}$ be an energy functional. 
Consider the single species (informal)  gradient flow
\begin{align}\label{eq:general_gf}
    \dot u = -\mathbbm{K}(u) DE(u)\,,
\end{align}
where $\mathbbm K(u):(T_u U)^* \to T_u U$ is the Onsager operator corresponding to the metric $d$, mapping from the dual of the tangent space at $u$ to the tangent space at $u$. $D$ is a differential operator on the energy $E$.  We would like a condition under which the functional $$V_2(x):=\frac{1}{2}\<DE(u),\mathbbm{K}(u) DE(u)>$$ converges exponentially along solutions to \eqref{eq:general_gf} (see the Lyapunov functionals defined in \eqref{eq:lyapunov_functions} for comparison). Taking a time derivative of $V_2$,
\begin{align*}
    \dot V_2(u) &= \frac{1}{2} \<DE(u), \frac{\d}{\d u}\left(\mathbbm{K}(u) DE(u)\right)[\dot u]> +\frac{1}{2}\<\dot u D^2 E(u),\mathbbm{K}(u) DE(u)>  \\
    &= \frac{1}{2} \<DE(u), \frac{\d}{\d u}\left(\mathbbm{K}(u)\right)[\dot u] DE(u)> +\<\dot u D^2 E(u),\mathbbm{K}(u) DE(u)>\,,
\end{align*}
where the last equality above holds if $\mathbbm K(u)$ is a linear operator, and $ \frac{\d}{\d u}\left(\mathbbm{K}(u)\right)[\dot u] DE(u)$ denotes the derivative of $\mathbbm K$ with respect to $u$ in the direction $\dot u$, acting on $DE(u)$. A convexity condition allows to conclude that
\begin{align*}
    \dot V_2(u) \le -\lambda \< DE(u),\mathbbm{K}(u) DE(u)>=-2\lambda V_2(u)\,,
\end{align*}
guaranteeing exponential convergence of $V_2$. For more details on the above general framework for gradient flows, see \cite{mielke_geodesic_2013}.
 Next, we instantiate this general framework for the Euclidean case, where $U=\R^d$ and $d = \norm{\cdot}$, and in the case of measures, selecting $U=\P_2$ and $d=\Wass_2$. In the Euclidean setting, we identify $DE(u)$ as  $\nabla E(u)$ and $D^2E(u)$ as $\nabla^2 E(u)$; in the measure setting, we identify $DE(u)$ as $\delta E(u)$ and $D^2E(u)$ as $\delta^2 E(u)$. We have
\begin{align*}
    &\text{Euclidean:} & \mathbbm{K}(u) &= \id\,,    &\frac{\d}{\d u}\left(\mathbbm{K}(u)\right)[\dot u] &=0\,, & \\
    & \text{Measures:}  & \mathbbm{K}(u) &= -\div{u \nabla \cdot} \,,     &\frac{\d}{\d u}\left(\mathbbm{K}(u)\right)[\dot u]  &= -\div{\dot u \nabla \cdot}\,.&
\end{align*}
In the Euclidean setting,
\begin{align*}
    \dot V_2(u) = -\<\nabla E(u),\nabla^2 E(u) \cdot \nabla E(u)> \le -\lambda \<\nabla E(u),\nabla E(u)> = -2\lambda V_2(u)\,,
\end{align*}
due to the convexity condition  $\nabla^2 E(u) \succeq \lambda \Id_d$ for all $u\in\R^d$. In the measure setting, we have
\begin{align*}
    \dot V_2(u) &= -\frac{1}{2} \< \delta E(u) ,\div{\dot u\nabla \delta  E(u)} > - \<\div{u \nabla \delta E(u)}\delta^2 E(u),\div{u\nabla \delta E(u)}> \\
    &= \frac{1}{2} \int \norm{\nabla \delta E[u](x)}^2 \div{u(x) \nabla \delta E[u](x)} \d x\\
    &- \iint \<\nabla\delta E[u](x), \nabla^2 \delta^2 E[u](x,\hat x) \nabla \delta E[u](\hat x)> \d u(x) \d u(\hat x) \\
    &= -\int \<\nabla \delta E[u](x),\nabla^2 \delta E[u](x) \nabla \delta E[u](x)> \d u(x) \\
    &\quad - \iint \<\nabla\delta E[u](x), \nabla^2 \delta^2 E[u](x,\hat x) \nabla \delta E[u](\hat x)> \d u(x) \d u(\hat x)\,.
\end{align*}
We recognize the above expression as the left-hand side in the second-order convexity condition (single-species monotonicity in \cref{lem:second_order_monotonicity_v}), which results in 
\begin{align*}
    \dot V_2(u) \le -2\lambda V_2(u)\,.
\end{align*}
In summary, since in the Euclidean case $\mathbbm K$ does not depend on $u$, and so $\frac{\d}{\d u}\left(\mathbbm{K}(u)\right)[\dot u] DE(u) =0$ and there is only one Hessian-like term in the second-order monotonicity condition in the finite-dimensional case (see \cref{def:finite_dim_monotone} or \cref{tab:convexity_monotonicity_comparison_first_order}). In contrast, this term is not zero in the measure setting (see \cref{prop:monotonicity_implication} or \cref{tab:convexity_monotonicity_comparison_second_order}) as the derivative of the Onsager operator results in $\int \<\nabla \delta E[u](x),\nabla^2 \delta E[u](x) \nabla \delta E[u](x)> \d u(x)$ in the $\Wass_2$ case. If the Onsager operator was nonlinear in the object on which it acts (for instance when there is a cosh structure \cite{peletier_cosh_2023} instead of a quadratic dissipation structure), then there could be three Hessian-like terms, with the third coming from the derivative of $\mathbbm K$ with respect to $DE$.
\begin{table}[h]
    \centering
    \caption*{Monotonicity and Convexity over $\R^d$}
    \begin{tabular}{c|ccc}
        \textcolor{gray}{ $\forall\, x,y\in\R^d$} &  First Order  & & Second Order \\
        \hline
         \rule{0pt}{4ex}  Convexity   & $\< \nabla f(x)-\nabla f(y),x-y> \ge \lambda \norm{x-y}^2$  & & $\frac{1}{2}(\nabla^2 f(x) + \nabla^2 f(x)^\top) \succeq \lambda \Id_d$  \\ 
       \rule{0pt}{4ex} Monotonicity   &  $ \< \overline\nabla V(x)-\overline \nabla V(y), x-y > \ge \lambda \norm{x-y}^2$  & & $ \frac{1}{2}(\overline \nabla^2 V(x)+ \overline \nabla^2 V(x)^\top ) \succeq \lambda  \Id_d$   \\[2ex]
    \end{tabular}
    \caption{(Coupled Gradient Flow Setting) The key difference between monotonicity and convexity is the operator $\overline \nabla$ which is \textit{not} a gradient operator;  it is an entry-wise operator which takes a gradient with respect to different states.}
    \label{tab:convexity_monotonicity_comparison_first_order}
\end{table}
\begin{table}[h]
    \centering
    \caption*{Monotonicity and Convexity over $\P(\R^d)$}
    \begin{tabular}{c|ccc}
        & First Order  \\
         \hline
       \rule{0pt}{4ex} Convexity & $ \int  \left\langle \nabla \delta_{\rho_1} F[\rhozero_1](x)-\nabla \delta_{\rho_1} F[\rhoone_1](y),  x-y\right\rangle\d \gamma_1(x) \ge \lambda \Wass_2(\rhozero_1,\rhoone_1)^2$   \\
       \rule{0pt}{2ex} & \textcolor{gray}{$\forall\, \rhozero_1,\rhoone_1\in \A_1$, $ \forall\, \gamma_1\in\Gamma^*(\rhozero_1,\rhoone_1)$}  \\
      \rule{0pt}{4ex} Monotonicity &   $ \int  \left\langle \ndr F[\rhozero](x)-\ndr F[\rhoone](y), x-y\right\rangle\d \gamma(x,y) \ge \lambda \Wbar(\rhozero,\rhoone)^2$  \\
      \rule{0pt}{3ex}  &  \textcolor{gray}{$\forall\, \rhozero,\rhoone\in \A$,  $ \forall\, \gamma\in\Gamma^*(\rhozero,\rhoone)$} \\
    \end{tabular}
    
    \begin{tabular}{c|ccc}
     \rule{0pt}{4ex}& Second Order \\
        \hline
        \rule{0pt}{4ex} Convexity & $\int (\iint \langle x-y,\nabla^2 \delta^2_{\rho_1} F[\rhos_1](I^{(s)}(x,y),I^{(s)}(\hat x,\hat y)) (\hat x-\hat y)\rangle  \d \gamma_1(x,y) \d \gamma_1(\hat x,\hat y) $  \\
       &  $+ \int\langle x-y,\nabla^2 \delta_{\rho_1} F[\rhos_1](I^{(s)}(x,y)) (x-y) \rangle  \d \gamma_1(x,y)) \d s \ge \lambda \Wbar(\rhozero_1,\rhoone_1)^2$  \\
      \rule{0pt}{3ex}  & \textcolor{gray}{$\forall\, \rhozero_1,\rhoone_1\in\A_1$, $\gamma_1\in\Gamma^*(\rhozero_1,\rhoone_1)$} \\
     \rule{0pt}{4ex}  Monotonicity&  $\int (\iint \langle x-y,\nnddr F[\rhos](I^{(s)}(x,y),I^{(s)}(\hat x,\hat y)) (\hat x-\hat y)\rangle \d \gamma(x,y) \d \gamma(\hat x,\hat y) $ \\
      & $+ \int \langle x-y, \nndr F[\rhos](I^{(s)}(x,y)) (x-y) \rangle \d \gamma(x,y)  )\d s \ge \lambda \Wbar(\rhozero,\rhoone)^2$ \\
     \rule{0pt}{3ex} & \textcolor{gray}{$\forall\, \rhozero,\rhoone\in\A$, $\gamma\in\Gamma^*(\rhozero,\rhoone)$}
     \\[2ex]
    \end{tabular}
    \caption{(Coupled Gradient Flow Setting) Similar to the $\R^d$ setting, the difference between monotonicity and convexity is the entry-wise operators. The second order inequalities have two Hessian-like terms in the measure setting, unlike the $\R^d$ setting.} 
    \label{tab:convexity_monotonicity_comparison_second_order}
\end{table}

\section{Interaction Kernel Computations}\label{sec:interaction_kernel}
Under appropriate conditions we prove that the nonlocal cross-interaction terms appearing in \cref{subsec:monotonicity_cross_interaction} are geodesically displacement convex  \cref{lem:convexity_W}, and provide Hessian lower bounds kernels such as the attractive-repulsive kernel and power law kernel. Next, going beyond the condition in \cref{lem:interaction_kernels}, we show monotonicity under more general criteria in \cref{lem:Cross-Interaction-Kernels_appendix}.
\begin{lemma}[Convexity of $\mathcal W$]\label{lem:convexity_W}
    Consider the functional $\mathcal W$ defined as
    \begin{align*}
        \W_{ij}(\rho_i,\rho_j) = \iint W_{ij}(x_i,x_j) \d \rho_i(x_i) \d \rho_j(x_j) \,.
    \end{align*}
    Let the interaction kernel $W_{ij}\in C^2(\R^{d_i}\times \R^{d_j};\R)$ satisfy $\Hess{W_{ij}} \succeq \lambdat \Id_{d_i+d_j}$.  Then the second derivative of $\mathcal W_{ij}$ along geodesics between $\rhozero$ and $\rhoone$ satisfies
    \begin{align*}
        \ddss \W_{ij}(\rhos_i,\rhos_j) \ge  \lambdat \left(\Wass_2(\rhozero_i,\rhoone_i)^2 + \Wass_2(\rhozero_j,\rhoone_j)^2 \right) \quad \forall\, s\in[0,1]\,.
    \end{align*}
\end{lemma}
\begin{proof}
    Select $\rhozero,\rhoone\in\A$ and consider a geodesic $(\rhos,\vs)$ between them, solving \eqref{eq:geodesic_eqns}.
    The first derivative is given by
    \begin{align*}
        \dds \W_{ij}(\rhos_i,\rhos_j) &=\int \<\nabla_{x_i} \delta_{\rho_i} \W_{ij}[\rhos_i,\rhos_j](x_i), \vs_i(x_i)> \d \rhos_i(x_i) \\ 
        &\quad+  \int \<\nabla_{x_j} \delta_{\rho_j} \W_{ij}[\rhos_i,\rhos_j](x_j) , \vs_j(x_j)> \d \rhos_j(x_j) \\
        & = \iint \<\begin{bmatrix}
             \nabla_{x_i} W_{ij}(x_i,x_j) \\ \nabla_{x_j} W_{ij}(x_i,x_j)
        \end{bmatrix},
        \begin{bmatrix}
            \vs_i(x_i) \\  \vs_j(x_j)
        \end{bmatrix}> \d \rhos_{i}(x_{i})\d \rhos_j (x_j) \,.
    \end{align*}
    Differentiating again, we have
    \begin{align*}
        \ddss \W_{ij}(\rhos_i,\rhos_j) = \iint \begin{bmatrix}
            \vs_i(x_i) \\ \vs_j(x_j)
        \end{bmatrix}^\top
         H(x_i,x_j) 
        \begin{bmatrix}
           \vs_i(x_i) \\ \vs_j(x_j)
        \end{bmatrix}
        \d \rhos_i(x_i) \d \rhos_j(x_j)
    \end{align*}
    where
    \begin{align*}
    H(x_i,x_j) = \Hess{W_{ij}} \coloneqq
    \begin{bmatrix}
            \nabla_{x_i}^2 W(x_i,x_j)  & \nabla_{x_j,x_i}^2 W_{ij}(x_i,x_j) \\
            \nabla_{x_i,x_j}^2 W(x_i,x_j) & \nabla_{x_j}^2 W_{ij}(x_i,x_j)
        \end{bmatrix}\,.
    \end{align*}
\sloppy    The Hessian $H$ is well-defined because $W\in C^2$. Under the condition for the Hessian $\Hess{W}\succeq \lambdat \Id_{d_i+d_j}$, we have that 
    \begin{align*}
        \ddss \W_{ij}(\rhos_i,\rhos_j) &\ge \lambdat \iint 
        \norm{\begin{bmatrix}
            \vs_i(x_i) \\ \vs_j(x_j)
        \end{bmatrix}}^2
        \d \rhos_i(x_i) \d \rhos_j(x_j)\\
        &= \lambdat \big( \Wass_2(\rhozero_i,\rhoone_i)^2  + \Wass_2(\rhozero_j,\rhoone_j)^2 \big)\,.
    \end{align*}
\end{proof}
In the setting where $\R^{d_i}=\R^{d_j}$ and $x=x_i-x_j$, we compute the lower bounds for the Hessians of $W^{(k)}$, $W^{(m)}$, and $W^{(p)}$, given by
\begin{align*}
    W^{(k)}(x) &= \frac{\norm{x}^k}{k}\,, \quad W^{(m)}(x) = C_r \exp\bigg(-\frac{\norm{x}^2}{l_r}\bigg) - C_a \exp\bigg(-\frac{\norm{x}^2}{l_a}\bigg)\,, \\
    W^{(p)}(x) &= \frac{\norm{x}^a}{a} - \frac{\norm{x}^b}{b} \,,
\end{align*}
where $l_a>l_r>0$, $C_r>C_a>0$ and $C_r/C_a < (l_r / l_a)^{-d_i}$ and $a>b$ with $a,b$ even integers. These lower bounds allow us to apply \cref{lem:interaction_kernels} to systems with interaction kernels of this form.
The kernel $W^{(k)}=\frac{\norm{x}^k}{k}$ has a Hessian given by $\Hess{W^{(k)}} \succeq 0 \Id_{d_i}$ for integers $k\ge 1$. The Hessian of the Morse-like potential $W^{(m)}$ is more involved; the gradient is given by
\begin{align*}
    \nabla W^{(m)}(x) = 2x \left(- \frac{C_r}{l_r} \exp(-\norm{x}^2/l_r) + \frac{C_a}{l_a} \exp(-\norm{x}^2/l_a) \right)\,,
\end{align*}
and the Hessian is 
\begin{align*}
    \Hess{W^{(m)}} &= 2\Id_{d_i}\left( - \frac{C_r}{l_r} \exp(-\norm{x}^2/l_r) + \frac{C_a}{l_a} \exp(-\norm{x}^2/l_a )\right) \\
    &\quad + 4x x^\top\left(\frac{C_r}{l_r^2} \exp(-\norm{x}^2/l_r) -  \frac{C_a}{l_a^2} \exp(-\norm{x}^2/l_a )\right)\,.
\end{align*}
The smallest eigenvalues of each matrix are given by
\begin{align*}
    \lambda_{m1} &\coloneqq 2\inf_x \left(- \frac{C_r}{l_r} \exp(-\norm{x}^2/l_r) + \frac{C_a}{l_a} \exp(-\norm{x}^2/l_a ) \right)\\
    \lambda_{m2} &\coloneqq 4 \min \left\{ \inf_x \left(\norm{x}^2 \left(\frac{C_r}{l_r^2} \exp(-\norm{x}^2/l_r) -  \frac{C_a}{l_a^2} \exp(-\norm{x}^2/l_a )\right)\right),0 \right\} \,,
\end{align*}
and since $l_a,l_r >0$, the smallest eigenvalues satisfy $\lambda_{m1},\lambda_{m2}>-\infty$ and so there exists $\lambdat = \lambda_{m1} + \lambda_{m2}$ such that $\Hess{W^{(m)}} \succeq \lambdat \Id_{d_i}$.

For the third interaction kernel $W^{(p)}$, the Hessian is given by
\begin{align*}
    \Hess{W^{(p)}} =
         \left((a-2)\norm{x}^{a-4}-(b-2)\norm{x}^{b-4}\right)  xx^\top +(\norm{x}^{a-2} -\norm{x}^{b-2})\Id_{d_i} \,.
\end{align*}
Because $a$ is even and $a>b$, the Hessian has a uniform lower bound on the smallest eigenvalue for fixed $a,b$. When $b=2$, $\Hess{W^{(p)}} \succeq -\Id_{d_i}$. When $b>2$, both $(a-2)\norm{x}^{a-4}-(b-2)\norm{x}^{b-4}$ and $\norm{x}^{a-2} -\norm{x}^{b-2}$ have uniform lower-bounds because $a>b$, and therefore $\Hess{W^{(p)}}$ has a finite minimum eigenvalue $\lambdat$. See \cite{balague_nonlocal_2013} for more details on the estimates for $W^{(p)}$. 
Below, we present a general condition on interaction kernels that results in monotonicity.
\begin{lemma}[Cross-Interaction Kernels]\label{lem:Cross-Interaction-Kernels_appendix}
    Consider the interaction term $\W_{ij}:\A_{i} \times \A_{j} \to \R$, parameterized by the interaction kernel $W_{ij}\in C^2(\R^{d_{i}}\times \R^{d_{j}},\R)$, given by
    \begin{align*}
        \W_{ij}(\rho_{i},\rho_{j}) = \iint W_{ij}(x_{i}, x_{j}) \d\rho_{i}(x_{i}) \d\rho_{j}(x_{j})\,, 
    \end{align*}
Let $F_i=\sum_{j=1}^n \W_{ij}$ for all $\iin$. When $\Hess{W_{ij}}\succeq \lambda_{ij} \Id_{d_{i}+d_{j}}$ and 
\begin{align}\label{eq:interaction_kernel_inequality}
\begin{split}
    \alpha\sum_{i,j=1}^n &\< \nabla_{x_i}[ W_{ij}(x_i,x_j)+W_{ji}(x_j,x_i) ]- \nabla_{x_i} [ W_{ij}(y_i,y_j)+W_{ji}(y_j,y_i)], 
            x_i-y_i > \\
            &\le  \sum_{i,j=1}^n \< \nabla_{x_i} W_{ij}(x_i,x_j)- \nabla_{x_i} W_{ij}(y_i,y_j), 
            x_i-y_i > \quad \forall\, x_i,y_i\in\R^{d_i}\,,\, x_j,y_j\in\R^{d_j}\,,
\end{split}
\end{align}
then $F$ is $\alpha\lambda$-monotone in $\A$, where $\lambda_i=\sum_{j=1}^n \lambda_{ij}+\lambda_{ji}$ and $\lambda=\min_i\lambda_i$. 
\end{lemma}
\begin{proof}
     Using the exact form of the Taylor expansion for $\W_{ij}$, and the bound for the second derivative from \cref{lem:convexity_W}, $\ddss \W_{ij}(\rhos_i,\rhos_j) \ge \lambda_{ij}(\Wass_2(\rhozero_i,\rhoone_i)^2 + \Wass_2(\rhozero_j,\rhoone_j)^2)$, we have along geodesics between $(\rhozero_i,\rhozero_j)$ and $(\rhoone_i,\rhoone_j)$,
    \begin{align*}
        \W_{ij}(\rhoone_i,\rhoone_j) &= \W_{ij}(\rhozero_i,\rhozero_j) + \left. \dds \W_{ij}(\rhos_i,\rhos_j) \right|_{s=0} \\
        &\qquad+  \int_0^1 (1-t) \left. \left( \ddss \W_{ij}(\rhos_i,\rhos_j) \right)\right|_{s=t} \d t \\
        & \ge \W_{ij}(\rhozero_i,\rhozero_j) + \left. \dds \W_{ij}(\rhos_i,\rhos_j) \right|_{s=0} + \frac{\lambda_{ij}}{2}\sum_{k=i,j}\Wass_2(\rhozero_k,\rhoone_k)^2\,, \\
        \W_{ij}(\rhozero_i,\rhozero_j) & = \W_{ij}(\rhoone_i,\rhoone_j) - \left. \dds \W_{ij}(\rhos_i,\rhos_j) \right|_{s=1} -  \int_1^0 t \left. \left( \ddss \W_{ij}(\rhos_i,\rhos_j) \right)\right|_{s=t} \d t \\
        &\ge \W_{ij}(\rhoone_i,\rhoone_j) - \left. \dds \W_{ij}(\rhos_i,\rhos_j) \right|_{s=1} + \frac{\lambda_{ij}}{2}\sum_{k=i,j}\Wass_2(\rhozero_k,\rhoone_k)^2 \,.
    \end{align*}
    The term $\dds \W_{ij}(\rhos_i,\rhos_j)$, using geodesics \eqref{eq:geodesic_eqns}, is given by
    \begin{align*}
        \dds \W_{ij}(\rhos_i,\rhos_j) &= \int \bigg(\<\nabla_{x_i} W_{ij}(x_i,x_j),\ws_i(x_i)> \\
        &\qquad+ \<\nabla_{x_j} W_{ij}(x_i,x_j),\ws_j(x_j)> \bigg)\d \rhos_i(x_i) \d \rhos_j(x_j) \,.
    \end{align*}
    Recall $\rhos_i=(I_i^{(s)})_\sharp \gamma_i$ where $I_i^{(s)}(x_i,y_i)=(1-s)x_i+sy_i$, and that $\ws_i(I^{(s)}_i(x_i,y_i))=y_i-x_i$. Evaluating $\dds \W_{ij}(\rhos_i,\rhos_j)$ at $s=0$ and $s=1$,
    \begin{align*}
        \dds \W_{ij}(\rhos_i,\rhos_j)\bigg|_{s=0} &= \iint \<\begin{bmatrix}
             \nabla_{x_i} W_{ij}(x_i,x_j) \\ \nabla_{x_j} W_{ij}(x_i,x_j)
        \end{bmatrix},
        \begin{bmatrix}
           y_i-x_i \\  y_j-x_j
        \end{bmatrix}> \d \gamma_{i}(x_{i},y_i)\d \gamma_j (x_j,y_j)  \\
         \dds \W_{ij}(\rhos_i,\rhos_j)\bigg|_{s=1} &= \iint \<\begin{bmatrix}
             \nabla_{x_i} W_{ij}(y_i,y_j) \\ \nabla_{x_j} W_{ij}(y_i,y_j)
        \end{bmatrix},
        \begin{bmatrix}
            y_i-x_i \\  y_j-x_j
        \end{bmatrix}> \d \gamma_{i}(x_{i},y_i)\d \gamma_j (x_j,y_j) \,.
    \end{align*}
    Adding the two Taylor expansion inequalities results in
    \begin{align}\label{eq:monotonicity_inequality_interaction_kernels}
        \iint \<\begin{bmatrix}
            \nabla_{x_i} W_{ij}(x_i,x_j) - \nabla_{x_i} W_{ij}(y_i,y_j) \\ \nabla_{x_j} W_{ij}(x_i,x_j)-\nabla_{x_j} W_{ij}(y_i,y_j)
        \end{bmatrix}, \begin{bmatrix}
            x_i-y_i \\  x_j-y_j
        \end{bmatrix}>\d \gamma_i(x_i,y_i) \d\gamma_j(x_j,y_j) \qquad
        \\
        \ge \lambda_{ij} \sum_{k=i,j}\Wass_2(\rhozero_k,\rhoone_k)^2\,.
    \end{align}
    Summing \eqref{eq:monotonicity_inequality_interaction_kernels} over all $i,j\in[1,\dots,n]$,
    \begin{align*}
      \sum_{i,j=1}^n  \iint \<
            \nabla_{x_i}[ W_{ij}(x_i,x_j)+W_{ji}(x_j,x_i) ]- \nabla_{x_i} [ W_{ij}(y_i,y_j)+W_{ji}(y_j,y_i)], 
            x_i-y_i > \d \gamma(x,y) \qquad \\
            \ge \lambda  \Wbar(\rhozero,\rhoone)^2\,,
    \end{align*}
    where $\lambda_i=\sum_{j=1}^n \lambda_{ij}+\lambda_{ji}$
    and $\lambda=\min_i\lambda_i$. Then multiplying both sides by $\alpha$ and applying \eqref{eq:interaction_kernel_inequality},
    \begin{align*}
       \sum_{i,j=1}^n \iint \< \nabla_{x_i} W_{ij}(x_i,x_j)- \nabla_{x_i} W_{ij}(y_i,y_j), 
            x_i-y_i >\d \gamma(x,y) \ge \alpha \lambda \Wbar(\rhozero,\rhoone)^2\,.
    \end{align*}
    Therefore $F$ is $\alpha\lambda$-monotone in $\A$.
\end{proof}

\section{Optimal and Admissible Plans in the Monotonicity Inequality}\label{sec:optimal_admissible_plans}
In this section, we prove the equivalence between the monotonicity condition using optimal plans and using admissible plans, in the setting where $v$ is monotone over $\Pbar$.
The proofs use techniques from \cite{cavagnari_lagrangian_2023}, where the authors first prove the equivalence for plans which have finite support, and then extend this result to any probability measure with finite second moment.

Recall that $\Gamma(\rhozero_i,\rhoone_i)$ is the set of all plans with marginals $\rhozero_i,\rhoone_i$, and $\Gamma^*(\rhozero_i,\rhoone_i)\subset\Gamma(\rhozero_i,\rhoone_i)$ is the set of optimal plans with marginals $\rhozero_i,\rhoone_i$. Let $\Pf(\R^{d_i})\subset \P_2(\R^{d_i})$ be the set of measures on $\R^{d_i}$ with finite support, defined as
\begin{align*}
    \P^{(f)} = \big\{\rho_i\in\P_2(\R^{d_i}) \ \big| \ \supp \rho_i = \{ x_i^{(1)}, \dots, x_i^{(M_i)}  \}\,, \ M_i<\infty \big\} \,,
\end{align*}
and let $\Pfbar \coloneqq \Pf(\R^{d_1}) \times \dots \times \Pf(\R^{d_n})$. We define $I^{(t)}_i(x_i,y_i) = (1-t)x_i+ty_i $.
\begin{theorem}[Local optimality of discrete interpolations (Theorem 6.2, \cite{cavagnari_lagrangian_2023})]\label{thm:local_optimality}
    Let $\rhozero_i,\rhoone_i\in\Pf(\R^{d_i})$,  select any $\gamma_i\in\Gamma(\rhozero_i,\rhoone_i)$, and define $\rhotau_i:=(I_i^{(\tau)})_{\sharp} \gamma_i$. Then there exists a finite number of points $\tau^{(i)}_0=0<\tau_1^{(i)}< \dots < \tau_{K_i}^{(i)}=1$ such that for every $k=1,\dots,K_i$, $\rhotau_i$ for $\tau\in[\tau_{k-1}^{(i)},\tau_k^{(i)}]$ is a minimal constant-speed geodesic and 
    \begin{align*}
        \Wass_2^2(\rho_i^{(t_1)},\rho_i^{(t_2)}) = |t_1-t_2|^2 \int \norm{x-y}^2 \d \gamma_i(x,y)\,, \quad \forall\, t_1,t_2\in[\tau_{k-1}^{(i)},\tau_k^{(i)}] \,.
    \end{align*}
\end{theorem}
\begin{proposition}[Admissible Couplings for Finitely-Supported Marginals]\label{prop:admissible_couplings_finite_supp}
    If $\Pfbar \subseteq \A$ and  $v$ is $\lambda$-monotone in $\A$, i.e. for any $\rhozero,\rhoone \in  \Pfbar$ with $\gamma^*_i \in \Gamma^*(\rhozero_i,\rhoone_i)$, $\gamma^*=\prod_{i=1}^n \gamma_i^*$, it holds 
    \begin{align*}
         -\int \<x-y,v[\rhozero](x)-v[\rhoone](y)> \d \gamma^*(x,y) \ge \lambda \int \norm{x-y}^2 \d \gamma^*(x,y) \,,
    \end{align*}
    then for any $\gamma_i\in \Gamma(\rhozero_i,\rhoone_i)$, $\gamma=\prod_{i=1}^n \gamma_i$ such that $\rhotau_i:=(I_i^{(\tau)})_{\sharp} \gamma_i$ satisfies $\rhotau\in\A$ for all $\tau\in[0,1]$, it holds
    \begin{align*}
         -\int \<x-y,v[\rhozero](x)-v[\rhoone](y)> \d \gamma(x,y) \ge \lambda \int \norm{x-y}^2 \d \gamma(x,y)  \,.
    \end{align*}
\end{proposition}
\begin{proof}
     We select $\gamma_i\in\Gamma(\rhozero_i,\rhoone_i)$ for all $i$. First, we will construct a set of time increments such that this admissible plan is locally optimal for all species in each increment. By \cref{thm:local_optimality}, for each species $i$ there exists a set of points $(\tau_j^{(i)})_{j=1}^{K_i}$ such that $\rho_i^{(\tau)}$ is a constant-speed geodesic for $\tau\in[\tau_{k-1}^{(i)},\tau_{k}^{(i)}]$. Taking the union of all such time steps for each species and renaming them in strictly increasing order,
     \begin{align*}
         \bigcup_{i=1}^n (\tau_j^{(i)})_{j=1}^{K_i} = \{\overline\tau_0<\overline\tau_1<\dots<\overline\tau_{\overline K} \}\,,
     \end{align*}
      the local optimality of $\gamma_i^{(\tau)}$ holds for all species in each interval $[\overline \tau_{k-1},\overline \tau_{k}]$. The set of time steps $(\overline\tau_k)$ has size $\overline K \le \sum_{i=1}^n K_i - 2n+2$, where the $-2n+2$ comes from the same initial and final time steps for all species. Also by \cref{thm:local_optimality}, the optimal coupling from $\rho^{(\overline\tau_{k-1})}_i$ to $\rho^{(\overline\tau_k)}_i$ is given by $\gamma^{(\overline\tau_{k-1},\overline\tau_k)}_i$, where
    \begin{align*}
        \gamma^{(s,t)}_i := [I_i^{(s)},I_i^{(t)}]_\sharp \gamma_i \,.
    \end{align*}
    Hence, by $\lambda$-monotonicity,
    \begin{align*}
        -\int \<x-y,v[\rho^{(\overline \tau_{k-1})}](x) - v[\rho^{(\overline \tau_k)}](y) >\d \gamma^{(\overline\tau_{k-1},\overline\tau_k)}(x,y) \ge \lambda \int \norm{x-y}^2 \d \gamma^{(\overline\tau_{k-1},\overline\tau_k)}(x,y)\,.
    \end{align*}
    Next, we apply the pushforward definition of $\gamma^{(\overline\tau_{k-1},\overline\tau_k)}$ to the terms in the monotonicity inequality,
    \begin{align*}
        \int \norm{x-y}^2 \d \gamma^{(\overline\tau_{k-1},\overline\tau_k)}(x,y)  = (\overline \tau_{k}-\overline \tau_{k-1})^2 \int \norm{x-y}^2 \d\gamma(x,y)  \,, 
    \end{align*}
    and
    \begin{align*}
         &\int \<x-y,v[\rho^{(\overline \tau_{k-1})}](x) - v[\rho^{(\overline \tau_k)}](y) >\d \gamma^{(\overline\tau_{k-1},\overline\tau_k)}(x,y) =
        \\
        &\qquad
        (\overline \tau_k-\overline \tau_{k-1}) \int \<x-y, v[\rho^{(\overline \tau_{k-1})}](I^{(\overline \tau_{k-1})}(x,y)) - v[\rho^{(\overline \tau_k)}](I^{(\overline \tau_{k})}(x,y)) >\d \gamma(x,y) \,.
    \end{align*}
    The monotonicity condition becomes, for every $k\in[1,\dots,\overline K]$,
    \begin{align*}
         - \int \<x-y, v[\rho^{(\overline \tau_{k-1})}](I^{(\overline \tau_{k-1})}(x,y)) - v[\rho^{(\overline \tau_k)}](I^{(\overline \tau_{k})}(x,y)) >\d \gamma(x,y) 
         \ge \left(\overline \tau_{k}-\overline \tau_{k-1}\right) \lambda \int \norm{x-y}^2 \d\gamma(x,y) \,.
    \end{align*}
    Thanks to the integration against the same measure $\gamma$ for every $k$, summing from $k=1$ to $\overline K$ results in
    \begin{align*}
        \lambda \int \norm{x-y}^2 \d \gamma(x,y) &\le -\int \<x-y,v[\rhozero](I^{(0)}(x,y))-v[\rhoone](I^{(1)}(x,y))>\d\gamma(x,y) \\
        &=-\int \<x-y,v[\rhozero](x)-v[\rhoone](y)>\d\gamma(x,y)\,.
    \end{align*}
\end{proof}
In order to extend \cref{prop:admissible_couplings_finite_supp} to all measures in $\Pbar$, we adopt the approximation method from \cite{cavagnari_dissipative_2023}, where $\rhozero,\rhoone\in\Pbar$ are approximated by $ \rho^{(n,0)}, \rho^{(n,1)}\in\Pfbar$ such that $\Wbar(\rho^{(n,0)},\rhozero) \to 0$ and $\Wbar(\rho^{(n,1)},\rhoone) \to 0$ as $n\to \infty$. In order to pass to the $n\to\infty$ limit in the monotonicity condition, we require that the velocity fields $v[\rho](x)$ satisfy an appropriate notion of continuity with respect to $\rho$ and $x$.
\begin{definition}[Strong-Weak Convergence]
    For $x^{(n)},v^{(n)}\in\R^d$, a sequence $(x^{(n)},v^{(n)})$ converges to $(x,v)$ in the strong-weak topology, denoted $(x^{(n)},v^{(n)}) \overset{sw}{\to} (x,v)$, if
    \begin{align*}
        x^{(n)}\to x\,, \quad f(v^{(n)}) \to f(v) \quad \forall \, f\in \mathscr B(\R^d)\,,
    \end{align*}
    where $\mathscr B(\R^d)$ is the set of linear bounded functionals on $\R^d$.
\end{definition}
The set $C^{sw}_2$ is given by
\begin{align*}
        C^{sw}_2:= \{\zeta:&\R^d \times \R^d \to\R \ 
        | \ \lim_{n\to\infty} \zeta(x^{(n)},v^{(n)}) = \zeta(x,v) \quad \forall \, (x^{(n)},v^{(n)}) \overset{sw}{\to} (x,v)\,;\  \forall\, \varepsilon>0\\
        & \exists \, A_\varepsilon\ge 0 \text{ such that } |\zeta(x,v)| \le A_\varepsilon(1+\norm{x}^2) + \varepsilon \norm{v}^2\,, \forall\, (x,v)\in \R^d \times \R^d \}\,.
\end{align*}
\begin{definition}[Demicontinuity of $v$]\label{def:demicontinuity_v}
    The velocity field $v:
    \A\times \R^d\to\R^d$ is \emph{demicontinuous}  at $\rho\in\A$ if for any $\rho^{(n)} \to \rho$ in $(
    \A,\Wbar)$, it holds that
    \begin{align*}
        \lim_{n \rightarrow \infty} \int \zeta(x,v[\rho^{(n)}](x))\d \rho^{(n)}(x) = \int  \zeta(x,v[\rho](x))\d \rho(x)\,, \quad \forall \ \zeta \in C^{sw}_2\,.
    \end{align*}
\end{definition}
\begin{assumption}[Demicontinuity of $v$]\label{assump:demicontinuity_v}
   The velocity field $v$ is \emph{demicontinuous} in $\A$.
\end{assumption} 
\begin{corollary}[Admissible Couplings]\label{cor:admissible_couplings}
    Let \cref{assump:demicontinuity_v} hold. If $\Pfbar \subseteq\A $, and for any $\rhozero,\rhoone \in \A$ with $\gamma^* \in \Gamma^*(\rhozero,\rhoone)$ it holds 
    \begin{align*}
       -\int \<x-y,v[\rhozero](x)-v[\rhoone](y)> \d \gamma^*(x,y)\ge  \lambda \int \norm{x-y}^2 \d \gamma^*(x,y) \,,
    \end{align*}
    then for any $\gamma\in \Gamma(\rhozero,\rhoone)$ such that $\rhotau\in\A$ for all $\tau\in[0,1]$,
    \begin{align*}
       - \int \<x-y,v[\rhozero](x)-v[\rhoone](y)> \d \gamma(x,y)\ge  \lambda \int \norm{x-y}^2 \d \gamma(x,y) \,.
    \end{align*}
\end{corollary}
\begin{proof}
The proof follows exactly as the approximation procedure in the proof of \cite[Theorem 7.6]{cavagnari_lagrangian_2023}, with the convergent sequences defined for each species. 
\end{proof}
In particular, if $\A=\Pbar$, then $\Pfbar \subseteq\A $ and so \cref{thm:admissible_couplings} holds.
\begin{remark}\label{rmk:monotonicity-equiv-counterex}
    \cref{cor:admissible_couplings} does not necessarily hold for sets $\A$ which do not contain $\Pfbar$. For example, consider the single-species setting where $v[\rho]=-\nabla \log \rho$. The domain of $v:L^1(\R^d) \to \R^d$ excludes measures with finite support. Writing $v[\rho]=-\nabla \delta_\rho \Hc(\rho)$, where $\Hc=\int \rho \log \rho$, it holds that 
    \begin{align*}
        \int \<x-y,\nabla\log\rho(x)-\nabla\log \rhot(y) > \d \gamma^*(x,y) \ge 0
    \end{align*}
    for $\gamma^*\in\Gamma^*(\rho,\rhot)$, because $\Hc$ is $0$-displacement convex over absolutely continuous measures. We check if the inequality still holds for the admissible plan $\gamma(x,y)=\rho(x)\rhot(y)$:
    \begin{align*}
        \int \<x-y,\nabla \log \rho(x)-\nabla \log \rhot(y) > \d \rho(x) \d \rhot(y) &= -\int \<x,\nabla \rho(x)> \d x - \int \<y,\nabla \rhot(y)>\d y 
        \le -2d<0\,,
    \end{align*}
    so the monotonicity inequality does not hold for general $\gamma\in\Gamma(\rho,\rhot)$.
\end{remark}

\printbibliography
\end{document}

%% file: arxiv_preamble.tex
\newif\ifsiamart
\makeatletter%
\@ifclassloaded{siamart171218}%
  {\siamarttrue}%
  {\siamartfalse}%
\makeatother%

\ifsiamart\else
\usepackage{authblk}
\fi
\usepackage{silence}
\usepackage[T1]{fontenc}
\usepackage[utf8]{inputenc}
\usepackage{csquotes}
\usepackage[UKenglish]{babel}
\usepackage{paralist}
\usepackage[shortlabels]{enumitem}
\usepackage[margin=.6in]{geometry}
\usepackage{xcolor}
\usepackage{graphicx}
\usepackage{bbm}
\usepackage{amsmath}
\usepackage{amssymb}
\usepackage{amsfonts}
\usepackage{stmaryrd}
\usepackage{mathrsfs}
\usepackage{mathtools}
\usepackage{epstopdf}
\usepackage{comment}

\usepackage{xspace}
\usepackage{float}
\usepackage{stackengine}

\ifsiamart\else
\usepackage{amsthm}
\usepackage{hyperref}
\hypersetup{%
    linktocpage=true,
    colorlinks=true,
    citecolor=red,
    linkcolor=blue,
    urlcolor=blue,
}
\usepackage[nameinlink,capitalise]{cleveref}
\newcommand{\email}[1]{\href{mailto:#1}{#1}}
\newenvironment{keywords}{\noindent \textbf{Keywords.}}{}
\newenvironment{AMS}{\noindent \textbf{AMS subject classification.}}{}
\fi

\usepackage{setspace}
\usepackage{amssymb,bbm,mathrsfs}
\usepackage{mathtools}
\usepackage{multicol}
\usepackage[normalem]{ulem}
\usepackage{caption}
\usepackage{subcaption}
\usepackage{array}
\usepackage{centernot}

\DeclareMathOperator{\supp}{supp}
\def\A{\mathcal A}
\def\B{\mathcal B}
\def\W{\mathcal W}

\def\Hc{\mathcal H}
\def\L{\mathcal L}

\def\P{\mathcal P}

\def\R{\mathbb R}

\newcommand{\Wass}{\mathscr{W}}

\newcommand{\rhot}{\tilde{\rho}}

\newcommand{\lambdat}{\tilde{\lambda}}

\newcommand{\mb}{\mathbb}
\DeclareMathOperator*{\E}{\mb{E}}

\DeclareMathOperator{\id}{id}
\DeclareMathOperator{\Id}{I}

\newcommand{\eps}{\varepsilon}

\newcommand{\Wbar}{\overline{\mathscr{W}_2}}

\newcommand{\Hess}[1]{\text{Hess}\left( #1 \right)}
\DeclareMathOperator*{\argmin}{arg\,min}

\renewcommand{\d}{\mathsf{d}}
\newcommand{\ddt}{\frac{\d}{\d t}}

\newcommand{\dds}{\frac{\d}{\d s}}

\newcommand{\ddss}{\frac{\d^2}{\d s^2}}

\newcommand{\norm}[1]{\left\lVert#1\right\rVert}
\renewcommand{\div}[1]{\mathrm{div}\left( #1 \right)}

\def\<#1,#2>{\left\langle #1,\,#2\right\rangle}

\ifsiamart
\newtheorem{remark}[theorem]{Remark}
\newtheorem{assumption}{Assumption}
\else
\theoremstyle{plain}

\newtheorem{theorem}{Theorem}
\newtheorem{lemma}[theorem]{Lemma}
\newtheorem{corollary}[theorem]{Corollary}
\newtheorem{proposition}[theorem]{Proposition}

\newtheorem{remark}[theorem]{Remark}
\newtheorem{definition}[theorem]{Definition}
\newtheorem{example}[theorem]{Example}
\newtheorem{assumption}{Assumption}

\makeatletter
\newcommand{\mathleft}{\@fleqntrue\@mathmargin0pt}
\newcommand{\mathcenter}{\@fleqnfalse}
\makeatother
\fi

\crefname{lemma}{Lemma}{Lemmas}
\crefname{remark}{Remark}{Remarks}
\crefname{assumption}{Assumption}{Assumptions}
\crefname{proposition}{Proposition}{Propositions}
\crefname{section}{Section}{Sections}
\crefname{subsection}{Subsection}{Subsections}
\crefname{equation}{}{}
\Crefname{equation}{Equation}{Equations}
\newlist{lemmaenum}{enumerate}{3}
\setlist[lemmaenum]{label=(\alph*),ref=\,(\alph*)}
\crefname{lemmaenum}{Lemma}{Lemmas}
\newlist{assumpenum}{enumerate}{5}
\setlist[assumpenum]{label=(\alph*), font={\bfseries}}
\newlist{auxenum}{enumerate}{2}
\setlist[auxenum]{label=(\alph*),ref=(\alph*)}
\crefname{auxenumi}{Item}{Items}
\crefname{enumi}{}{}
\crefname{equation}{}{}
\crefname{assumpenumi}{}{}
\crefname{assumpenumii}{}{}
\Crefname{assumpenumi}{Assumption}{Assumptions}
\Crefname{assumpenumii}{Assumption}{Assumptions}
\Crefname{assumpenumii}{Assumption}{Assumptions}
\crefrangeformat{assumpenumi}{#3#1#4--#5#2#6}
\Crefname{lemmaenumi}{Part}{Parts}
\Crefname{figure}{Figure}{Figures}
\numberwithin{equation}{section}
\numberwithin{theorem}{section}

\usepackage{tikz}
\usetikzlibrary{shapes.misc}
\usetikzlibrary{positioning}

\usepackage[url=true,backend=biber,style=trad-abbrv,url=false,doi=false,isbn=false]{biblatex}
\DeclareFieldFormat{volume}{volume \textbf{#1}}
\DeclareFieldFormat[article]{volume}{\textbf{#1}}

\ifsiamart\else
\let\oldparagraph=\paragraph
\renewcommand\paragraph[1]{\oldparagraph{#1.}}
\fi

\newcommand{\Pbar}{ \overline{\mathcal{ P}_2}}
\newcommand{\divs}[2]{\mathrm{div}_{#1}\left( #2 \right)}
\newcommand{\ndr}{\overline {\nabla \delta_\rho } }

\newcommand{\nnddr}{\overline {\nabla^2 \delta^2_\rho}}
\newcommand{\nndr}{\overline {\nabla^2 \delta_\rho}}

\newcommand{\Pf}{\P^{(f)}}
\newcommand{\Pfbar}{\overline{\P}^{(f)}}
\newcommand{\rhoni}{\rho_{-i}} 

\newcommand{\rhotau}{\rho^{(\tau)}}

\newcommand{\rhoone}{\rho^{(1)}}

\newcommand{\rhozero}{\rho^{(0)}}

\newcommand{\rhos}{\rho^{(s)}}

\newcommand{\nv}{\overline{\nabla \varphi}}

\newcommand{\iin}{i\in[1,\dots,n]}
\newtheorem{property}{Property}

\newcommand{\vs}{v^{(s)}}
\newcommand{\ws}{w^{(s)}}